\documentclass[11pt,twoside,template]{article}

\usepackage{graphicx}
\usepackage{amssymb,amsmath,amsthm}
\usepackage{geometry}
\usepackage{placeins}
\usepackage{enumerate}
\usepackage{nicefrac}
\usepackage{hyperref}

\usepackage{lmodern}
\usepackage{listings}
\usepackage{xcolor} %red, green, blue, yellow, cyan, magenta, black, white
\definecolor{mygreen}{RGB}{28,172,0} % color values Red, Green, Blue
\definecolor{mylilas}{RGB}{170,55,241}

\usepackage{soul}

\newcommand{\comment}[1]{} % comments off: must choose one or the other
\newcommand{\cec}{\color{black}}
\newcommand{\peter}{\color{black}}
\newcommand\R{{\mathbb R}} % Real numbers
\newcommand\N{{\mathbb N}} % Natural numbers
\newcommand\pr{{\mathbb P}} % Probability
\newcommand\E{{\mathbb E}} % Expectation
\newcommand\cE{{\mathcal{E}}} % Expectation
\newcommand\emm{{\mathfrak m}} % Expectation

\newcommand\id{{\mathbf{1}}} % Identity function 1

\newcommand\I{E}%\mathcal{E}} % Set I
\newcommand\dx{\mathrm{d}x} %dx 
\newcommand\dy{\mathrm{d}y} %dy
\newcommand\df{\mathrm{d}f} %df
\newcommand\dw{\mathrm{d}w} %dw
\newcommand\dz{\mathrm{d}z} %dz
\newcommand\di{\mathrm{d}} %d

\newcommand\e{\mathrm{e}} %e

\newcommand\tetta{\frac{\varrho}{\varrho+ 1}} % delta/(delta + 1)
\newcommand\sss{\scriptscriptstyle}
\newtheorem{theo}{Theorem}
\newtheorem{prop}[theo]{Proposition}
\newtheorem{lem}[theo]{Lemma}
\newtheorem{cor}[theo]{Corollary}

\newtheorem*{note}{Note}

\newtheorem{myenv}{Assumption}

\newtheorem{myenb}{Assumption}

\newcommand{\heap}[2]  {\genfrac{}{}{0pt}{}{#1}{#2}}
\newcommand{\ssup}[1] {{\scriptscriptstyle{({#1}})}}
\newcommand{\sfrac}[2] {\mbox{$\frac{#1}{#2}$}}
\newcommand{\bor}[1] {{#1}_\kappa}

\title{Competing growth processes with random growth rates\\ and random birth times}
\author{C\'ecile Mailler, Peter M\"orters and Anna Senkevich}
\date{} % To get rid of the date, uncomment

\begin{document}
\maketitle
\begin{abstract}
\noindent  {\peter Comparing individual contributions in a strongly interacting system of stochastic growth processes can be a very difficult problem. 
This is particularly the case when new growth processes are initiated depending on the state of previous ones and the growth rates of the individual processes
are themselves random.
%in a dynamic random network, the largest set in a partition-valued stochastic process, or the largest family in an evolving population at a given time, can be a very difficult problem. This is particularly the case when the underlying stochastic process has complex dependencies and the individual strength of an object has an impact that only plays out over time. 
We propose a novel technique to deal with such problems and show how it can be applied to a broad range of examples where it produces new insight and surprising results. The method relies on two steps: In the first step, which is highly problem dependent, the  growth processes are jointly embedded into continuous time so that their evolutions after
initiation become approximately independent 
% in the new time scale  
while we retain some control over the initiation times. % themselves. 
Once such an embedding is achieved, the second step is to apply a Poisson limit theorem that enables a comparison of the state of the processes initiated in a critical window and therefore allows an asymptotic description of the extremal process. 
In this paper we prove a versatile limit theorem of this type and show how this tool can be applied to obtain novel asymptotic results  for a variety of interesting stochastic processes. These  include  (a) the maximal degree in different types of preferential attachment networks with fitnesses like the well-known Bianconi-Barab\'asi tree and a network model of Dereich, (b) the most successful mutant in a branching processes evolving by selection and mutation, and (c) the ratio between the largest and second largest cycles in a random permutation with random cycle weights, which can also be interpreted as a disordered version of Pitman's Chinese restaurant process.}

%
%In this paper we prove such a versatile limit theorem, based on extreme value theory, 
%and show how the technique can be used to study extremal behaviour in different types of preferential 
%attachment networks with fitness, branching processes  with selection and mutation,  and random permutations with random cycle weights. 

%\noindent We study a population consisting of families initialised at random birth times with growth 
%rates given by i.i.d.\ random variables. 
%Allowing arbitrary dependence between birth times and growth rates 
%enables  us to study a wide range of examples including different types of preferential 
%attachment networks with fitness, branching processes  with selection and mutation,  
%and random permutations with random cycle weights. Our main result is a versatile Poisson limit theorem 
%which implies convergence of the scaled
%size of the largest family at large times to a Fr\'echet distribution and convergence of the standardised
%birth time of this family to a Gaussian distribution, in the case where the growth rates are sampled from 
%the maximum domain of attraction of the Gumbel distribution. 
\end{abstract}

{\footnotesize
\setcounter{tocdepth}{2}
\tableofcontents}
\clearpage
%%%%%%%%%%%%%%%%%%%%%%%%%%%%%%%%%%%%%%%%%%%%%%%%%
\section{Introduction} \label{sec:intro}
\subsection{Motivation} \label{sec:motivation}
\noindent 

Suppose a population of immortal individuals evolves as follows:  We start with one individual with a fitness sampled from a fixed bounded distribution $\mu$. When the population consists of $n$ individuals, the next individual selects its parent from the $n$ existing individuals with a probability proportional to their individual fitnesses. With high probability the new individual inherits the fitness from its parent and joins the parent's family, but with small probability $\beta>0$ 
the individual is a mutant and founder of a new family, getting a fitness sampled independently of everything else from 
the distribution~$\mu$. Even for such a simple model of a population evolving by selection and mutation {\peter the structure at large finite times, i.e.\ when the system is not in equilibrium, can be hard to analyse. The difficulty is that it takes time until an individual born with high fitness can use its advantage to build a large family.  Quantities like the relative size of the largest family when the total population has a given large size depend on these delays and therefore involve a comparison of many different random influences which are typically very hard to  control.}
\medskip

In this paper we investigate a broad class of problems {\peter loosely similar to  the above} providing a novel technique to their solution. 
%The \emph{first step} in our technique is to choose 
For the method to work one needs  an embedding of  the problem into continuous time that makes the growth processes of the individuals approximately independent. 
{\peter{Such embeddings have been used as a tool for urn processes since the seminal work of Athreya and Karlin~\cite{AK68} and can be constructed for a wide range of models.}}
In our example the embedding is achieved by equipping every individual with fitness $f$ with an independent Poisson process of intensity $f$ initiated at the individual's birthtime. 
The jump times of the Poisson process correspond to the times when the individual is chosen as a parent. Then, given a population of $n$ individuals %with fitnesses $f_1, \ldots, f_n$ 
the probability that each individual is the next parent is proportional to its fitness. Each family is equipped with an independent fitness sampled from a 
distribution~$\mu$ and, starting from its birthtime, grows as an independent Yule process with parameter $(1-\beta)f$, where $f$ is the fitness of the family and $\beta$ the mutation probability. The downside of looking at the problem in this time-scale is that the families' birthtimes depend in a complex way on the multitude of independent growth processes and all we can hope for is an asymptotic expansion of the birthtime  $\tau_n$ of the $n$th family.
\medskip

The main step in our technique is to use extreme value theory and the approximate independence of the growth processes in our embedding to provide  asymptotic properties of the largest family. As in our example we assume that the growth rates  are sampled from an i.i.d.\  sequence $F_1, F_2, \ldots$ of~bounded 
random variables, while the birth times $\tau_1, \tau_2, \dots$ may be random and depend in an {\peter arbitrarily complex} fashion on the growth processes. In the most interesting cases the birth times are themselves arising from an exponentially growing process so that the largest family at time $t$ arises in competition of the few families born early, which have a longer time to grow, and the many families born late, among which the occurence of a higher birth rate is more probable. 
%The situation we investigate arises for example in various dynamic network models, where the families are nodes and their size is the degree, or in variants of the chinese restaurant process, where the families are tables and their size is the number of occupants. 
We will give interesting examples below, but first we give a flavour of the problem by a calculation based on the simplest nontrivial scenario.
\smallskip

For this purpose let the birth time of the $n$th family be $\tau_n=\frac1\lambda \log n$  and its size at time~$t$~be
$$Z_n(t)= \left\{ \begin{array}{ll}
\e^{(t-\tau_n) F_n} &  \mbox{ if } \tau_n<t,\\
0 & \mbox{ otherwise.}\\
\end{array}
\right.$$
Suppose $\mu$ is the law of $F_n$ on the interval $(0,1]$ and 
let $1\ll T(t) \ll t$. Then  % we determine later,  probability that $\max_n Z_n(t) \leq e^{(t-T(t))+x}$ equals
\begin{align*}
\pr\Big( \e^{-(t-T(t))}\max_n Z_n(t) \leq \e^{x}\Big) & =
\pr\Big( (t-\tau_n) F_n \leq (t-T(t))+x \, \, \forall n \colon \tau_n\leq t \Big)\\
& = \prod_{\tau_n \leq T(t)-x} \pr\Big(F_n \leq \frac{t-T(t)+x}{t-\tau_n} \Big)\\
\end{align*}
\begin{align*}
\phantom{\pr\Big( \e^{-(t-T(t))}\max_n Z_n(t) \leq \e^{x}\Big)} & \phantom{=
\pr\Big( (t-\tau_n) F_n \leq (t-T(t))+x \, \, \forall n \colon \tau_n\leq t \Big)}\\
& = \exp\bigg( \sum_{n\leq \e^{\lambda (T(t)-x)}} \log \Big(1-\mu\Big(\big( \sfrac{t-T(t)+x}{t-\tau_n},1\big]\Big)\Big)\bigg) \\
& = \exp\bigg( -  (1+o(1)) \,\sum_{n\leq \e^{\lambda (T(t)-x)}} 
\mu\Big(\big( \sfrac{t-T(t)+x}{t-\tau_n},1\big]\Big)\bigg).
\end{align*}
The task is now to choose $T(t)$ such that, as $t\uparrow\infty$, 
$$\sum_{n\leq \e^{\lambda (T(t)-x)}} 
\mu\Big(\big( \sfrac{t-T(t)+x}{t-\tau_n},1\big]\Big) \longrightarrow \phi(x),$$
for some nondegenerate function $\phi$. The solution depends on the tail of $\mu$ at one.
Supposing for example that 
%$x \mapsto \mu((1-x,1])$ is regularly varying at $0$ with index 
%$\alpha>0$, which is the Weibull case of extreme value theory, we get
$\mu((1-x,1])\sim x^\alpha$ as $x\downarrow 0$, for some index $\alpha>0$, we get
\begin{align*}
\sum_{n\leq \e^{\lambda (T(t)-x)}} 
\mu\Big(\big( \sfrac{t-T(t)+x}{t-\tau_n},1\big]\Big) 
& \sim  \frac1{t^\alpha} \sum_{n\leq \e^{\lambda (T(t)-x)}}  \big( {T(t)-\tau_n-x}\big)^\alpha.
\end{align*}
Letting $T(t)=\frac\alpha\lambda \log t$ this is equivalent to
$$\frac1{t^\alpha} \int_0^{t^\alpha \e^{-\lambda x}} \Big( -\sfrac1{\lambda} \log\big(\sfrac{n}{t^\alpha}\big) -x\Big)^\alpha \, dn
= \e^{-\lambda x} \int_0^\infty \lambda \e^{-\lambda u} u^\alpha \, du =
\e^{-\lambda x} \lambda^{-\alpha}\Gamma( \alpha+1),$$
using the substitution $u=-\frac1\lambda \log\big(\sfrac{n}{t^\alpha}\big)-x$. Hence we have that
$$\e^{-t} \big( \sfrac{(\lambda t)^\alpha}{\Gamma(\alpha+1)}\big)^{\frac1\lambda}  \, \max_n Z_n(t) \Longrightarrow \Phi_\lambda,$$
where $\Phi_\lambda$ is the Fr\'echet distribution with parameter $\lambda$. 
\smallskip\pagebreak[3]

This result, and further asymptotic results on the birthtime and fitness of the largest family, can be generalised to a framework where 
\begin{itemize}
\item 
$\mu$  is a sufficiently smooth distribution  in the maximum domain of attraction of either the  \emph{Weibull} or the \emph{Gumbel distribution} of extreme value theory, 
\item
the growth processes $(Z_n(\tau_n+s) \colon s\geq 0)$ are asymptotically %need not be deterministic but may be 
independent random processes with growth rates 
%of a more general form, for example Yule processes with parameter 
given as $\gamma F_n$, for some~$\gamma>0$,
\item the birth times $\tau_n$ 
are themselves random and may depend on the growth processes. % $(Z_n(\tau_n+s) \colon s\geq 0)$.
\end{itemize} 
Generalising the above calculation to such a setup requires, of course, more sophisticated methods. 
%We describe our results below,  see Theorem~\ref{theo:22} and Corollary~\ref{cor:23}. 
Our approach is to describe the state of a family at time $t$ as a point in the space $(-\infty, \infty) \times (-\infty, \infty) \times  (0,\infty)$, where the first coordinate corresponds to its birth time, the second to its fitness  and the third to its size at time $t$. If $\mu$ is in the maximum domain of attraction of the \emph{Weibull distribution}, introducing a $t$-dependent scaling of the three coordinates (so that the focus is on a carefully chosen window) and letting $t\to\infty$ we obtain a limiting point process, 
see Theorem~\ref{theo:22}. In this limiting process the point with the maximal 
third coordinate identifies the largest family, allowing to read off limit theorems for its size, fitness and birthtime, see Corollary~\ref{cor:23}.  
%A similar result in a different framework is contained in the paper \cite{Main}. 
{A similar result %(for $\mu$ in the domain of attraction of the Weibull distribution) 
was proved in \cite{Main} in the context of reinforced branching process. Our result extends that of \cite{Main} to the more general context of competing growth processes, allowing for a much wider range of applications.}
\smallskip

The \emph{main technical results} of the present paper %, Theorem~3 and Corollary~4, 
provide corresponding results for the case that $\mu$ is in the maximum domain of  attraction of the \emph{Gumbel distribution}. 
%This case is considerably more difficult because the technique of \cite{Main} cannot be applied. 
{This case is considerably more difficult than the Weibull case and new ideas are needed.}
The reason for this is that the window in which one has to search for the largest family is larger, having unbounded width in the first component. Therefore for a limit theorem the first component requires scaling, and hence the scaling of the second component depends not only on $t$ but also on~$n$, the birth rank  of the family. Using some additional regularity properties of the fitness distribution~$\mu$ however allows to make the scaling of the third component independent of $n$, so that we can still achieve a powerful Poisson limit theorem (Theorem~\ref{theo:theo_main}) as well as convergence of the scaled family size to a Fr\'echet distribution and of the standardised birth time to a  Gaussian distribution (Corollary~\ref{cor:lim_fam}). Taken together, our results give an essentially complete picture for the behaviour of the largest family for fitness distrbutions~$\mu$ with bounded support.
{Fitness distributions with unbounded support lead to superexponentially growing processes, which have more complex behaviour and cannot be treated here.}\pagebreak[3]
\smallskip

{\cec As application of our main technical result, we obtain results on the extremal behaviour of a variety of models that all fall under our general framework of competing growth processes:}
Our \emph{main examples} of competing growth processes originate from the study of dynamic network models. In these models new vertices get born at random times and are connected to existing vertices by certain rules. The degree of a vertex grows over time with a growth rate given by the attractiveness, or fitness, of the vertex.  {\cec We show asymptotic results} for the vertex of maximal degree at a large time $t$ and describe its degree, fitness and birthtime as a function of $t$: {\cec see Section~\ref{sec:BB} for the Bianconi and Barab\'asi network~\cite{Bianconi} and Section~\ref{sec:Dereich} for a model of Dereich~\cite{Dereich}}. {\cec Applications of our main technical result also include asymptotic results on the largest family in} the population process process with selection and mutation described above {\cec (see Section~\ref{exa:RBP})}, and on the largest tables in a disordered Chinese restaurant process for which we derive a surprising result 
on the relative sizes of the two largest occupied tables {\cec (see Section~\ref{sec:cycles})}.
%process of random permutations where new elements are inserted into the permutation with probabilities relating to cycle length and weight. 
We will {\cec explain how to get these results} in Section~\ref{sec:exa}.
%Before introducing our general setup we  look at two examples
%of dynamic network models that can be studied using our framework. A wider range of examples is 
%given in Section~2.
\smallskip

The paper is structured as follows. In Section~\ref{sec:definition} we give a full definition of our framework and assumptions on the embedded process and state the main results. Section~\ref{sec:f} gives examples of fitness distributions to which our results apply. Section~\ref{sec:exa} is devoted to a range of interesting examples of growth processes and  describes applications of our general results to these examples. 
The further sections are devoted to the proofs and 
their structure will be explained at the end of Section~\ref{sec:exa}.

%%%%%%%%%%%%%%%%%%%%%%%%%%%%%%%%%%%%%%%%%%%%%%%%%

\subsection{Our framework and principal technical results} \label{sec:definition}
Let $\mu$ be a probability distribution on {\peter the nonnegative real numbers with $\mathfrak s=$esssup$(\mu)<\infty$. To rule out less interesting cases we assume that
$\mu$ has no atom at zero or at $\mathfrak s$. Without loss of generality we can and will further assume that $\mathfrak s=1$ and hence that $\mu$ is supported on the interval $(0,1)$. Let}
\begin{itemize}
\item $(F_n)_{n\ge1}$ be i.i.d.\ $\mu$-distributed random variables;
\item $(\tau_n)_{n\ge1}$ be a non-decreasing sequence of positive random variables with $\tau_1 = 0$;
\item $Z_n(t) = X_n(F_n(t - \tau_n)){\cec \boldsymbol 1_{t\geq \tau_n}}$ for a family $(X_n(t):t\geq 0)_{n \geq 1}$ of non-decreasing integer-valued processes. 
\end{itemize}
Define $M(t) : = \max \{n \geq 1 \colon \tau_n \le t \}$ and \smash{$N(t) : = \sum_{n=1}^{M(t)} Z_n(t)$}.  We view this as a population of immortal individuals and we refer to $Z_n(t)$ as the size of the $n$th family, $M(t)$ the number of families in the system and $N(t)$ the total size of the population respectively, at time $t$. From this perspective $\tau_n$ represents the foundation time of the $n$th family. Furthermore, we see $F_n$ as a fitness parameter of the $n$th family, determining the rate at which new offspring are born into it. 
\smallskip
\pagebreak[3]

In this paper we aim at proving convergence results for the maximal family in the population. For this we require the following assumptions on the growth processes and fitness distribution. 
%%%%%%%%%%%%%%%%%%%%%%%%%%%%%%%%%%

\begin{myenv}[Families' foundation times]	
	 There exists $\lambda > 0$ such that for all $n \in \N$
	\begin{equation*}
		\tau_n =  \tau^*_n + T + \varepsilon_n, 
	\end{equation*}
where $\tau_n^* := \frac{1}{\lambda} \log n $,  $T$ is a finite random variable, and 
$\varepsilon_n \rightarrow 0$ almost surely as $n \rightarrow \infty$.
 	\label{as:Nerman}
\end{myenv}

%\begin{rem}
%	Assumption {\bf (A1)} implies that $\sup_{n\geq m} |\tau_n^* - \tau_n - T| \to 0$ in probability when $m\to\infty$. {\peterNecessary?}
%	\label{rem:1}
%\end{rem}

\begin{myenv}[Growth processes]
There exist $\gamma > 0$ and an i.i.d.\ sequence of  processes\\ ${((Y_n(t) : t \ge 0))_{n\ge 1}}$
% P  changed from ii.i.d.\  processes
 independent of $(F_n)_{n\ge 1}$, such that 
\begin{equation*}
	\Delta_n(t) := \sup_{u \ge t} \e^{-\gamma u} \Big |X_n (u) - Y_n(u)  \Big |
	\quad{\cec \text{ (defined for all }t\geq 0)}
\end{equation*}
satisfies for all $\varepsilon, {\kappa} > 0$, 
\begin{equation}
	\sup_{n\in \bor{I}(t)} \pr \big(\Delta_n (t) \ge \varepsilon \big|\, (F_i)_{i\in\N}\big) \rightarrow 0, \quad \text{
in probability as $t \rightarrow \infty$,}
	\label{equ:conv}
\end{equation}
where {$\bor{I}(t)$} is a collection of indices specified below in dependence on the fitness distribution~$\mu$.
% P replaced \bar{I}^*(t) by {I}^*(t) to make it loook better here. Specify the box later.
% := \big \{n :\frac{|\tau^*_n - \sigma_t|}{\sqrt{\sigma_t}} \le \kappa\big\}$. \\
	\label{as:growth2}
\end{myenv}
%
%\begin{note}[]
%	Note that, by the Dominated Convergence Theorem, Equation \eqref{equ:conv} implies $\Delta_n (t) %\rightarrow 0$ in probability as $t \rightarrow \infty$.
%\end{note}

\begin{myenv}[Growth rate]
	There exists a {\cec non-negative} random variable $\xi$ such that {\peter $$\E \big[\xi^{\frac{\lambda}{\gamma}}\big]<\infty \qquad \mbox{  and } $$}
	\\[-9mm]
	\begin{equation*}
	\e^{-\gamma t} Y_1(t) \longrightarrow \xi , \quad \text{ almost surely as $t \rightarrow \infty$.}
	\end{equation*} 
	{\cec The distribution of $\xi$ is absolutely continuous with respect to the Lebesgue measure. By $\nu$  we denote its density on $[0, \infty)$.} 
	%We further assume that $\nu$ has finite first moment. 
	\label{as:growth}
\end{myenv}
\ \\[-10mm]

\begin{myenv}[Concentration of growth] % P introduced eta for better usability and conditioning
There exist $c_0 , \eta> 0$ such that, for $n \in \N$, we have
\begin{equation*}
	\pr\big(\max_{u \ge 0} X_n(u) \e^{-\gamma u} \ge x \big|\, (F_m)_{m\in\N}\big) \le c_0 \e^{- \eta x}, \quad \text{ for all $x \ge 0$.}
\end{equation*}	%
	\label{as:fam_max}
\end{myenv}
\ \\[-20mm]

{\cec\begin{note}
On the one hand, Assumption {\bf (A1)} implies that, for all finite times $t$, the number of families born before time $t$ is finite almost surely. On the other hand, Assumption {\bf (A4)} implies that each family stays finite at all finite times almost surely. Assumptions {\bf (A1)} and {\bf (A4)} together thus imply that our competing growth process does not explode in finite time.
\end{note}}

Beyond these four assumptions on the growth processes we need assumptions on the fitness distribution~$\mu$. We discuss two different possible classes of fitness distributions~$\mu$. The first class,
the main case discussed in this paper,  corresponds to $\mu$ being in the maximum domain of attraction of the Gumbel distribution. We make the following assumptions.
\begin{myenv}[$\mu$ in the maximum domain of attraction of the Gumbel distribution]\ \\
The function {\cec $m : x\mapsto -\log \mu (x,1]$ defined for all $x\in [0,1)$} is twice differentiable {\cec on $[0,1)$} and satisfies
\begin{enumerate}[\emph{\bf{(A5.\arabic{enumi})}}\quad]%[(5.1)]
\renewcommand{\theenumi}{(A5.\arabic{enumi})}
	\item $m'(x) > 0$ and $m''(x) >0$ for all $x \in {\cec [0,1)}$; \label{as:i}
	\item $\lim_{x \uparrow 1} \frac{m''(x)}{(m'(x))^2} = 0$; \label{as:ii} 
	\item $\exists \varkappa > 0$ such that  $\lim_{x\uparrow 1} \frac{m''(x)m(x) x}{(m'(x))^2} = \varkappa$; \label{as:kappa}
	\item $\lim_{x\uparrow 1} \frac{m(x)}{m'(x)} = 0$.  \label{as:iv}
\end{enumerate}
\label{as:fit}
\end{myenv}

\begin{note}
Assumption~{\bf (A5)} is sufficient for $\mu$ to be in the maximum domain of attraction of the Gumbel distribution {\cec (see \cite[Section~1.1]{Resnick})}, and contains the most important cases, but it is not formally necessary.  We discuss this further in Section~\ref{sec:f}.
\end{note}
\pagebreak[3]

\noindent Under Assumption~{\bf (A5)}, {\cec for all $t\geq 0$}, we define $\sigma_t$ as {\peter the minimum of~1 (for technical reasons)} and the unique solution of
	\begin{equation}
		(\log g)'( \lambda x) = \frac{1}{\lambda(t-x) },
		\label{equ:sigma}
	\end{equation}
where $g(x) = m^{-1}(x)$, see Lemma \ref{lem:sigma} for a proof of existence and uniqueness of $\sigma_t$. We then define the collection of indices in {\bf (A2)} as
\begin{equation}
\bor{I}(t):= \big \{n :\sfrac{|\tau^*_n - \sigma_t|}{\sqrt{\sigma_t}} \le \kappa\big\}, {\quad \text{ for all } \kappa > 0, {\cec t\geq 0}.}
\label{equ:IA}
\end{equation}
%\smallskip

%\noindent 
The other class of distributions $\mu$ we consider is the maximum domain of attraction of the Weibull distribution class {\cec (see \cite[Section~1.2]{Resnick})}. 
%In this case we  state the results {\peter without proofs. The proofs can be found at the online respository linked in the bibliography to item~\cite{Senkevich} .}

\begin{myenb}[$\mu$ in the maximum domain of attraction of the Weibull distribution]
{\peter The distribution $\mu$  has a regularly varying tail in one, meaning that there exists }
$\alpha > 0$ and a slowly varying function $\ell$ such that 
		$\mu(1 - \varepsilon, 1) = \varepsilon^\alpha \ell(\varepsilon)$ {\cec for all $\varepsilon\in[0,1]$}.
	\label{as:RV}
\end{myenb}
{
\noindent {For all $t\geq 0$}, we set
\begin{equation}
\sigma_t := \tau_{n(t)}, \quad \text{where } \quad n(t) = \bigg \lceil \frac{1}{\mu(1 - t^{-1}, 1)}\bigg\rceil 
\label{equ:sigmaB}
\end{equation}
and use this to define 
\begin{equation}
\bor{I}(t) := \big \{n : | \tau^*_n - \sigma_t| \le 2 |T| + \kappa\big\} ,  \quad \text{ for all } \kappa > 0,
\label{equ:IB}
\end{equation}
for use in Assumption~{\bf (A2)}. Assumption~{\bf (B5)} implies that $n(t) = \lceil {t^\alpha}/{\ell(t^{-1})} \rceil $ and so 
$\log n(t) \sim \alpha \log t {\cec - \log(\ell(t^{-1}))}.$
Using this {\cec together with Assumption {\bf (A1)}} we can write 
\begin{equation*}
\tau_{n(t)}= \frac{1}{\lambda} \log n(t) + T + \varepsilon_{n(t)}  = \frac{\alpha}{\lambda} \log t - \frac{1}{\lambda} \log(\ell(t^{-1}))+ T + o(1),
\end{equation*}
{\cec almost surely} as $t \rightarrow \infty$, by Assumption~{\bf (A1)}. 
\pagebreak[3]
%for $n \in I(t)$,  
%\begin{equation*}
%	\tau_n - \tau_{n(t)} = \frac{1}{\lambda} \log \frac{n}{n(t)} + \varepsilon_n \rightarrow 0 \quad \text{ as $n \rightarrow \infty$.}
%\end{equation*}	}
\medskip
%%%%%%%%%%%%%%%%%%%%%%%%%%%%%%%%%%%%%%%%%%%%

%\subsection{Project 2}

%\noindent 
We now state our results, first in the easier case of  $\mu$ satisfying Assumption~{\bf (B5)}.  
For all $t \ge 0$, we define the point process
\begin{equation}
	\Gamma_t = \sum_{n = 1}^{M(t)} \delta \big (\tau_n - \sigma_t, t(1 - F_n), 
\e^{-\gamma (t - \sigma_t)} Z_n(t) \big),
%\e^{-\gamma t + \frac{\gamma\alpha}{\lambda} \log t + \gamma T }  Z_n(t) \big),
	\label{equ:ppG}
\end{equation}
on $(-\infty,\infty) \times (0,\infty) \times (0,\infty)$, where $\delta(x)$ is the Dirac mass at $x$.
We %will 
look at the limits of $\Gamma_t$, strengthening the result considerably by partially compactifying the underlying space.\\

\begin{theo}[Poisson limit]
	Under assumptions {\bf (A1-4)} and {\bf (B5)} {\peter as $t\to\infty$} the point process $(\Gamma_t)_{t \ge 0}$ converges vaguely\footnotemark\hspace{1pt} {in distribution} on the space $[-\infty, \infty] \times [0, \infty] \times (0, \infty]$ to the Poisson point process with {\peter locally finite} intensity measure 
	\begin{equation*}
		\di \zeta(s, f, z) = \alpha f^{\alpha-1} \lambda \e^{\lambda s} \e^{\gamma (s + f)}\nu(z \e^{\gamma (s+ f)}) \: \di s\: \di f\: \di z,
	\end{equation*}
where $\nu$ is as in {\bf (A3)}.
	\label{theo:22}
\end{theo}

\footnotetext{ {\peter We say that a sequence of Radon measures $(\mu_n)_{n \in \mathbb{N}}$ on 
a locally compact Polish space $\mathbb X$ converges \emph{vaguely} to $\mu$ %, and denote this by $\mu_n \xrightarrow{v} \mu$, 
iff
%	\begin{equation*}
$		\int f \di \mu_n \rightarrow \int f \di \mu,$ as $n \rightarrow \infty$,
%	\end{equation*}
for all continuous functions $f \colon \mathbb X \rightarrow \R$ with compact support.  This makes the space of Radon measures itself a Polish space, and if 
$(\mu_n)_{n\geq 0}$ is a sequence of {\it random} measures in this space we say that it converges vaguely in distribution to a random Radon
measure $\mu$ iff for all continuous bounded functions~$F$ on this space, the sequence $(\E F(\mu_n))_{n\geq 0}$ converges to $\E F(\mu)$ as~$n \rightarrow \infty$.
By the Portmanteau theorem the convergence also holds for bounded functions $F$ that are continuous at \emph{almost every} $\mu$.}
}

Observe that the compactification of the intervals in Theorem~1 ensures that the point with the largest $z$-component in the Poisson process corresponds asymptotically to the family of maximal size. 
Theorem~\ref{theo:22} therefore implies the following distributional limits (denoted by $\Rightarrow$) for the size, fitness and the foundation time of the largest family. 
{Note that the {\peter open bracket} in the third coordinate of the domain on which the point process converge is crucial, as the domain of convergence cannot be extended to $[-\infty, \infty]\times [-\infty, \infty]\times [0,\infty]$.}

\begin{cor}[Limits of family characteristics]\ \\[-5mm]
	\begin{enumerate}[(i)]
		\item  {\peter As $t \rightarrow \infty$, we have}
			\begin{equation*}
				\e^{-\gamma t + \frac{\gamma\alpha}{\lambda} \log t + \gamma T } 
\max_{n \in \N} Z_n(t)\Rightarrow W,
			\end{equation*}
			and $W$ is Fr\'echet distributed with  shape parameter $\nicefrac\lambda\gamma$ and scale parameter $$s=\big(\Gamma(\alpha + 1) \lambda^{-\alpha}   \E \big[\xi^{\frac{\lambda}{\gamma}}\big]\big)^{\frac{\gamma}{\lambda}}.$$
\item Denoting by $V(t)$ the fitness of the family of maximal size at time $t$, as $t \rightarrow \infty$, we have 
		\begin{equation*}
			t(1 - V(t)) \Rightarrow V,
		\end{equation*}
	where $V$ is Gamma distributed with shape parameter~$\alpha$ and scale parameter~$\lambda$. %changed from $\gamma$ CHECK!
		\item Denoting by $S(t)$ the birth time of the family of maximal size at time $t$, as $t \rightarrow \infty$, we have
			\begin{equation*}
				S(t) - \sigma_t \Rightarrow U,
			\end{equation*}
			where $U$ is a real valued random variable. 
	\end{enumerate}
	\label{cor:23}
\end{cor}
{\cec
\begin{note}
By Theorem~\ref{theo:22}, we know that, with high probability as $t\to\infty$, the family that realises $\max_{n\geq 1} Z_n(t)$ is unique, and thus the definitions of $S(t)$ and $V(t)$ are not ambiguous. 
\end{note}}

{\peter The proofs of Theorem~\ref{theo:22} and Corollary~\ref{cor:23} are carried out in the third author's PhD thesis~\cite{Senkevich} and can be found at the online respository linked in the bibliography to item~\cite{Senkevich}. They are not repeated here to limit the length of this paper. The proofs use ideas analogous to those in the present paper, but the execution of these ideas is much simpler. A similar result is contained in \cite{Main} in the less general context of reinforced branching processes (see Section~\ref{exa:RBP} for details about these processes) using methods that can neither be generalised to the broader class of models considered here, nor at all to the Gumbel case.}
\bigskip

%%%%%%%%%%%%%%%%%%%%%%%%%%%%%%%%%%%%%%%%%%%%%%%%%

%%%%%%%%%%%%%%%%%%%%%%%%%%%%%%%%%
To now state \emph{our main technical result} we look at fitness distributions satisfying Assumption~{\bf (A5)}. For all $t \geq 0$, we define 
\begin{equation}
	\Gamma_t = \sum^{M(t)}_{n=1} \delta\Big( \frac{\tau_n - \sigma_t } {\sqrt{\sigma_t }}, \frac{F_n - g \big(\log(n \sqrt{\sigma_t }) \big)}{g'\big(\log(n \sqrt{\sigma_t }) \big)}, \e^{- \gamma g(\lambda \sigma_t)(t-\sigma_t )- a_1 g(\lambda \sigma_t) \log \sigma_t + \gamma T} Z_n(t) \Big), 
	\label{equ:ppg}
\end{equation}
where $\delta(x)$ is the Dirac mass at $x$, and $ a_1  : =  \frac{ \gamma}{2 \lambda}$.
{\cec Note that, by definition (see Equation~\eqref{equ:sigma}), $\sigma_t\geq 1$ almost surely, and thus $\log(n \sqrt{\sigma_t })\geq 0$ lies in the domain of definition of $g$ for all $t\geq 0$ and $n\geq 1$, implying that $\Gamma_t$ is well defined for all $t\geq 0$.}
\begin{theo}[Poisson limit]Under Assumptions {\bf (A1-5)}, {\cec as $t\to\infty$,} the point process $(\Gamma_t)_{t \ge 0}$ converges vaguely  {in distribution} on the space $[-\infty, \infty] \times [-\infty, \infty] \times (0, \infty]$ to the Poisson point process with {\peter locally finite} intensity measure 
	\begin{equation*}
		\di \zeta(s, \: f, \: z) = \lambda \e^{-f} \e^{s^2 a_2 - f a_3} \nu (z\e^{s^2 a_2 -f a_3}) \, \di s \: \df \: \di z,
	\end{equation*} 
where $ a_2 : = \frac{\gamma}{2} \varkappa$, $a_3 : = \frac{\gamma}{\lambda}$ and $\nu$ is as in (A.3). 
	\label{theo:theo_main}
\end{theo}

\begin{note}
The existence of a density for the random variable $\xi$ is assumed in {\bf (A3)} for convenience. For example,
Theorems~\ref{theo:22} and~\ref{theo:theo_main} continue to hold if $\nu=\delta_1$ as in our motivating example. 
\end{note}

The technical difference between Theorems~\ref{theo:22} and~\ref{theo:theo_main} is that in the latter the 
first (birthtime) coordinate needs to be scaled. As a result the scaling of the second (fitness) component depends on the birth rank $n$ of the family as well as on the observation time $t$. Therefore we cannot derive a general scaling limit for the fitness of the largest family as in Corollary~\ref{cor:23}. Results for the birth time and size of this family, however, are still possible.

\begin{cor}[Limits of family characteristics]{}
\ \\[-4mm]
	\begin{enumerate}[(i)]
		\item {\peter As $t \rightarrow \infty$, we have}
			\begin{equation*}
				\e^{-\gamma g(\lambda \sigma_t)(t- \sigma_t ) - a_1 g(\lambda \sigma_t) \log \sigma_t + \gamma T} \max_{n \in \N} Z_n(t) \Rightarrow W, 
			\end{equation*}
		where $W$ is Fr\'echet distributed with shape parameter $\nicefrac\lambda\gamma$ and scale parameter 
		$$s=\big(\sqrt{\sfrac{2 \pi\lambda }{\varkappa}}  \E \big[\xi^{\frac{\lambda}{\gamma}}\big]\big)^{\frac{\gamma}{\lambda}}.$$
		\item Denoting by $S(t)$ the birth time of the family of maximal size at time $t$, as $t \rightarrow \infty$, we have
			\begin{equation*}
				\frac{S(t) - \sigma_t } {\sqrt{\sigma_t }} \Rightarrow U,
			\end{equation*}
		where $U$ is normally-distributed with mean $0$ and variance $\frac{1}{\lambda \varkappa}$. 
	\end{enumerate}
	\label{cor:lim_fam}
\end{cor}

\begin{note}
Observe that irrespective of whether $\mu$ is in the maximum domain of attraction of the Weibull or Gumbel
distribution, the size of the largest family scaled by a deterministic function of time and the random factor $\e^{\gamma T}$ converges to a Fr\'echet distribution.
\end{note}

{\cec The intuition behind Theorem~\ref{theo:theo_main} (see Sections~\ref{sec:local_convergence} and~\ref{sec:compactification} for the proof, and Figure~\ref{fig} for a visual aid) is that the only families that have a chance to be the largest at (large) time~$t$ are the ones born at time $\sigma_t \pm \mathcal O(\sigma_t)$ and whose fitness is of order $g(\lambda\sigma_t)$.} {\peter This fixes a ``moving window" in which we have to look for the representation of the largest family by its birth time and fitness.
%Informally, we say that these families are inside ``the window'': w
We prove in Section~\ref{sec:local_convergence} that the point process $\Gamma_t$ restricted to the moving window converges to a Poisson point process, and, in Section~~\ref{sec:compactification}, that the probability that a family outside the window is largest converges to zero.}
%be seen in the limiting Poisson process.}

%%%%%%%%%%%%%%%%%%%%%%%%%%%%%%%%%%%%%%%%%%%%%%%%%
% GRAPHIC
%%%%%%%%%%%%%%%%%%%%%%%%%%%%%%%%%%%%%%%%%%%%%%%%%
\setlength{\unitlength}{1cm}

\begin{figure}
\begin{center}
\begin{picture}(12,7)
\thicklines
\put(2,0.5){\vector(0,1){5.5}} % y-axis
\put(1.5,1){\vector(1,0){11}} % x-axis

\put(1.6,0.5){$0$}
\put(1.6,4.9){$1$}

\put(2,5){\line(1,0){9}}
\put(11,5){\line(0,-1){4}}

%Box
\qbezier(4.7,3.5)(5.9,4)(7.1,4.0) % bottom line of the box
{\color{lightgray} % Shading
\linethickness{1mm}
\qbezier(4.75,3.7)(5.9,4.2)(7.05,4.2) 
\qbezier(4.75,3.6)(5.9,4.05)(7.05,4.075) 
\qbezier(4.75,3.7)(5.7, 4.25)(7.05,4.55) 
\linethickness{2mm}
\qbezier(4.80,3.7)(5.7, 4.2)(7.0,4.4) 
\qbezier(4.80,3.7)(5.7, 4.2)(7.0,4.3) 
\qbezier(4.80,3.7)(5.7, 4.2)(7.0,4.2) 
}
\qbezier(2, 2)(5.7,5.2)(11,4.9) % continuation of the upper line box

\put(4.7,3.5){\line(0,1){0.28}} %left side of the box
\put(7.1,4.0){\line(0,1){0.6}} %right side of the box

% lines
\put(4.7,3.5){\line(0,-1){2.7}} % left
\put(7.1,4.1){\line(0,-1){3.3}} % right

\put(5.9,3.4){\vector(1,0){1.2}} % vector right
\put(5.9,3.4){\vector(-1,0){1.2}} % vector left

% dash lines

 \multiput(6,4.2)(-0.8,0){6}{\line(1,0){0.16}} % horisontal
 \multiput(6.08,4.02)(0,-0.7){5}{\line(0,-1){0.16}} % vertical 

%text
\put(11,0.5){$t$}
\put(6,0.5){$\sigma_t$}
\put(0.65,4.1){\footnotesize $\approx g(\lambda \sigma_t)$}
\put(0.5,5.5){ Fitness}
\put(11.5,0.5){Time}

\put(2.5,4.5){\footnotesize No families (Section \ref{sec:old})}
\put(4.9,2.7){\footnotesize Unfit families}
\put(5.0,2.3){\footnotesize (Section \ref{sec:unfit})}
\put(8.1,4.3){\footnotesize Young families}
\put(8.3, 3.9){\footnotesize (Section \ref{sec:young})}
\put(2.7,2){\footnotesize Old families}
\put(2.65,1.6){\footnotesize  (Section \ref{sec:young})}

\put(5.5,3.53){\footnotesize $\mathcal{O}(\sqrt{\sigma_t})$}

%dots inside the box
\put(6.95,4.45){\circle*{0.12}}
\put(6.55,4.2){\circle*{0.12}} %16
\put(5.91,4.07){\circle*{0.12}} % 25biggest dot
\put(6.21,4.02){\circle*{0.12}}
\put(5.16,3.93){\circle*{0.12}} 
\put(5.1,3.81){\circle*{0.12}} %18
\end{picture}
\end{center}
\caption{\cec A graphical representation of  the proof of Theorem~\ref{theo:theo_main}. {\peter Families are represented by their birth time and fitness.}
The largest family at large time~$t$ is most likely born at time $\sigma_t\pm \mathcal O(\sqrt{\sigma_t})$, and has fitness of order $g(\lambda\sigma_t)$. 
{\peter Loosely speaking, families that are too old are not fit enough to be large enough, families that are too young have not had sufficient time to grow, and 
families with a small fitness grow too slowly to compete. Only families in the shaded window appear in the limiting Poisson point process and compete to be the largest.}}
\label{fig}
\end{figure}
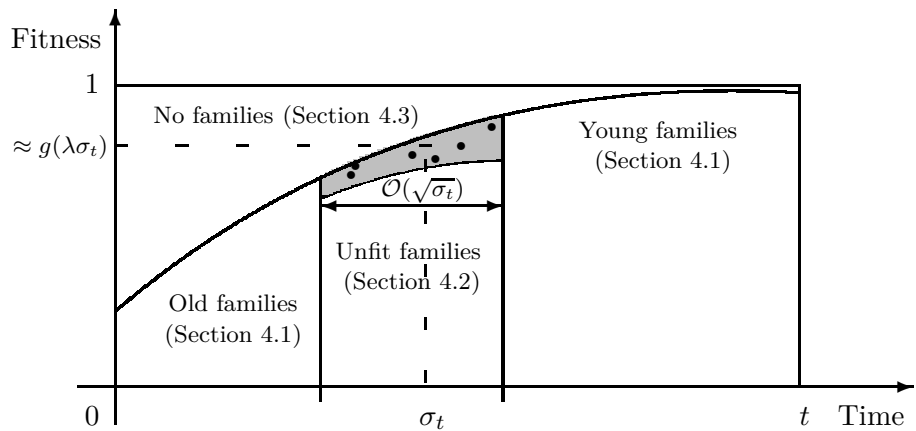

%%%%%%%%%%%%%%%%%%%%%%%%%%%%%%%%%%%%%

%%%%%%%%%%%%%%%%%%%%%%%%%%%%%%%%%%%%%%%%%%%%%
\subsection{Examples of fitness distributions}
\label{sec:f}
The five following functions $m(x) = - \log \mu(x,1)$, {\cec defined for all $x\in[0,1)$}, satisfy Assumption {\bf (A5)}:
\begin{enumerate}[(1)]
	\item	$m(x) = (1-x)^{-\varrho} -1$, where $\varrho > 0$; \label{exa:f1}
	\item $m(x) = \e^{\frac{1}{1-x}} - \e$;
	\item $m(x) = \frac{x}{1 - x}$;
	\item $m(x) = \e^{\frac{1}{\sqrt{1 - x}}} - \e$;
	\item $m(x) = \tan \big( \frac{\pi x}{2}\big)$.
\end{enumerate}	
%We prove the following lemma and leave the four other cases to the reader. 
%\begin{lem}
%	The function $m(x) = (1-x)^{-\eta} -1$, where $\eta > 0$, satisfies Assumption {\bf (A5)}. 
%	\label{lem:fit1}
%\end{lem}
%\begin{proof} First of all, we have $m(0) = 0 $ and $m(1) = \infty$, so that $\mu(x,1) = \e^{-m(x)}$ is a probability measure. We now consider each condition {\bf (A5.1)} to {\bf (A5.4)} in order. We have ${m'(x) = \eta(1-x)^{-(\eta+1)} > 0}$ and ${m''(x) = \eta(\eta + 1)(1-x)^{-(\eta + 2)} > 0}$, for all $x \in [0,1]$, which implies {\bf (A5.1)}. Also, we have that
%		\begin{equation*}
%			\lim_{x \uparrow 1} \frac{m''(x)}{(m'(x))^2} = \lim_{x \uparrow 1}  \frac{(\eta+1)(1-x)^\eta }{ \eta } = 0,
%		\end{equation*}
%which implies {\bf (A5.2)} and		
%		\begin{equation}
%			\lim_{x \uparrow 1}  \frac{m''(x) m(x) x}{(m'(x))^2} = \lim_{x \uparrow 1}  \frac{(\eta + 1) x \big(1 - (1-x)^\eta \big)}{\eta} = \frac{\eta + 1}{\eta} >0,
%			\label{equ:k}
%		\end{equation}
%which implies {\bf (A5.3)} and $\varkappa = \frac{\eta + 1}{\eta}$.	Finally, we get	
%		\begin{equation*}
%			\lim_{x\uparrow 1} \frac{m(x)}{m'(x)} = \lim_{x\uparrow 1} \frac{(1-x)^{-\eta} - 1}{\eta (1-x)^{-(\eta + 1)}} = 0,
%		\end{equation*}
%which implies {\bf (A5.4)}.	
%\end{proof}

Assumptions {\bf (A5.1)} and {\bf (A5.2)} imply that the fitness distribution $\mu$ lies in the maximum domain of attraction of the Gumbel distribution, see \cite[ch. 3.3.3]{MEE}. 
%More precisely, let $(F_n)_{n\ge1}$ be a sequence of i.i.d.\:random variables with common distribution $\mu$, then
%	\begin{equation}
%		\frac{\max_{i\le n} F_i - g(\log n)}{g'(\log n)} \rightarrow V, \quad \text{weakly in distribution as $n \rightarrow \infty$,}
%		\label{equ:gumbel}
%	\end{equation}
%where $V \sim \Lambda$, i.e. $\pr(V \ge x) = \exp\{-\e^{-x}\}$, for all $x\in \R$ (for more details see \cite[ch. 0.3]{Resnick}). \\
%
%\noindent 
Although most of the natural examples satisfy Assumptions {\bf (A5.3)} and {\bf (A5.4)}, some probability distributions in the maximum domain of attraction of the Gumbel distribution do not fall into our framework,
for example\begin{enumerate}[(6)]
	\item
		$m(x) =  \log \Big( \frac{\e}{1-x}\Big) \log \log \Big( \frac{\e}{1-x} \Big),$
	\end{enumerate}
see \cite{Senkevich, Resnick} for details. 
\pagebreak[3]

\section{Examples and applications}
\label{sec:exa}

In this section we present a selection of examples covered by our main results. We emphasise that our framework goes well beyond the setup of reinforced branching processes treated in \cite{Main} and also that we pick only a small number of representative results out of a wealth of consequences that we can draw from Theorem~3 and Corollary~4. 
%
%All proofs are done in Section~\ref{appendix}.

\subsection{Branching processes with selection and mutation}
\label{exa:RBP}

\subsubsection{\cec A simple selection and mutation model}\label{sub:Kingman}
{\peter Our first example is the model of a population evolving by selection and mutation mentioned at the beginning of Section~\ref{sec:intro}, which we embed into continuous
time as follows.} We start with one individual with genetic fitness sampled from $\mu$. Individuals never die and give birth
at a rate given by their fitness to an independent random number of offspring. Note that variations in individual fitness lead to a selection effect: an individual born at time $t$ selects its parent from the population alive at time $t$ with a probability proportional to their fitness.  At birth each individual independently either inherits the parent's fitness or, with probability $0<\beta<1$, is a mutant getting a fitness sampled from $\mu$ independently of everything else. Similar to the deterministic 
Kingman's model \cite{Kingman2, Kingman1} at mutation all genetic information from a particle's ancestry  is lost. For a discussion of the relevance of these models in the theory of evolution see \cite{Evolution}. 
\smallskip

\pagebreak[3]
In our framework the non-decreasing sequence of birth times $\tau_1, \tau_2, \ldots$ of mutants 
constitute the foundation times of new families, their fitnesses are $F_1, F_2,\ldots$ and $Z_n(t)$ is the
number of non-mutant offspring of the $n$th mutant at time $t$. If $(p_k)_{k\geq 1}$ is the distribution of offspring numbers at a birth event denote by $\emm=\sum_{k\geq 1} k p_k$ the mean offspring number and assume that % there exists $\eta>0$ such that $\sum_{k\geq 1} \e^{\eta k} p_k<\infty$. 
{\peter $(p_k)_{k\geq 1}$} has finite support.
We assume that mutations have a reasonable chance to produce fit individuals, as expressed in the Malthusian condition % P checked!
\begin{equation*}
\beta \int_0^1 \frac{\di \mu (x)}{1 - x}  > 1.
\end{equation*}
Under this condition there is a unique solution $\lambda>(1-\beta)\emm$ of the equation
\begin{equation*}
\beta \emm \int_0^1 \frac{x}{\lambda - (1-\beta)\emm x} \, \di \mu (x) = 1.
\end{equation*}
In Section~\ref{sub:GRBP}, we show that {\bf (A1-4)} are satisfied  with $\gamma = (1-\beta) \emm$. If $p_1=1$ this is a reinforced branching process as studied in \cite{Main}. The generalisation to arbitrary offspring distribution is not difficult (see Section~\ref{sub:GRBP} for details). As an example of the limit theorems implied by our main result, we look at the birth time $S(t)$ of the largest family at time $t$ in the case of Gnedenko's distribution (Example (3) in Section 1.3)
$$\mu(x,1)=\e^{-\frac{x}{1-x}}, \qquad \mbox{ for } {\cec x\in [0,1)},$$
see \cite[Exemple~2]{Gnedenko}. We find a leading order term for $S(t)$ of
$$\sigma_t=\frac1{\lambda}\big( \sqrt{\lambda t+1}-1\big)$$
and $\varkappa=2$. Corollary~4 therefore gives a central limit theorem of the form
$$\frac{S(t)-\sqrt{t/\lambda }}{\sqrt[4]{t/\lambda}} \to {\mathcal N}\big(0,(2\lambda)^{-1}\big) \text{ in distribution as }t\to\infty.$$

\subsubsection{\cec General reinforced branching processes}\label{sub:GRBP}
We now give a general construction for the reinforced branching process where at a birth event, {\cec for all $i,j\in\{0, 1, \ldots\}$,} with probability $p_{ij}$ we create $i$ 
new offspring of the same family and $j$ new families. {\peter We assume $p_{00}=0$ and denote  the first and second marginal by $(p^{\ssup 1}_i)$ and $(p^{\ssup 2}_j)$,   with positive and finite means  $m^{\ssup 1}$ and $m^{\ssup 2}$, respectively. Hence, as individuals are immortal, the branching process is supercritical.  We also assume that the first marginal is bounded, that is, it has finite support.} We can construct the model on an explicit probability space. Let
	\begin{itemize}
		\item $F$ be a $\mu$-distributed random variable,
		\item independently of $F$ construct a continuous time jump process $Y=(Y(t):t \ge 0)$ as follows
\begin{itemize}
\item start at time $0$ in state $Y(0)=1$,
\item if $Y$ is in state $k\in\N$ the next jump event follows at rate $k$,
\item let $0<t_1<t_2<t_3< \ldots$ be the increasing sequence of times at which jump events happen,
\item at  jump time $t_n$ sample a pair $(J_n,L_n) \in\N_0 \times \N_0$ %independently
from $(p_{ij})$ and 
increase $(Y(t):t \ge 0)$ by $J_n$ (which may be zero), {\cec i.e.\ set $Y(t)= Y(t_{n-1})+J_n$ for all $t\in [t_n, t_{n+1})$}.
\end{itemize}
\item given the above let $\Pi = (\Pi(t) : t \ge 0)$ be the jump process which has a jump of height $L_n$ (which may be zero) at time $t_n$.
	\end{itemize}
{\cec We let $((F_n, Y_n, \Pi_n))_{n\geq 1}$ be a sequence of i.i.d.\ copies of $(F,Y,\Pi)$.}
%We let $(\Omega, \mathcal{F}, \pr)$ be the countable product of the joint law of $(F, Y, \Pi)$ and denote the coordinate process by $(F_n, Y_n, \Pi_n)$, for $n \in \N$. 
The process $(Y_n(t) \colon t\ge 0)$ describes the
creation of new family members, and the process $(\Pi_n(t) \colon t\ge 0)$ the creation of new families 
descending from
the $n$th family (in a standardised time-scale). To construct our original objects on this probability space we
let $\tau_1 = 0$ and $Z_1(t) = Y_1(F_1t)$ and, {\peter for $n\geq 2$ and $\tau_1, \ldots, \tau_{n-1}$ already constructed}, iteratively define {\cec (recall that, for all $t\geq 0$ and $m\geq 1$, we denote by $\Delta\Pi_m(t) = \Pi_m(t)-\Pi_m(t-)$)}
	\begin{equation*}
		\tau_n=\inf\{t > \tau_{n-1} : \exists m \in \{1, ..., n-1\} \text{ with } \Delta \Pi_m(F_m(t - \tau_m)) >0\},
	\end{equation*}	
and if $\Delta \Pi_m(F_m(t - \tau_m))=k\ge 2$ also set
$\tau_{n+k-1}=\cdots =\tau_{n+1}=\tau_n$. {\peter For $j=0,\ldots, k-1$ let}
	\begin{equation*}
		Z_{\peter n+j}(t) = \begin{cases} Y_{\peter n+j}(F_{\peter n+j}(t-\tau_{\peter n+j})), \quad & \text{if } t \ge \tau_{\peter n+j} \\
		0, \quad & \text{otherwise.}
		\end{cases} 		
	\end{equation*}
We let $M(t) = \max\{n: \tau_n \le t\}$ and  $N(t) = \sum_{n=1}^{M(t)} Z_n(t)$.
%, and denote by $\rho_1, \rho_2, ...$ the jump times of $(N(t): t \ge 0)$. 
%This construction defines a reinforced branching process described in the introduction. 
Now $(Y_n(F_n(t - \tau_n)) \colon t\ge \tau_n)$ gives the sizes of the $n$th family, and $(\Pi_n(F_n(t - \tau_n)) \colon t \ge \tau_n)$ the times of creation of the new families which descend directly from the $n$th family.  This construction defines a reinforced branching process in a slightly more general way than in \cite{Main}. 
\medskip

We now check that reinforced branching processes with, for some $\eta'>0$,
\begin{equation}\label{eq:cond_pij}
\sum_{i,j{\cec \geq 0}} \e^{\eta' (i+j)} p_{ij} < \infty 
\end{equation}
satisfy Assumptions {\bf (A1-4)}.
The process $(M(t): t > 0)$ is a general branching process, also known as a Crump-Mode-Jagers process, with the laws of offspring times given by the random point process $(\Pi^*(t):t>0)$ given by $\Pi^*(t)=\Pi(Ft)$. Assuming that there exists $\lambda >0$, called \emph{Malthusian parameter}, such that
\begin{equation}\label{malthus}
\int^\infty_0 \e^{-\lambda s} \E \Pi^* (\di s) = 1,
\end{equation}
we can apply a strong law of large numbers by Nerman (see \cite{Nerman}) which shows that 
under an $x \log x$ condition on $\Pi^*$ there exists a positive, {\peter finite} random variable $W$, such that 
\begin{equation*}
\lim_{t \rightarrow \infty} \e^{-\lambda t} M(t) = W \quad \text{ almost surely.}
\end{equation*}
This gives us that $\log M(t)=\log W +\lambda t +o(1)$ almost surely {\cec as $t\uparrow\infty$}.
 {\peter Hence $\tau_n\uparrow\infty$ as $n\to\infty$ since $M(t)<\infty$ for all $t\geq 0$, otherwise $W$ would be infinite with non-zero probability.} 
 Plugging $t=\tau_n$ yields that
$\tau_n=\frac1\lambda \log n+T + \varepsilon_n$ for $T=- \frac1\lambda \log W$ and a sequence $(\varepsilon_n)$ converging to~$0$ almost surely.% 
\medskip%

{\peter  Note that $(Y(t)\colon t > 0)$ is a supercritical continuous-time Galton-Watson process with $\mathbb E Y(t) = \mathrm e^{m^{\ssup 1}t}$, see also
Lemma~\ref{lem:Yule} below. Given $(Y(s)\colon s<t)$ the jump rate of  $\Pi$  at time $t$ is  $Y(t-)$ and when it jumps, its increment is distributed as $p^{\ssup 2}$, 
and thus we get that
$$\mathbb E\Pi^*(\mathrm ds) = \int_0^1m^{\ssup 2} {\cec f} \mathrm e^{m^{\ssup 1}fs}\mathrm d\mu(f)
{\cec \mathrm ds}.$$ 
Therefore,} the Malthusian condition~\eqref{malthus} reads as
\begin{align*}
1 & = \int^\infty_0 \e^{-\lambda s} \, \E \Pi^* (\di s) =  m^{\ssup 2} \int \frac{f}{\lambda-fm^{\ssup 1}} \, \mu(\mathrm df),
\end{align*}
which has a solution $\lambda>m^{\ssup 1}$ if and only if
$$m^{\ssup 2}\int_0^1 \frac{f}{1-f} \, \mu(\mathrm df)>m^{\ssup 1}.$$
The $x \log x$ condition  states that for the random variable $X= \int_0^\infty  \e^{-\lambda s} \Pi^*(ds)$ we have $\E X \log^+ X<\infty$. It is 
{straightforward} to check that under our assumption on the moments of $(p_{ij})$ we even have $\E X^2<\infty$ so that this condition and hence {\bf (A1)} holds. 
{\cec Indeed, we have
\begin{align*}
\mathbb E X^2
&=\int_0^{\infty}\int_0^{\infty} \mathrm e^{-\lambda(s+u)}\mathbb E[\Pi^*(\mathrm ds)\Pi^*(\mathrm du)] \\
&=\int_0^1\mathrm d\mu(f)\left(
\int_0^\infty\int_0^\infty (m^{\sss (2)})^2 f^2\mathrm e^{-(\lambda-m^{\sss (1)}f)(s+u)} \mathrm ds\mathrm du 
%\int_0^1\mathrm d\mu(f)
{\peter + \int_0^{\infty} \mathbb E[\zeta^2\, ] f\mathrm e^{-(2\lambda-m^{\sss (1)}f)s} \mathrm ds} \right),
\end{align*}
where {\peter $\zeta$ has the law $p^{\sss (2)}$.}
We thus get
\[\mathbb EX^2
=\int_0^1 \frac{f^2(m^{\sss (2)})^2}{(\lambda-m^{\sss (1)}f)^2}
+ \frac{f\mathbb E\zeta^2}{2\lambda-m^{\sss (1)}f}
\,\mathrm d\mu(f)
\leq \frac{(m^{\sss (2)})^2}{\lambda-m^{\sss (1)}} 
\int_0^1 \frac{f}{\lambda-m^{\sss (1)}f}\mathrm d\mu(f) + 
\mathbb E\zeta^2 \int_0^1\frac{f}{\lambda-m^{\sss (1)}f}\mathrm d\mu(f),
\]}%
%where we have used {\peter the Cauchy-Schwarz} inequality for the first integral, {\color{red}how?} 
{\peter where, for the first integral, we have used the fact that if $\lambda>m^{\sss (1)}$, then $\frac{f}{\lambda-m^{\sss (1)}f}$  is bounded by $1/(\lambda-m^{\sss (1)})$.
%and, for the second integral, the fact that $2\lambda\geq \lambda$.
Since $\lambda$ is the Malthusian parameter, we get
\[\mathbb EX^2\leq \frac{(m^{\sss (2)})^2}{\lambda-m^{\sss (1)}} +\mathbb E\zeta^2<\infty,\]
because $\mathbb E\zeta^2<\infty$, by Equation~\eqref{eq:cond_pij}; 
this implies that the $x\log x$ condition is indeed satisfied.
} 
\smallskip

We let  $Y_n = X_n$ so that $\Delta_n(t) = 0$ {\cec for all $t\geq 0$}, so the convergence in Assumption~{\bf (A2)} is trivially satisfied. The process $(Y_n(t) \colon t\geq 0)$ is a continuous-time Galton-Watson process with offspring distribution {\peter $(p^{\ssup 1}_i)$, where the immortal individual itself is not counted as offspring,} and hence Assumptions {\bf (A3-4)} follow from Lemma~\ref{lem:Yule} below, parts~(c),(d) and~(e), respectively.

\begin{lem}[Galton-Watson process $(Y(t) \colon t\geq 0)$ with {\peter bounded} offspring distribution $(p^{\ssup 1}_i)$]\ \\%[-5mm]
With $\gamma = m^{\ssup 1}$, we have
	\begin{enumerate}[(a)]
		\item $\E[Y(t)] = \e^{\gamma t}$.
 % $\E[\Pi(t)] = \frac{m^{\ssup 2}}{m^{\ssup1}}\, \e^{m^{\ssup 1} t}$
		\item $(\e^{-\gamma t} Y(t))_{t\ge0}$ is a uniformly integrable martingale.
		\item The almost sure limit of $\lim_{t \rightarrow \infty} \e^{-\gamma t} Y(t)$ is an absolutely continuous random variable~{\peter$\xi$}.
		\item There exists $\eta>0$ such that  $\E \exp \{ \eta \xi \} <\infty$.
		\item There exists $c_0>0$ such that
		$\displaystyle
		\pr\big(\max_{t \ge 0} \e^{-\gamma t} Y(t)   \ge x \big) 
\leq c_0  \, \e^{-\eta x},$ for all $x\geq 0$.
	\end{enumerate}
	\label{lem:Yule}
\end{lem}
\begin{proof} (a), (b), (c) are standard. See Athreya and Ney \cite{Ney} {\peter for (a) and Theorem~III.7.2 therein for (c)}, {\cec and Asmussen and Hering~\cite[Theorem~2.1]{Asmussen} for (b). Note that the latter is stated for the discrete-time Galton-Watson process, but this implies uniform integrability also for the continuous-time process.}
Denote the martingale limit in (c) by {\peter $\xi=\e^{-\gamma \infty} Y(\infty)$.} 
{\cec (d) follows from~\cite[Corollary 2.2]{Liu}.}
%For  (d) note that, as in \cite[Theorem~2.1]{Liu}, one can use the {\cec moment} assumption on $(p_{ij})$ {\cec (see Equation~\eqref{eq:cond_pij})} to check that  there exists $\eta>0$ 
%such that
%{\peter $\liminf \E \exp \{ \eta \e^{-\gamma t} Y(t)\} <\infty$ and hence  $\E \exp \{ \eta \xi \} <\infty$.
By Jensen's inequality,  $(\exp \{ \eta \e^{-\gamma t} Y(t)\} : t \in[0,\infty])$ is a sub-martingale.
Doob's weak maximal inequality, see \cite[Page 443]{Doob}, gives
$\pr\big(\max_{t \ge 0} \e^{-\gamma t} Y(t) \ge x \big) 
= \pr\big(\max_{t \ge 0} \exp\{ \eta \e^{-\gamma t} Y(t)\}  \ge \e^{\eta x} \big) 
\leq \E [\exp( \eta \xi)] \, \e^{-\eta x}.$
\end{proof}

The selection and mutation model of Section~\ref{sub:Kingman} is a particular case of this more general framework of reinforced branching processes. Recall that at every reproduction event, there is a random number of offspring distributed as $(p_k)_{k\geq 1}$, and each of them is a mutant with probability $\beta$, independently from the rest. Therefore, for all $i,j\geq 0$,
%its offspring distrib $(p_k)$ at a birth event and mutation probability $\beta$ the process becomes a reinforced branching process with offspring distribution $(p_{ij})$ given~by
\[p_{ij}=p_{i+j} \binom{i+j}{i} \, (1-\beta)^i \beta^{j},\]
so that $m^{\ssup 1}=(1-\beta)\emm$ and $m^{\ssup 2}=\beta \emm$,
where $\emm = \sum_{k\geq 1}kp_k$. {\peter Moreover, if  $(p_k)_{k\geq 1}$ has finite support, so has
the first marginal \smash{$(p^{\ssup 1}_k)_{k\geq 0}$}.}

%Recall that, in Section~\ref{sub:Kingman}, we assume that
 %there exists $\eta>0$ such that $\sum_{k\geq 0}p_k \mathrm e^{\eta k}<\infty$. This implies that
%\[\sum_{i,j\geq 0} \binom{i+j}i p_{i+j}(1-\beta)^i\beta^j \mathrm e^{\eta(i+j)}
%= \sum_{\ell=0}^\infty \sum_{i=0}^\ell \binom{\ell}i p_\ell(1-\beta)^i\beta^{\ell-i} \mathrm e^{\eta\ell}
%= \sum_{\ell=0}^\infty p_\ell \mathrm e^{\eta\ell}<\infty,\]
%implying that Equation~\eqref{eq:cond_pij} is satisfied, as required.

\subsection{Preferential attachment networks with fitness} \label{sec:preferential}

\subsubsection{\bf Preferential attachment tree of Bianconi and Barab\'asi} \label{sec:BB} \smallskip

This model is a random tree where at each step a new vertex is added and connected to an existing vertex with a probability depending on the fitness of the vertices. The model was 
introduced by Bianconi and Barab\'{a}si in \cite{Bianconi}.
%\smallskip
%
We start with two vertices connected by an edge, and endowed with fitnesses sampled independently 
from $\mu$. At every step $n\ge 3$ a new vertex arrives, gets a fitness sampled from $\mu$
independently of everything else, and connects to one existing vertex chosen randomly from the $n-1$ existing vertices 
with a probability proportional to the product of their fitness and their degree. \smallskip

{\cec The preferential attachment tree of Bianconi and Barab\'asi can be embedded in continuous time and then represents a reinforced branching process as in \cite{Main}, its continuous-time embedding is the reinforced branching process (see Section~\ref{sub:GRBP}) with \smash{$p_{11}=1$}  so that
\smash{$m^{\ssup1}=m^{\ssup 2}=1$}, see \cite{Main} for details. Here families correspond to vertices and the family size  is the vertex degree. At every birth event a new vertex of degree one (equivalently a new family) is created and by establishing an edge to an existing vertex the degree of this vertex is increased by one (equivalently one existing family is getting a new member). At time~$\tau_n$ the $n$th vertex is introduced and, for $m>n$,  the degree of this vertex when the $m$th vertex is introduced is~$Z_n(\tau_m)$.}
In this embedding $\tau_n$ is the birthtime 
of the $n$th vertex, $F_n$ its fitness and $Z_n(t)$ its degree at time~$t$.
We showed in Section~\ref{sub:GRBP} that under the Malthusian condition
$$\int_0^1 \frac{\mu(dx)}{1-x} > 2$$
the process satisfies Assumptions
{\bf (A1-4)} with $\gamma=1$ and $\lambda>1$ the unique solution of the equation
\begin{equation*}
\int_0^1 \frac{x}{\lambda - x} \, \di \mu (x) = 1.
\end{equation*}
We now give an application of our result for the network with fitness distribution 
\begin{equation*}
	\mu(x,1)= \e^{1-(1-x)^{-\varrho}}, \quad  \text{for $x\in [0,1)$,}
\end{equation*}
where $0<\varrho<1$, see Example~\eqref{exa:f1} in Section~\ref{sec:f}.
We  estimate $\sigma_t$, as defined in Equation~\eqref{equ:sigma}. Using that 
\smash{$g(x)= m^{-1}(x) = 1 - (x+1)^{-\frac{1}{\varrho}},$} we have that, {\cec for all $t\geq 0$ large enough,}
$x = \lambda \sigma_t$ is the unique solution of 
\begin{equation*}
	(\log g)'(x)= \frac{1}{\lambda t + 1 - (x+1)},
\end{equation*}
which we can rewrite as 
$\lambda t + 1  = \varrho(x + 1)^{\frac{\varrho + 1}{\varrho}} + (1 - \varrho)(x + 1).$
From this we get
\begin{equation}
	\sigma_t =  x_0 t^\tetta + \mathcal{O} \Big(t^\frac{\varrho - 1}{\varrho + 1}\Big) 
	\quad{\cec\text{as }t\to\infty},
%= x_0 t^\tetta\Big(1 + \mathcal{O}\Big(t^{-\frac{1}{\varrho + 1}}\Big)\Big),
	\label{equ:sigma_approx}
\end{equation}
where $x_0 = \lambda^{-\frac{1}{\varrho+1}} \varrho^{-\tetta}$. By definition of $\varkappa$ in Assumption {\bf (A5.3)} we get %$\varkappa = \frac{\varrho + 1}{\varrho}$ (by Equation \eqref{equ:k}).\\
		\begin{equation*}
			\varkappa = \lim_{x \uparrow 1}  \frac{m''(x) m(x) x}{(m'(x))^2} = \lim_{x \uparrow 1}  \frac{(\varrho + 1) x \big(1 - (1-x)^\varrho \big)}{\varrho} = \frac{\varrho + 1}{\varrho}.
			%\label{equ:k}
		\end{equation*}
As an example we apply Corollary~\ref{cor:lim_fam}(i). 
Denoting by $a_4:= \varrho^{-\frac{\varrho}{\varrho + 1}} + \varrho^\frac{1}{\varrho + 1}$ and $a_5:=\frac{\varrho}{2(\varrho + 1)}$, we get the following distributional limit for the size of the largest family. Asymptotically as $t \rightarrow \infty$, 
			\begin{equation*}
				\e^{- \big(t  - \lambda^{-\frac{1}{1 + \varrho}}  a_4 t^\tetta + \frac{1}{\lambda} \big) - \frac1\lambda a_5 \log t +  T} \max_{n \in \N} Z_n(t) \Rightarrow W, 
			\end{equation*}
where $W$ is a Fr\'echet distributed random variable with shape parameter $\lambda$ and scale parameter $s$ given by \smash{$s^\lambda = \sqrt{\frac{2 \pi \varrho}{\varrho + 1}} \Gamma(\lambda+1).$} To get a result, which is independent of the continuous time embedding we look at the time $\tau_n$ when the $(n+1)$st vertex is introduced. The largest degree at this instance satisfies\footnote{We write $a_n\asymp b_n$ iff $a_n=\mathcal O(b_n)$ and $b_n = \mathcal O(a_n)$.}
%$$\max_{m\leq n} Z_m(\tau_n) \asymp n^{\frac1\lambda}
%\e^{ a_4 \tau_n^\tetta - a_5 \log \tau_n} ,$$
$$\max_{m\leq n} Z_m(\tau_n) 
\asymp n^{\frac1\lambda}
\e^{ \frac1\lambda a_4 (\log n)^\tetta - \frac1\lambda a_5 \log\log n}  ,$$
where the implied constants are positive random variables. %{\color{cyan}Check!}

\subsubsection{\bf Preferential attachment network of Dereich} \label{sec:Dereich} \smallskip

\noindent 
Dereich in~\cite{Dereich} defined an alternative preferential attachment model with fitness that can be studied without a Malthusian condition. In the model a new vertex is connected to each existing vertex
independently by a random number of edges, defining a multigraph.
\smallskip

Start with one vertex labelled one, with fitness $F_1$ drawn from $\mu$ and no edges. Denote the graph by $\mathcal{G}_1$. Given $\mathcal{G}_m$ with vertex set $\{1, ... ,m\}$ we build $\mathcal{G}_{m+1}$ by introducing the vertex labelled $m+1$, giving it fitness $F_{m+1}$  drawn from $\mu$ and connecting it independently to each vertex $n\in \{1, ... ,m\}$ by a random number $E_{n,m+1}$ of directed edges (from vertex $m+1$ to $n$), which is Poisson distributed with rate 
$$ r_{n,m} := \beta F_n  \frac{1+ \text{indegree of $n$ in $\mathcal{G}_m$}}{m},$$
where $0<\beta<1$ is a fixed parameter. 
%{\color{cyan}Seems that I can only do it under condition $\beta<1$.}
\smallskip

This model can be embedded into continuous space by letting 
$\tau_n=\frac1\lambda\sum_{i=1}^{n-1} \frac1i$, for $\lambda>0$,
%{\color{cyan} I do not get any use out of this parameter except sanity check.}
% was: $\lambda<1/\beta$ {\color{cyan}check!}, 
be the time when the 
$n$th vertex is introduced and defining
$Z_n(\tau_m)$, $m\geq n$ to be the indegree of vertex~$n$ prior to the establishment  of vertex~$m+1$, or in other words the number of edges pointing from vertices $n+1,\ldots,m$ to vertex $n$. {\peter Note that the indegree process $(Z_n(t) \colon t\geq \tau_n)$ has $Z_n(\tau_n)=0$ and it is actually the process 
$(1+Z_n(t) \colon t\geq \tau_n)$ that corresponds to the family sizes in our general framework. This is taken into account when we check below} that this model 
{\cec satisfies assumptions {\bf (A1-4)}} without any Malthusian condition for
{\peter $\gamma=\lambda\beta$.}  {\cec But we first show what kind of information can be obtained by applying our main results to this model.}
\smallskip

We look at the fitness $V(t)$ of the vertex $m\in\{1,\ldots,n-1\}$ with largest degree at the time $t=\frac1\lambda \log n + C + o(1)$ when the $n$th vertex is introduced (where $C$ denotes the Euler-Mascheroni constant) again in the case Gnedenko's distribution (Example (3) in Section~\ref{sec:f}). Recall that in this case
$g(x)=\frac{x}{1+x}$ and $\lambda\sigma_t= \sqrt{\lambda t+1}-1$. We denote by $S(t)$ the time of creation of this vertex;  by Corollary~\ref{cor:lim_fam}, we have \smash{$S(t) = \sigma_t + (W+o(1))\sqrt{\sigma_t/\lambda}$} in distribution when $t\uparrow\infty$, where $W$ is a centred Gaussian {\cec random variable} of variance $\nicefrac12$. Theorem~\ref{theo:theo_main} gives that, in distribution when $t\uparrow\infty$,
\begin{align*}
V(t) & = g\Big(\lambda \sigma_t + \sqrt{\lambda \sigma_t}\big(W+o(1)\big)\Big) 
+ {\mathcal O}\Big(g'\big(\lambda \sigma_t + \sqrt{\lambda \sigma_t}(W+o(1)) \big)\Big)\\
&= 1-\frac1{\lambda\sigma_t}+\frac{W+o(1)}{\sqrt{\lambda\sigma_t}}
= 1-\frac1{\sqrt{\lambda t}} + \frac{W+o(1)}{(\lambda t)^{\nicefrac14}},
\end{align*}
so that there is asymptotic normality for the fitness of the vertex of maximal degree. This is in  contrast to the result in Corollary~\ref{cor:23}(ii) for the case of $\mu$ in the maximum domain of attraction of the Weibull distribution, where $t(1-V(t))$ converges to a Gamma distribution. 

\bigskip
The rest of this section is devoted to the proof of {\bf (A1-4)} for the model of Dereich.
Assumption {\bf (A1)} is straightforward for the deterministic choice
$$\tau_n = \frac1\lambda \sum_{i=1}^{n-1} \frac 1i = \frac1\lambda \log n + \frac C \lambda +o(1),$$
where $C$ is the Euler-Mascheroni constant. To show that Assumption {\bf (A2)} is satisfied 
we introduce a coupling of the indegree processes $(Z_n(t) \colon t\ge 0)$ to independent Yule processes. For all $n\geq 1$, $u\geq 0$, we set
\[X_n(u)=Z_n\big(\sfrac{u}{F_n}+\tau_n\big).\]

\begin{prop}\label{prop:coupling}
There exists a coupling of the processes $(X_n(u) \colon u\geq 0)$ and 
a  sequence 
$(Y_n(u) \colon u\geq 0)$ of independent Yule processes with parameter $\gamma=\beta\lambda$ such that {\peter (see Equation~\eqref{equ:IA} for the definition of $\bor{I}(t)$)}, 
{\peter as $t\to \infty$,}
$$\sup_{n\in \bor{I}(t)} \pr\big( \sup_{u \ge t} \e^{-\gamma u} \big |1+X_n (u) - Y_n(u)  \big | \geq \varepsilon \big| (F_i)_i\big) 
\longrightarrow 0.$$
%for $\gamma:=\beta\lambda$. 
\end{prop}

%{\color{cyan} $X$ and $Y$ are started in zero. So actually, $1+Y$ is the Yule process.
%We can sort this later.}

\noindent
To prove this start with a sequence $(Y_n(u) \colon u\geq 0)$ of independent Yule processes with parameter $\gamma$.
For $m\geq n+1$ we take %{\peter $J_n(n)=0$ and }
%$$\tilde X_n(F_n(\tau_m-\tau_n)) = Y_n(F_n(\tau_{m-1}-\tau_n)) + J_n(m),$$
%where 
$$ J_n(m)=Y_n(F_n(\tau_{m}-\tau_n)) - Y_n(F_n(\tau_{m-1}-\tau_n)) .$$
We need the following lemma.
\begin{lem} 
Given $n$ there is a coupling of $J_n(m)$ and random variables $P_n(m)$, $m\geq n+1$, 
such that {\peter
\begin{itemize}
\item conditionally on $F_n=f\in(0,1)$  the random variable $P_n(n+1)$ is Poisson distributed
with parameter $\beta f \frac{1}{n+1}$, and \\[-5mm]
\item for $m\geq n+2$, conditionally on $F_n=f\in(0,1)$  and $\sum_{\ell=n+1}^{m-1} P_n(\ell)=k\in\{0, 1, \ldots\}$, the random variable $P_n(m)$ is Poisson distributed
with parameter $\beta f \frac{1+k}m$,
\end{itemize}}
and
$\displaystyle\sup_{n\in \bor{I}(t)} \pr\big( J_n(m)\not=P_n(m) \mbox{ for some $m\geq n+1$}\big) \longrightarrow 0, \text{ as $t \rightarrow \infty$.} $
% PM not: almost surely. NOTHING IS RANDOM HERE.
\end{lem}

% Then for m = n+1 we have $k = 0$, which means that the number of offspring at the first birth event at time \tau_{n+1}, P_n(n+1) is Poisson distributed with parameter $\beta f \frac{1}{n+1}$, which is consistent with there being only one particle that could give birth. 

\begin{proof}
We abbreviate $Y_n^*(t)=Y_n(F_n(t-\tau_n))$ and note that, {\cec conditionally on $F_n$,}
$(Y_n^*(t) \colon t\geq {\cec \tau_n})$ is a continuous time Galton-Watson process starting with one individual at time $\tau_n$ and individuals performing binary branching at rate $\gamma F_n$. The coupling is now
performed in two steps.
\begin{itemize}
\item[$(a)$]  For $m \ge n+1$, we let $\cE_{n,m}$ be the event that all of the individuals alive at time $\tau_{m-1}$ have at most one descendant in the interval $[\tau_{m-1}, \tau_{m})$. This means that an individual existing at time $\tau_{m-1}$ can only give birth to at most one individual, which in turn does not reproduce before $\tau_{m}$. 
%We say that an individual ``jumps twice" if it has two or more descendants in a given interval. This includes the direct offspring and the individual's ``grandchildren". Hence $\cE^c_{n,m}$ is the event that at least one individual jumps twice in the interval~$[\tau_{m}, \tau_{m+1})$. 
Denoting 
$$\cE_n(t) = \bigcap_{\heap{m\geq n+1}{\tau_{m+1}<t}}  \cE_{n,m},$$ 
we show that
	\begin{equation*}
		\sup_{n\in \bor{I}(t)}\pr (\cE_n^{\rm c}(t)) \rightarrow 0, \quad \text{ as $t \rightarrow \infty$.} 
	\end{equation*}
\item[$(b)$] 
{\cec For all $m\geq n+1$,} there are random variables  $J_n^*(m)$, which, conditionally on $F_n=f$ 
and $\sum_{l={\peter n+1}}^{m-1} J_n^*(l)=k$, are binomially distributed with parameters $1+k$ and 
$\beta f/m$, such that $J_n^*(m)=J_n(m)$ on $\cE_{n,n+1}\cap \cdots \cap \cE_{n,m}$. 
We can couple $J_n^*(m)$ to random variables $P_n(m)$, which given
$F_n=f{\cec \in(0,1)}$ and $\sum_{\ell={\peter n+1}}^{m-1} P_n(\ell)=k{\cec \in\{0, 1, \ldots\}}$ are Poisson distributed
with parameter $\beta f \frac{1+k}m$ such that 
$$\sup_{n\in \bor{I}(t)} \pr\big( J_n^*(m)\not=P_n(m) \mbox{ for some $m\geq n+1$}\big) \longrightarrow 0, 
\quad \text{ as $t \rightarrow \infty$.} $$
\end{itemize}
It is clear that the lemma follows from claims $(a)$ and $(b)$.
\medskip\pagebreak[3]

% % % % % % % % % % % % % % % % % % % % %

We now prove $(a)$. Fix $n \in \bor{I}(t)$ and let $m > n$. Denote by $\eta = \gamma F_n$  and by
$W_\theta$ an independent random variable, exponentially distributed with parameter $\theta$. Recall that in a Yule process of rate $\eta$ each particle gives birth to one offspring after an exponentially distributed waiting time with rate $\eta$, independently of everything else. Thus the {\peter conditional} probability that a fixed particle has at least one offspring in the interval $[\tau_{m}, \tau_{m+1})$ is equal to 
%\begin{eqnarray*}
$$	\pr (W_\eta \le  \tau_{m+1} - \tau_{m} {\peter | (F_i)_i}) = 1 - \e^{-\frac {\eta}{\lambda m}} \le \sfrac{\beta F_n}{m} .$$ 
%\le \frac{\beta F_n}, $$ ?????
%	\label{equ:ystar}
%\end{eqnarray*}
%where we have used that $\tau_{m+1} - \tau_{m} = \frac 1 \lambda( \sum_{i=1}^{m+1} \frac 1i - \sum_{i=1}^{m} \frac 1i) = \frac{1}{\lambda(m+1)}$ and $1 - e^{-x} \le x$ for all $x$ in $\R$. 
Furthermore, the probability of a given particle having at least 2 descendants in $[\tau_{m}, \tau_{m+1})$ is equal to
\begin{eqnarray}
	\label{equ:W}
	\pr \big(W_\eta + W_{2\eta} \le \tau_{m+1} - \tau_{m}  {\peter | (F_i)_i}\big) & = &  1 + \e^{-\frac {2\eta}{\lambda m}} - 2 \e^{-\frac {\eta}{ \lambda m}} \\
	& = & (1 - \e^{-\frac{\beta F_n}{m}})^2  \le \sfrac{(\beta F_n)^2}{m^2}, \nonumber
\end{eqnarray}
where $W_{2\eta}$ is the minimum of two independent exponentially distributed waiting times
with rate~$\eta$. Using the law of total probability we can express the probability that at least one particle of $(Y_n^*(t) \colon t\geq 0)$ at time $\tau_m$ has at least 2 descendants in the interval $[\tau_{m}, \tau_{m+1})$,
\begin{align}
	\mathbb P(\cE^{\rm c}_{n,m+1} {\cec | (F_i)_i}) 
	= \sum_{k=1}^\infty \mathbb P\big(\cE^c_{n,m+1} |Y^*_n(\tau_{m})=k , {\peter (F_i)_i}\big) \pr(Y^*_n(\tau_{m})=k) 
%	& \le & \sum_{k=1}^\infty \frac{(\beta F_n)^2}{m^2} \: k \pr(Y^*_n(\tau_{m})=k) =
\leq \sfrac{(\beta F_n)^2}{m^2} \E[Y^*_n(\tau_{m}){\peter | (F_i)_i}]. \label{equ:s}
\end{align}	
By Lemma \ref{lem:Yule}(a), we have
\begin{equation}
	\E\big[Y^*_n(\tau_{m}){\cec | (F_n)_n}\big] = \E\big[{Y}_n(F_n(\tau_{m} - \tau_n)){\peter | (F_i)_i}\big] = \e^{\beta \lambda F_n(\tau_{m}-\tau_n)} \asymp \big(\sfrac mn\big)^{\beta F_n},
	\label{equ:E}
\end{equation}	
where we have used the fact that $\tau_m - \tau_n = \frac 1 \lambda \log (\frac{m}{n}) + \mathcal O(1)$ almost surely for all $m\geq n$ and $n$ large {\cec (see~\cite[Theorem~III.9.3]{Ney})}.
We now look at $n$ such that {$n\in \bor{I}(t)$ or, equivalently,} 
$$\e^{\lambda(\sigma_t - \kappa \sqrt{\sigma_t})} \le n \le \e^{\lambda(\sigma_t + \kappa \sqrt{\sigma_t})}.$$ 
Putting this together with Equations~\eqref{equ:s} and~\eqref{equ:E} we get
\begin{eqnarray*}
	\pr(\cE^c_n(t){\peter |(F_i)_i})&  \le & \sum_{m=n}^{\infty} \pr (\cE^c_{n,m+1}{\peter |(F_i)_i}) \le {\rm const. } \sum_{m{\cec \geq} \e^{\lambda(\sigma_t-\kappa \sqrt{\sigma_t})}}
	\frac{(\beta F_n)^2}{m^2}  \Big(\frac mn\Big)^{\beta F_n} \\ & \le &  {\rm const. } \frac{(\beta F_n)^2}{\e^{\beta \lambda F_n (\sigma_t-\kappa \sqrt{\sigma_t})}} \sum_{m=\e^{\lambda(\sigma_t-\kappa \sqrt{\sigma_t})}}^{\infty} m^{\beta F_n - 2} \\
& \le  &  {\rm const. } \frac{(\beta F_n)^2}{\e^{\beta \lambda F_n(\sigma_t-\kappa \sqrt{\sigma_t})}} \int_{\e^{\lambda(\sigma_t-\kappa \sqrt{\sigma_t})}}^{\infty} x^{\beta F_n - 2} \:\dx 
 \le  {\rm const. }\frac{(\beta F_n)^2}{1 - \beta F_n} \e^{-\lambda(\sigma_t-\kappa \sqrt{\sigma_t})}, 
\end{eqnarray*}
which goes to zero, as $t \rightarrow \infty$. This completes the proof of $(a)$.
%{\color{cyan}Here we have used that $\beta<1$, as otherwise the integral diverges at infinity.}
\smallskip

%%%%%%%%%%%%%%%%%%%%%%%%%%%%%%%%%%%%%%%%%%%%%

\pagebreak[3]

%%%%%%%%%%%%%%%%%%%%%%%%%%%%%%%%%%%%%%%%%%%%%

%\begin{lem} \label{lem:coupling}
%Let $\beta <1$ and $(\hat{Y}_n)_{n\in I(t)}$ be a sequence of i.i.d.\:Yule processes of rate $\beta$. Given $F_n$ for all $n \in I(t)$, let $Y_n^*(t) = \hat{Y}_n(F_n(t - \tau_n))$. On $\cE_n(t)$, there exists a coupling of $(Y^*_n)_{n\in I(t)}$ and $(Z_n)_{n\in I(t)}$ such that 
%\begin{equation*}
%	Y^*_n(\tau_m) = Z_n(\tau_{m}) + 1 \quad \text{almost surely $\forall n \in \bor{I}(t)$ and $\forall m \ge n $.}
%\end{equation*}
%\end{lem}

\noindent 
To show $(b)$ fix $n \in \bor{I}(t)$ and let $m \ge n+1$. 
%For notational convenience we define $\eta = \beta F_n$, and 
Note that the existence of $J_n^*(m)$ binomially distributed with parameters $k+1$ 
and $\beta f/{m}$ such that $J_n^*(m)=J_n(m)$ on $\cE_{n,n+1}\cap \cdots \cap \cE_{n,m}$
is easy because on this event there are $k+1$ individuals alive at time $\tau_{m-1}$ and each independently produces offspring with probability $\beta f/{m}$.
\smallskip

Moreover, by Lemma \ref{lem:Yule}(c), we have
$Y^*_n(\tau_{m}) \sim \xi_n \e^{\gamma F_n(\tau_{m}- \tau_n)} =\mathcal{O} \big((\frac{m}{n})^{\beta F_n}{\xi_n} \big)$ almost surely,when $m\uparrow\infty$, {\cec where $(\xi_n)_{n\geq 1}$ is a sequence of i.i.d.\ standard exponential random variables (because the martingale limit of a Yule process is a standard exponential, see, e.g., \cite[Section~4]{Kendall})}. 
Applying Theorem~9 of \cite[ch.4.12]{Probability}  %page 129
%\cite[ch.4.12, Theorem (9)]{Probability}, 
the conditional total variation distance between $J_n^*(m)$ and $P_n(m)$ {\peter conditionally on $(F_i)_i$ and $Y^*_n(\tau_{m-1})$} satisfies
\begin{equation*}
	d_{\text{TV}} (J_n^*(m), P_n(m)) 
	\le 2 \sum_{i = 1}^{Y^*_n(\tau_{m-1})} \Big(\frac{\beta F_n}{m}\Big)^2
	= \frac{2(\beta F_n)^2 Y^*_n(\tau_{m-1})}{m^2} 
	= \mathcal{O}\Big(\frac{(\beta F_n)^2\xi_n}{m^{2 - \beta F_n}n^{\beta F_n}}\Big),
\end{equation*}
{\cec almost surely} when $m\uparrow\infty$, where the $\mathcal O$-term {\cec is uniform in}~$n$.
This implies that {\cec (see, e.g.\ \cite[Exercise~4.12.5]{Probability})} there exists a coupling of $J_n^*(m)$ and $P_n(m)$, $m\geq n+1$, such that
$$\pr\big( J_n^*(m)\not=P_n(m) \mbox{ for some $m\geq n+1$}{\peter |(F_i)_i}\big)
\leq {\rm const.}\,\xi_n \sum_{m=n}^\infty \frac{(\beta F_n)^2}{m^{2 - \beta F_n}n^{\beta F_n}}
\leq\frac{{\rm const.}\, \xi_n}n,$$
using again that $\beta<1$. 
This implies that
\[\sup_{n\in \bor{I}(t)}\pr\big( J_n^*(m)\not=P_n(m) \mbox{ for some $m\geq n+1$}{\peter |(F_i)_i}\big)
\leq {\rm const.}\, \sup_{n\in \bor{I}(t)} \{\nicefrac{\xi_n}n\}
\leq {\rm const.}\, \frac{\sup_{n\in \bor{I}(t)} \xi_n}{\inf(\bor{I}(t))},\]
{\cec where $\inf(\bor{I}(t))$ is the smallest element of the set $\bor{I}(t)$, i.e.\ $\lceil\exp(\lambda(\sigma_t-\kappa\sqrt{\sigma_t}))\rceil$ (see Equation~\eqref{equ:IA}).}
Note that the cardinality of $\bor{I}(t)$ is less than or equal to $2\kappa \sqrt{\sigma_t}$, and the $\xi_n$'s are i.i.d.\ standard exponential random variables. Thus, by extreme value theory {\cec (see, e.g., \cite[Equation~(1.1.2)]{Resnick})}, we get that, in distribution when $t\uparrow\infty$,
\[\sup_{n\in \bor{I}(t)}\xi_n = \log |\bor{I}(t)| + \mathcal O(1)
= \log(\sigma_t)/2 + \mathcal O(1).\]
{By definition of $\bor{I}(t)$, we also have that $\inf (\bor{I}(t)) = \sigma_t - \kappa \sqrt{\sigma_t}$}, thus implying that
$${\peter \sup_{n\in \bor{I}(t)}} \pr\big( J_n^*(m)\not=P_n(m) \mbox{ for some $m\geq n+1$}{\peter |(F_i)_i}\big)
\to 0\text{ when }t\uparrow\infty,$$
which concludes the proof.
\end{proof}

\noindent
To complete the proof of Proposition~\ref{prop:coupling} we define
$$X^*_n(F_n(t-\tau_n)) =\sum_{k=n+1}^m P_n(k), \qquad  \mbox{ for all }{\cec m\geq n+1 \text{ and }} \tau_m\leq t < \tau_{m+1}, $$
and note that, for all $n\geq 1$, the process $(X^*_n(t)\colon t\geq \tau_n)$ {\peter has the same distribution as $(X_n(t)\colon t\geq \tau_n)$.}
Moreover,
\begin{align*}
\pr\big( 1+X_n(F_n(\tau_m-\tau_n)) = Y_n(F_n(\tau_m-\tau_n)) & \mbox{ for all } m\geq n+1\big)\\
& =1-\pr\big( J_n(m)\not=P_n(m) \mbox{ for some $m\geq n+1$}\big)  
\end{align*}
because
$X_n(F_n(\tau_m-\tau_n))=\sum_{k=n+1}^m P_n(k)$ and $Y_n(F_n(\tau_m-\tau_n))=1+\sum_{k=n+1}^m J_n(k)$. Suppose now that $\tau_m\leq t < \tau_{m+1}$ and  $1+X_n(F_n(\tau_m-\tau_n)) = Y_n(F_n(\tau_m-\tau_n))$. Then, almost surely, as $m\uparrow\infty$,
\begin{align*}
|1+X_n(F_n(t-\tau_n)) - Y_n(F_n(t-\tau_n))|
& = |Y_n(F_n(\tau_m-\tau_n)) - Y_n(F_n(t-\tau_n))|\\ &
=(\xi_n+o(1)) \e^{\gamma F_n(t- \tau_n)}- \xi_n \e^{\gamma F_n(\tau_m- \tau_n)}\\
&\leq 
(\xi_n+o(1)) \e^{\gamma F_n(t- \tau_n)} \big( 1- \e^{-\beta F_n/m} \big),
\end{align*}
and hence
\begin{align*}
\sup_{n\in \bor{I}(t)}  & \pr\big( \sup_{u\geq t} \e^{-\gamma u} |1+X_n(u) - Y_n(u)| \geq \varepsilon {\peter |(F_i)_i}\big)
\\ & \leq \sup_{n\in \bor{I}(t)} \pr\Big( \xi_n( 1- \e^{-\beta F_n/n}) \geq \nicefrac\varepsilon2{\peter |(F_i)_i}\Big)  
+ \sup_{n\in \bor{I}(t)} \pr\big( J_n(m)\not=P_n(m) \mbox{ for some $m\geq n+1$}{\peter |(F_i)_i}\big) \\ & \longrightarrow 0 \mbox{ almost surely as $t\uparrow\infty$.}
\end{align*}
This completes the proof of Proposition~\ref{prop:coupling} and hence of Assumption~{\bf (A2)}.
Further, from Lemma~\ref{lem:Yule}(c) we see that Assumption {\bf (A3)} holds.
\medskip
 
Finally, to prove Assumption {\bf (A4)} we fix the fitnesses $(F_n)_{n\geq 1}$ and work conditionally {\cec on this sequence of random variables}. Note that, by definition, the jump of $(X_n(t) \colon t\geq \tau_n)$ at time $t=F_n(\tau_{m} - \tau_{n})$ given $X_n(F_n(\tau_{m-1}-\tau_n)) = k$ is Poisson distributed with parameter
$\beta F_n \frac{1+k}{m}$. Hence the processes $(M^{\ssup n}_m \colon m\geq n)$ given by
$$M^{\ssup n}_m:=\big(1+X_n(F_n(\tau_{m}-\tau_n)) \big)\,\prod_{\ell=n+1}^{m} \big(1+\sfrac{\beta F_n}{\ell} \big)^{-1}$$
are martingales, {\cec i.e.\ for all $n\geq 1$, for all $m\geq n+1$, $\mathbb E[M^{\ssup n}_{m+1}|(F_i)_i, M^{\ssup n}_{m}]=M^{\ssup n}_m$}. The scaling factor satisfies
$$f_m^{-1}:=\prod_{\ell=n+1}^{m} \big(1+\sfrac{\beta F_n}{\ell} \big)
=\e^{\gamma (F_n(\tau_{m}-\tau_n))}(1+o(1)).$$
Hence almost sure limits $\displaystyle M^{\ssup n}_\infty=\lim_{m\to\infty} M^{\ssup n}_m$ exist and Doob's submartingale inequality yields
\begin{align*}
\pr\big(\max_{u\geq 0} X_n(u) \e^{-\gamma u} \ge x\big)
&\leq \pr\big(\max_{m\geq n+1} M^{\ssup n}_m \ge \nicefrac{x}{2}{\cec (1+o(1))}{\cec | (F_i)_i}\big)\\
&\leq \E\big[\max_{m\geq n+1} \e^{2{\cec \varpi} M^{\ssup n}_m} {\cec | (F_i)_i}\big] \,  \e^{-{\cec \varpi}  x}
 \leq \E\big[\e^{2{\cec \varpi}  M^{\ssup n}_\infty}{\cec | (F_i)_i} \big] \,  \e^{-{\cec \varpi}  x},
\end{align*}
{\cec for all $\varpi>0$}.
It remains to show that there exists ${\cec \varpi} >0$ such that $ \E[\e^{{\cec \varpi}  M^{\ssup n}_\infty}] <\infty$
or, using Fatou's lemma, that \smash{$\E[\exp({\cec \varpi}  M^{\ssup n}_m)]$} remains bounded. 
Using the generating function for Poisson variables we~get
$$\mathbb E\big[\e^{{\cec \varpi}  M^{\ssup n}_m} \big| {\peter (F_i)_i},  X_n(F_n(\tau_{m-1}-\tau_n)) =k\big]
= \exp\big( (1+k) (\eta f_m + \beta F_n \sfrac1m (\e^{{\cec \varpi}  f_m}-1))\big).$$
Hence, using that $\e^{{\cec \varpi}  f_m}-1\leq {\cec \varpi}  f_m +C {\cec \varpi}^2 f_m^2$ for some constant $C>0$, we get
$$\E\big[\e^{{\cec \varpi}  M^{\ssup n}_m}{\cec | (F_i)_i}\big]
\leq \E\big[\e^{({\cec \varpi} +C{\cec \varpi}^2 \frac1m f_m) %(\frac{n}m)^{\beta F_n}) 
M^{\ssup n}_{m-1}}\big],$$ 
and iterating this we get an upper bound of $\e^{a_{m-n}}$ for the recursion $a_0={\cec \varpi} $
and $$a_{i+1}=a_i+Ca_i^2 \sfrac1{m-i} f_{m-i}, \quad \mbox{  for $i\geq 0$. }$$
As $f_\ell\asymp (n/\ell)^{\beta F_n}$ {\cec almost surely when $\ell\geq n+1\to\infty$}, 
there exists {\cec an almost-surely finite $(F_i)_i$-measurable random variable} $A$ such that
$$\prod_{\ell=n+1}^m \big(1+C\sfrac1{\ell} f_\ell\big)  \leq A \quad \mbox{ for all $m\geq n$ and $n$.}$$
Hence $(a_{m-n} \colon m\geq n)$ is bounded by one if $0<{\cec \varpi} <\nicefrac1{A}$. 
This completes the proof of {\bf (A4)}.

\subsection{Random permutations with random cycle weights} \label{sec:cycles}

%%%%%%%%%%%%%%%%%%%%%%%%%%%%%%%%%%%%%%%%%%	

%\noindent  
%Suppose we are given a permutation $\sigma$ of the indices $\{1,\ldots, n\}$ and, for each of the $k$ cycles of the permutation, a weight $W_j$, $j=1,\ldots,k$. We create a permutation $\sigma'$ of the indices $\{1,\ldots, n+1\}$ from this as follows
%\begin{itemize}
%\item with probability $\beta\in(0,1)$ the new index $n+1$ is mapped onto itself, creating a new cycle. This cycle is given a weight $W_{k+1}$ sampled, independently of everything else, from $\mu$;
%\item otherwise, pick one of the indices $\{1,\ldots, n\}$, each chosen with a probability proportional to the weight of its cycle, and insert the new index into its cycle so that if index $m$ is chosen we have $\sigma'(m)=n+1, \sigma'(n+1)=\sigma(m)$ and
%$\sigma'(i)=\sigma(i)$ for all $i\not=m,n+1$. 
%\end{itemize}
\noindent  
Let $\theta\geq 0$ be a fixed parameter and suppose we are given a permutation $\sigma$ of the indices $\{1,\ldots, n\}$ and, for each of the $k$ cycles of the permutation, a weight $W_j$, $j=1,\ldots,k$. Denote the length of the cycles by $Z_1,\ldots,Z_k$. We create a permutation $\sigma'$ of the indices $\{1,\ldots, n+1\}$ from this as follows
\begin{itemize}
\item either pick one of the indices $m\in\{1,\ldots,n\}$ from the $j$th cycle with probability $\frac{W_j}{n+\theta}$ and  insert the new index into its cycle so that we have $\sigma'(m)=n+1, \sigma'(n+1)=\sigma(m)$ and
$\sigma'(i)=\sigma(i)$ for all $i\not=m,n+1$;
\item with the remaining probability $1-\frac{\sum_{j=1}^k Z_j W_j}{n+\theta}$
%, where $Z_j$ denotes the size  of the $j$th cycle, 
the new index $n+1$ is mapped onto itself, creating a new cycle {\peter of length one.} This cycle is given a weight $W_{k+1}$ sampled, independently of everything else, from~$\mu$. 
\end{itemize}
The resulting process $(\sigma_n)$ can be seen as a \emph{disordered chinese restaurant process}. The idea is that the cycles correspond to tables and new customers either join a table with a probability proportional to both the weight and the number of seats on the table, or sit at a new table.  In the original chinese restaurant process customers chose to sit on a table with a probability proportional to the number of seats and the probability of introducing a new table is $\frac\theta{n+\theta}$, see~\cite[p. 92]{Exchangeability}. This corresponds to all weights being equal to one in our scenario. We briefly mention that this model differs from the model of Betz, Ueltschi and Velenik
on random permutations with cycle weights, as in their case the weight of a cycle is not random and instead depends on the size of the cycle, see \cite{betz2011}.
\medskip

%Our analysis applies to this model for arbitrary parameter $\theta\geq 0$. 
Let us show that this model falls into our framework of competing growth processes and satisfies Assumptions {\bf (A1-4)}. Key is again an embedding of the process in continuous time such that $T_n$ is the time when the $n$th customer enters the restaurant. We let $T_1=0$ and define $T_{n+1}$, $n\in\N$, inductively as follows. At time $T_{n}$ we start $n+1$ independent exponential clocks, one clock of parameter one for each of the $n$ customers seated in the restaurant and one additional clock of parameter $\theta$ for the creation of additional tables. We let $T_{n+1}$ be the time when the first of these clocks rings. 
\begin{itemize}
	\item If it is the clock corresponding to customer~$m$ sitting at table $j$ we toss a coin with success probability~$W_j$.
	\begin{itemize}
		\item If there is a success the $(n+1)$st customer joins this table, resp.\ in the language of random permutations
		the element $n+1$ is inserted in this cycle between elements $m$ and $\sigma_n(m)$,
		\item if there is no success the $(n+1)$st customer seats at  a new table which, if it is the $(k+1)$st occupied table, gets weight $W_{k+1}$. 
	\end{itemize}
\item If it is the clock for the creation of additional tables, the $(n+1)$st customer also sits at a new table which, if it is the $(k+1)$st occupied table, gets weight $W_{k+1}$.
\end{itemize}
Suppose $W_1, W_2, \ldots$ are given. We note that, as required, the overall probability that a new table is created at time $T_{n+1}$ is 
$$\frac{\sum_{j=1}^k Z_j(T_n) (1-W_j) + \theta}{n+\theta}= 1 - \frac{\sum_{j=1}^k Z_j(T_n) W_j}{n+\theta},$$
where $Z_j(T_n)$ is the number of occupants at the $j$th table at time $T_n$, and the probability that the $(n+1)$st customer joins the $j$th table is $Z_j(T_n)W_j/(n+\theta)$. Looking at the $j$th table, we let $\tau_j$ be the time when it is first occupied. If at time $t$ this table is occupied by $m$ customers
the rate at which new customers join this table is $m W_j$, independent of the occupancy of other tables. The processes $(Z_j(t+\tau_j) \colon t\geq 0)$ are therefore independent Yule processes with rate $W_j$. Hence Assumptions {\bf (A2-4)} are satisfied for $\gamma=1$ and where $X_n(u)=Y_n(u)$, $u\geq 0$, are given by $Z_n(t)=X_n(W_n(t-\tau_n))$.% 
\medskip%

Finally, to check Assumption {\bf (A1)} we note that the process of introduction of new tables is a general branching process with immigration. The immigration process corresponds to the creation of the additional tables, which is a homogeneous Poisson process with rate $\theta$. The point process of creation of tables by unsuccessful coin tossing is a Cox process $(\Pi(t) \colon t\geq 0)$, i.e.\ a Poisson process with random intensity. Its intensity is given by $(1-W)Y(t) \, \mathrm dt$ where $W$ has distribution $\mu$ and given $W$ the process $(Y(t) \colon t\geq 0)$ is a Yule process with parameter $W$.
The relevant results for general branching processes can be found in \cite{Nerman} with the case of branching processes with immigration treated in \cite{Olofsson}. The crucial assumption is the existence of a Malthuisan parameter $\alpha\geq 0$ such that
$$1= \int \e^{-\alpha t} \, \E \Pi(\mathrm dt)  = \int  \int_0^\infty (1-w) \e^{-\alpha t} \e^{wt} \, dt \, \mu(\mathrm dw) = \int \frac{1-w}{\alpha-w} \,\mu(\mathrm dw),$$
which is always satisfied for $\alpha=1$. As above {\peter it is a routine exercise to check the $x \log x$ condition on $\int \e^{-t} \Pi(\mathrm dt)$}. 
We obtain from \cite[Theorem 5.4]{Nerman}
for general branching processes without immigration (our case $\theta=0$) and modifications described in \cite[Theorem 4.2]{Olofsson} for the general case (stated there only for convergence in $L^1$) that there exists a positive random variable {\peter $M_\theta$ such that the total number $M(t)$ of tables which have been occupied by time $t$ satisfies
$$\e^{-t} M(t) \longrightarrow M_\theta \quad \mbox{ almost surely,}$$
from which we infer that $\tau_n= \log n - \log M_\theta + o(1)$, which is Assumption~{\bf (A1)} with $\lambda=1$.}

%The key argument is again to find a suitable embedding into continuous-time: we find one such that
%$$\tau_n= \log n + T_\theta + o(1),$$
%where $T_\theta$ is a random variable depending on the parameter $\theta\geq0$.
%The size $Z_n(t)$ of the $n$th cycle at time $t$ is such that $(Z_n(t+\tau_n) \colon t\ge 0)$ are independent Yule processes with parameter $F_n$, so that the key parameters in our assumptions are $\lambda=\gamma=1$.
{\cec We now give an example of a result that follows from our main technical result.} We look at the ratio $R(t)$ of the size of the largest and second largest cycle in the permutation at time $t$. {\peter We have
$$\id_{R(t) \geq x} = \int \id_{\Gamma_t([-\infty, \infty]\times [-\infty, \infty]\times (z,\infty))=0}
 \id_{\Gamma_t([-\infty, \infty]\times [-\infty, \infty]\times (z/x,z))=0}  \, \di \Gamma_t(s,f,z).$$}%
If $\mu$ satisfies Assumption {\bf (A5)}, then, by Theorem~\ref{theo:theo_main}, we hence have, for $x>1$,
{\peter with $N$ a Poisson point process with intensity measure~$\zeta$,}
\begin{align*}
\lim_{t\to \infty} \pr\big( R(t) \geq x\big)
& = {\peter \E \int \id_{N([-\infty, \infty]\times [-\infty, \infty]\times (z,\infty))=0}
 \id_{N([-\infty, \infty]\times [-\infty, \infty]\times (z/x,z))=0}  \, N(\di s \: \df \: \di z)} \\
& = {\peter \int} %_{(-\infty, \infty)\times (-\infty, \infty)\times (0,\infty)}}
\exp\big(- \zeta\big((-\infty,\infty) \times (-\infty,\infty)  \times (z/x,\infty)\big)\big)
\, \zeta(\di s \: \df \: \di z).
\end{align*}
Using that $\nu(x)=\e^{-x}$ and $a_3=1$ in the first equality {(similar as in \eqref{maxintegral} below)} and the change of variable $v = f - \log y$ in the second, we get that
\begin{align*}
\zeta\big((-\infty,\infty) \times (-\infty,\infty)  \times (z/x,\infty)\big)\big) 
& = {\cec \int_{-\infty}^{\infty}\int_{-\infty}^{\infty}} \di s \: \df  \, \e^{s^2 a_2 - 2f} \int_{z/x}^\infty \e^{-y\e^{s^2 a_2-f }}\, \mathrm dy\\
& = {\cec \int_{-\infty}^{\infty}\int_{-\infty}^{\infty}}  \di s\: \di v \, \e^{s^2 a_2 - 2v} \e^{-\e^{s^2 a_2 - v}} \int_{z/x}^{\infty} y^{-2} \,\mathrm dy
= a_5 \,\frac{x}{z},
\end{align*}
where $a_5$ is a positive constant.
Hence, substituting $f$ by $f+\log x$ in the final step,
\begin{align*}
\lim_{t\to \infty} \pr\big( R(t) \geq x\big) &
= {\cec \int_{-\infty}^{\infty}\int_{-\infty}^{\infty}}\di s \: \df  \:  \int_0^\infty \dz \,  \e^{-f} \e^{s^2 a_2 - f } \e^{-z(\e^{s^2 a_2 -f } )- a_5 \frac{x}z}\\
&
= {\cec \int_{-\infty}^{\infty}\int_{-\infty}^{\infty}} \di s \: \df  \: \int_0^\infty \dw \,  \e^{-f}  \e^{-w - a_5 \frac{1}w \e^{s^2 a_2 - f +\log x}}
= \frac1x.
\end{align*}
Similarly, if $\mu$ satisfies the assumptions (B.5),  we have
$\zeta\big((-\infty,\infty) \times (0,\infty)  \times (z/x,\infty)\big)\big) 
= a_6 \,\frac{x}{z} $, and hence by Theorem~1,
\begin{align*}
\lim_{t\to \infty} \pr\big( R(t) \geq x\big) &
= {\cec\int_{-\infty}^{\infty}} \di s \:  \int_0^\infty \df  \:  \int_0^\infty \dz \,  \alpha f^{\alpha-1}\e^{2s+ f } \e^{-z \e^{s+f } - a_6 \frac{x}z}\\
&= {\cec\int_{-\infty}^{\infty}} \di s \: \int_0^\infty   \df  \:  \int_0^\infty \dz \,  {\peter x} \alpha f^{\alpha-1}\e^{2s+ f } \e^{-z \e^{s+f +\log x} - a_6 \frac{1}z} = \frac1{x}.
	\end{align*}
substituting $s$ by {\peter $s-\log x$} in the final step. Note that this is in contrast to the case without disorder
where the cycles have macroscopic size and the distribution of the asymptotic ratio is given by the ratio of the two largest elements in the Poisson-Dirichlet distribution. 
\bigskip

The remainder of the paper is devoted to the proofs of Theorem~\ref{theo:theo_main} and Corollary~\ref{cor:lim_fam} and is structured as follows. In Section~\ref{sec:local_convergence} we look at the Poisson limit theorem given in Theorem~\ref{theo:theo_main}, but first in a space without compactifications. After some preparations we prove in Section~\ref{sec:convergence} a basic form of the limit theorem, see Proposition~\ref{prop:cv_gamma2}. This is derived from an approximation which corresponds to a classical Poisson convergence result for extremes in the first two components and an independent third component. In Section~\ref{sec:proof_local} a further approximation turns the basic form into the original form of the Poisson limit theorem, the crucial difference being that the scaling of the third component becomes independent of the birth rank $n$ of the family. Section~\ref{sec:compactification} is devoted to the compactification of the space, effectively showing that the points suppressed by the scalings do not provide the largest families. These points are either born too late 
(Section~\ref{sec:young}) or not fit enough (Section~\ref{sec:unfit}). In Section~\ref{sec:old} we show that there are no points outside our 
scaling window that are competitive in age and fitness. The proof of Theorem~\ref{theo:theo_main} is completed in Section~\ref{sec:main_proof} and the proof of Corollary~\ref{cor:lim_fam}, which crucially uses the compactification, in Section~\ref{sec:proof_corollary}.

\section{Local convergence of point processes} \label{sec:local_convergence}

In this section we prove {\cec a} convergence result for the point processes $(\Gamma_t)$ 
%and its approximations 
in a space without compactification. The strengthening of the results by compactification will follow in the next section. We begin by noting some preliminary results on the fitness distribution.

\subsection{Preliminaries on the fitness distribution} \label{sec:preliminaries}
First of all we show the existence and uniqueness of $\sigma_t$ as defined in Equation \eqref{equ:sigma}:
\begin{lem}
For all $t$ large enough, there exists a unique $\hat\sigma_t \in [0, t{\cec )}$, such that 
	\begin{equation*}
		(\log g)'(\lambda \hat\sigma_t) = \frac{1}{\lambda( t-\hat\sigma_t)}.
	\end{equation*}
Furthermore, as $t\uparrow\infty$, we have $\hat\sigma_t \rightarrow \infty$ {\cec (and thus, for all $t$ large enough $\sigma_t = \hat\sigma_t$)} and $\frac{\hat\sigma_t}{t} \to 0$.
	\label{lem:sigma}
\end{lem}
\begin{proof}
	{\cec For all $t\geq 0$, for all $x\in[0,\lambda t)$, we set}
	\begin{equation*}
		F(x) := (\log g)'(x) - \frac{1}{\lambda t - x}, 	
	\end{equation*}
so $F$ is continuous on $(0,\lambda t )$, {\cec since, by Assumption {\bf (A5)}, $m$, and thus $g$ are continuous and non-zero on, respectively $(0,1)$ and $(0,\infty)$}. Since {\cec $g = m^{-1} : [0,\infty)\to [0,1)$} and $g(0) = 0$, we have 
{\cec \[\lim_{x \downarrow 0 } (\log g)'(x) =\lim_{x \downarrow 0 } \frac{g'(x)}{g(x)}
= \lim_{x\downarrow 0}\frac1{xm'(x)} = \infty,\]
because $m'(0)<\infty$, since, by Assumption {\bf (A5)}, $m$ is differentiable on $[0,1)$.}
Therefore we get
	\begin{equation*}
		\lim_{x \downarrow 0} F(x) = \infty, 
\quad\text{ and }\quad
		\lim_{x \uparrow \lambda t }F(x) = -\infty.
	\end{equation*}
Hence by continuity of $F$, there exists $x \in (0, \lambda t)$ such that $F(x) = 0$. Furthermore such $x$ is unique because {\cec $F$ is a decreasing function: indeed, for all $x\in(0, \lambda t)$}
\begin{equation*}
	F'(x) = \frac{g''(x)}{g(x)} - \bigg( \frac{g'(x)}{g(x)}\bigg)^2 - \bigg(\frac{1}{\lambda t - x} \bigg)^2 < 0  \quad \text{for all $x \in (0, \lambda t )$,}
\end{equation*}
since \smash{$g''(x) = -\frac{m''(g(x))}{(m'(g(x)))^3}< 0$} by Assumption {\bf (A5.1)}. Setting $\hat\sigma_t = \frac{1}{\lambda} x$ proves existence and uniqueness as required, moreover $\hat\sigma_t$ is increasing in $t$. \smallskip

\noindent
It remains to show that $\hat\sigma_t \rightarrow \infty$ as $t \rightarrow \infty$. 
If $\hat\sigma_t$ was bounded, we had
$\frac{1}{\lambda t - \lambda\hat\sigma_t} \rightarrow 0$ as $t \rightarrow \infty$. This implies that 
$(\log g)'(\lambda \hat\sigma_t) = g'(\lambda \hat\sigma_t)/ g(\lambda \hat\sigma_t)\rightarrow 0$
and hence $g'(\lambda \hat\sigma_t) \rightarrow 0$. This implies $\frac{1}{m'(g(\lambda \hat\sigma_t))} \rightarrow 0$, i.e. $m' \big( g(\lambda \hat\sigma_t)\big) \rightarrow \infty$. From Assumption~{\bf (A5.4)}, we know that $m'(x) \uparrow \infty$ as $x \uparrow 1$ and
% for all $x \in [0,1)$, $m'(x) < \infty$, 
therefore $g(\lambda \hat\sigma_t) \rightarrow 1$ %. Since $g^{-1}(1) = m(1) = \infty$ 
and hence $\hat\sigma_t \rightarrow \infty$ as required.\smallskip

\noindent
Finally we show that $\frac{\hat\sigma_t}{t}\to 0$. By definition of $\hat\sigma_t$, we have % (see Equation \eqref{equ:sigma}),
%	\begin{equation*}
$		t = \hat\sigma_t + \frac{g(\lambda \hat\sigma_t)}{\lambda g'(\lambda \hat\sigma_t)},$
%	\end{equation*}
so we can write
	\begin{equation*}
		\lim_{t \rightarrow \infty} \frac{\hat\sigma_t}{t} = \lim_{t \rightarrow \infty} \frac{\hat\sigma_t}{\hat\sigma_t + \frac{g(\lambda \hat\sigma_t)}{\lambda g'(\lambda \hat\sigma_t)}} = \lim_{t \rightarrow \infty} \frac{1}{1 + \frac{g(\lambda \hat\sigma_t)}{ \lambda \hat\sigma_t g'(\lambda \hat\sigma_t)}}.
	\end{equation*}
	As  $\hat\sigma_t\to \infty$ as have $g(\lambda \hat\sigma_t)\to 1$ as $t \rightarrow \infty$ we get
	\begin{equation*}
		\lim_{t \rightarrow \infty} \frac{\hat\sigma_t}{t} = \lim_{x \uparrow 1} \frac{1}{1 + \frac{x}{\frac{m(x)}{m'(x)}}} = 0,
	\end{equation*}
since $\lim_{x\uparrow 1} \frac{m(x)}{m'(x)} = 0$ by Assumption {\bf (A5.4)}.
\end{proof}
%%%%%%%%%%%%%%%%%%%%%%%%%%%%%%%%%%%

\noindent
From Lemma \ref{lem:sigma}, it follows that $\lambda t \sim \frac{g(\lambda \sigma_t)}{g'(\lambda \sigma_t)}$ as $t \rightarrow \infty$. 
Since {\cec $\lim_{t\uparrow\infty}g(\lambda \sigma_t ) = 1$}, we get that
\begin{equation}
	g'(\lambda \sigma_t) \sim \frac{1}{\lambda t} {\cec \quad\text{ when }t\uparrow \infty}.
	\label{equ:c1}
\end{equation}
%%%%%%%%%%%%%%%%%%%%%%%%%%%%%%%%%%%

\begin{lem} We have 
	\begin{equation} 
		\lim_{t\rightarrow \infty}\sigma_t  t g'' (\lambda \sigma_t) = - \varkappa \lambda^{-2},
	\label{equ:small_o1}
	\end{equation} 
where $\varkappa$ is defined in Assumption {\bf (A5.3)}, and 
	\begin{equation}
		\lim_{t\rightarrow \infty} \sigma_t g'(\lambda \sigma_t) = 0.
	\label{equ:small_o2}
	\end{equation} 
\label{lem:small_o}	
\end{lem}

\begin{proof} %[Proof of Lemma \ref{lem:small_o}]	
{\cec Recall that, by Lemma~\ref{lem:sigma}, for all $t$ large enough,} 
$\sigma_t =m(g(\lambda \sigma_t ))/\lambda$, $t \sim \frac{g(\lambda \sigma_t)}{ \lambda g'(\lambda \sigma_t)}$ {\cec as $t\uparrow\infty$},
%(by Lemma \ref{lem:sigma}) 
and
\begin{equation*}
	g''(x) = -\frac{m''(g(x))}{(m'(g(x)))^3} = -\frac{m''(g(x))g'(x)}{(m'(g(x)))^2} 
	\quad{\cec\text{ for all }x\in[0,\infty)},
\end{equation*}
since $m'\circ g = 1/g'$. Substituting these into \eqref{equ:small_o1} and substituting $x = g(\lambda \sigma_t)$, we get
\begin{equation*}
	\lim_{t \rightarrow \infty} \sigma_t  t g'' (\lambda \sigma_t) = \lim_{t \rightarrow \infty}- \frac{ m(g(\lambda \sigma_t)) g(\lambda \sigma_t) m''(g(\lambda \sigma_t)) }{ (\lambda m'(g(\lambda \sigma_t))^2} = \lim_{x\uparrow 1} - \frac{ m''(x)m(x) x}{(\lambda m'(x))^2} = -\varkappa \lambda^{-2},
	\end{equation*}
by Assumption {\bf (A5.3)}. Similarly, using $g'(\lambda \sigma_t) = \frac{1}{m'(g(\lambda \sigma_t))}$, we have
\begin{eqnarray*}
	\lim_{t\rightarrow \infty} \sigma_t g'(\lambda \sigma_t) = \lim_{t \rightarrow \infty} \frac{m(g(\lambda \sigma_t))}{\lambda m'(g(\lambda \sigma_t ))} = \lim_{x \uparrow 1} \frac{m(x)}{\lambda m'(x)} = 0,
\end{eqnarray*}
by Assumption {\bf (A5.4)}.
\end{proof}

%%%%%%%%%%%%%%%%%%%%%%%%%%%%%%%%%%%%%%%%%%

\subsection{Convergence of a simpler point process} \label{sec:convergence}
%Approximation by a simpler point process}

In this section we prove the following proposition, which gives a more basic form of the Poisson limit in
a space without compactification.
% use to prove a stronger result. 
%%%%%%%%%%%%%%%%%%%%%%%%%%%%%%%%%%
\begin{prop}[]
We have vague convergence in distribution of the point process
	\begin{equation*}
		\Psi_t = \sum^{M(t)}_{n=1} \delta\Big( \frac{\tau_n - \sigma_t } {\sqrt{\sigma_t }}, \frac{F_n - g \big(\log (n \sqrt{\sigma_t }) \big)}{g' \big (\log (n \sqrt{\sigma_t }) \big)}, \e^{-\gamma F_n(t-\tau_n)} Z_n(t) \Big)
	\end{equation*}
to the Poisson point process with intensity 
	\begin{equation*}
		\zeta^*(\di s, \: \df, \: \di z) = \lambda \e^{-f} \nu(z) \di s \: \df \: \di z,
	\end{equation*}
on $(-\infty, \infty) \times (-\infty, \infty] \times [0, \infty]$.
	\label{prop:cv_gamma2}
\end{prop}

%%%%%%%%%%%%%%%%%%%%%%%%%%%%%%%%%%
We prove Proposition \ref{prop:cv_gamma2} in two steps:
\begin{enumerate}[(1)]
	\item In Lemma \ref{lem_9} we approximate $\Psi_t$ by the point process 
		\begin{equation*}
			\Psi^*_t =  \sum_{n \in \N} \delta\bigg( \frac{\frac{1}{\lambda} \log n -  \sigma_t  }{\sqrt{\sigma_t }}, 
			\frac{F_n - g \big(\log(n \sqrt{\sigma_t })\big)}{g' \big (\log(n \sqrt{\sigma_t }) \big)}, \xi_n \bigg),
		\end{equation*}
	where we have replaced the rescaled family sizes $\e^{-\gamma F_n (t-\tau_n)} Z_n(t)$ by their limits, denoted $\xi_n$, and the birth times $\tau_n$ by the approximate birth times $\frac{1}{\lambda} \log n$, using Assumptions {\bf (A3)} and {\bf (A1)} respectively.
	\item In Lemma \ref{lem:Kallenberg} we prove that $\Psi^*_t$ converges to the Poisson point process with intensity $\zeta^*$.
\end{enumerate}

\begin{lem}
	{\peter As as $t\to\infty$  the process $(\Psi^*_t)_{t\ge0}$ converges vaguely in distribution}
%{\cec weakly}\footnote{\peter We say that a sequence of finite measures $(\mu_n)_{n \in \mathbb{N}}$ on 
%a locally compact Polish space $\mathbb X$ converges \emph{weakly} to $\mu$ %, and denote this by $\mu_n \xrightarrow{v} \mu$, 
%iff
%$		\int f \di \mu_n \rightarrow \int f \di \mu,$ as $n \rightarrow \infty$,
%for all continuous bounded functions $f \colon \mathbb X \rightarrow \R$. Note that weak convergence implies vague convergence.} 
on $(-\infty, \infty) \times (-\infty, \infty] \times [0, \infty] $ to the Poisson point process with intensity $\zeta^*$. 
	\label{lem:Kallenberg}
\end{lem}
\begin{proof}
{\peter We apply Kallenberg's theorem, see \cite[Proposition 3.22]{Resnick}. Since $\zeta^*$ is diffuse it suffices}
% {\cec Note that this theorem applies to sequences of point processes indexed by $\mathbb N$ while our sequence is indexed by $t\in [0,\infty)$. But since our process is in fact} {\peter piecewise deterministic, Kallenberg's theorem applies to the process $(\Psi^*_{t_i})_{i\geq 1}$, where the $t_i$'s are the (successive) times at which $(\Psi^*_t)_{t\geq 0}$ switches.} Therefore, since $\zeta^*$ is diffuse, to prove Lemma~\ref{lem:Kallenberg}, it is enough 
 to show that, for every precompact relatively open box $B \subset (-\infty, \infty) \times (-\infty, \infty] \times [0, \infty]$, we have 
	\begin{enumerate}[(a)]
		\item $\pr(\Psi^*_t(B) = 0 ) \rightarrow \exp (-\zeta^*(B))$, as $t \uparrow \infty$, and
		\item $\E[\Psi^*_t(B)] \rightarrow \zeta^* (B )$, as $t \uparrow \infty$. 
	\end{enumerate}
%	{\peter Under these conditions, Kallenberg's theorem implies that $(\Psi^*_t)_{t\geq 0}$ converges vaguely)to the Poisson point process of intensity $\zeta^*$ on $(-\infty, \infty) \times (-\infty, \infty] \times [0, \infty]$.}
It further suffices to consider nonempty boxes $B$ of the form $(s_0, s_1) \times (f_0, f_1) \times (z_0, z_1)$, where $s_0, s_1 \in (-\infty, \infty), f_0, f_1 \in (-\infty, \infty]$,  $z_0, z_1 \in[ 0, \infty]$, and $s_0<s_1$, $f_0 < f_1$, $z_0 < z_1$.
Note that 
	\begin{equation*}
		\zeta^*(B) = \lambda (s_1-s_0)\big(\e^{-f_0} - \e^{-f_1}\big)\int_{z_0}^{z_1} \nu (x) \dx.
	\end{equation*}
	
	\begin{enumerate}[(a)]
		\item 
		Let 
		\begin{equation*}
			r(a) := \exp \big( \lambda (a \sqrt{\sigma_t } +  \sigma_t ) \big), \quad \text{for all $a{\cec \in\mathbb R}$,}
		\end{equation*}
and consider 
	\begin{eqnarray*}
		& & \hat{\Psi}^*_t =  \sum_{n \in \N} \delta\bigg( \frac{ \frac{1}{\lambda} \log n -  \sigma_t  }{\sqrt{\sigma_t }}, 
			\frac{F_n - g\big(\log(n \sqrt{\sigma_t })\big)}{g'\big(\log(n \sqrt{\sigma_t })\big)} \bigg),\\
		&&  \hat{\zeta}^*(\di s, \: \df) = \lambda \e^{-f}\di s \: \df.		
	\end{eqnarray*}
So that for $\hat{B}  =  (s_0, s_1) \times (f_0, f_1)$, we get $\hat{\zeta}^*(\hat{B}) =  \lambda (s_1 - s_0)\big( \e^{-f_0} - \e^{-f_1} \big)$. Denote, for all $a>0$, {\cec $x\in\mathbb R$}, 
\begin{equation*}
	\hat{f}_a(x) = g\big( \log(a \sqrt{\sigma_t }) \big) + x g' \big( \log(a \sqrt{\sigma_t })\big).
\end{equation*}
Then we have 
	\begin{eqnarray*}
		\pr(\hat{\Psi}^*_t(\hat{B}) = 0) & = & \prod^{r(s_1)}_{n=r(s_0)} \pr \bigg( \frac{F_n - g\big(\log(n \sqrt{\sigma_t })\big)}{g'\big(\log(n \sqrt{\sigma_t })\big)} \notin (f_0, f_1 )\bigg) \\
		& = & 
%\prod^{r(s_1)}_{n=r(s_0)} \Big[ \mu\big(0, \hat{f}_n(f_0) \big) + \mu\big(\hat{f}_n(f_1),1 \big) \Big]  = 
 \prod^{r(s_1)}_{n=r(s_0)} \Big[ 1 - \mu\big( \hat{f}_n(f_0), \hat{f}_n(f_1)\big) \Big]\\
%	\end{eqnarray*}
%Using the fact that $e^{-\mu(x_0,x_1)} = 1 - \mu(x_0,x_1) + o (\mu(x_0,x_1))$ when $x_0, x_1 \rightarrow 1$, we get , as $t \uparrow \infty$,
%	\begin{eqnarray*}
%		\pr(\hat{\Psi}_t^*(\hat{B}) = 0) &\sim &\prod ^{r(s_1)}_{n=r(s_0)} \e^ {- \mu\big( \hat{f}_n(f_0), \hat{f}_n(f_1) \big)} \\
		&\sim & \prod ^{r(s_1)}_{n=r(s_0)} \exp \Big\{- \mu\big( \hat{f}_n(f_0), 1 \big) + \mu\big( \hat{f}_n(f_1),1 \big) \Big\}.
	\end{eqnarray*}
Recalling that $\mu(x,1) = \e^{-m(x)} $, we get the following: {\cec when $t\to\infty$,}
	\begin{eqnarray*}
		\pr(\hat{\Psi}_t^*(\hat{B}) = 0) & \sim & \exp \bigg\{ \sum^{r(s_1)}_{n=r(s_0)} - \e^{ - m\big( \hat{f}_n(f_0)\big)} + \e^{ - m\big( \hat{f}_n(f_1)\big)}\bigg\} \\
		&  \sim & \exp \bigg\{ - \int^{r(s_1)}_{r(s_0)} \e^ {- m \big(\hat{f}_x(f_0) \big)} \dx + \int^{r(s_1)}_{r(s_0)}\e^ {- m \big(\hat{f}_x(f_1) \big)}  \dx \bigg \}.
	\end{eqnarray*}
We now evaluate the integrals in the exponent. For $i =0,1$ we have
\begin{align*}
	\int^{r(s_1)}_{r(s_0)} \e^ {- m \big(\hat{f}_x(f_i) \big)} \dx& =  \int^{r(s_1)}_{r(s_0)} \exp \Big\{ - m \Big( g\big( \log(x \sqrt{\sigma_t }) \big) + f_ig' \big( \log(x \sqrt{\sigma_t })\big) \Big) \Big\} \, \dx \\
	& =  \int^{s_1}_{s_0} \lambda \sqrt{\sigma_t} r(y) \exp \Big\{ - m \Big( g\Big( \log\big(r(y)\sqrt{\sigma_t }\big) \Big) + f_ig' \Big( \log\big(r(y) \sqrt{\sigma_t }\big)\Big) \Big) \Big\} \,\dy,
\end{align*}	
by a change of variables, with $x=\e^{\lambda (y \sqrt{\sigma_t} + \sigma_t)} = r(y)$. 
%\begin{equation*}
%	\begin{split}
%	& \pr(\hat{\Psi}_t^*(\hat{B}) = 0) \\
%	& \hspace{1cm} \sim \exp \bigg\{ - \int^{s_1}_{s_0} \exp \Big\{ - m \Big( g\Big( \log\big(r(y)\sqrt{\sigma_t }\big) \Big) + f_0g' \Big( \log\big(r(y) \sqrt{\sigma_t }\big)\Big) \Big) \Big\} \lambda \sqrt{\sigma_t} r(y)\dy \bigg \}.
%	\end{split}
%\end{equation*}
By the mean value theorem, for each $i\in\{0,1\}$, there exists a constant $c_3 \in \big[g\big( \log\big(r(y)\sqrt{\sigma_t }\big) \big), g\big( \log\big(r(y)\sqrt{\sigma_t }\big) \big) + f_0g' \big( \log \big(r(y)\sqrt{\sigma_t }\big)\big) \big]$, such that 
\begin{align*}
m \Big( g\big( & \log\big(r(y)\sqrt{\sigma_t }\big) \big) + f_ig' \big( \log \big(r(y)\sqrt{\sigma_t }\big)\big) \Big) \\
	& =  m \Big( g\big( \log(r(y) \sqrt{\sigma_t }) \big) \Big) + f_ig' \big( \log(r(y) \sqrt{\sigma_t })\big) m' \Big(g \big(\log (r(y) \sqrt{\sigma_t })\big) \Big) \\
& \phantom{quatsch}+ \sfrac{1}{2} \Big(f_i g' \big(\log(r(y) \sqrt{\sigma_t }) \big)\Big)^2 m''(c_3).
	\end{align*}	
Recall that, for $x \in \R$, we have $m(g(x)) = x$ and $g'(x) = \frac{1}{m'(g(x))}$, so the integral simplifies to
	\begin{eqnarray*}
		\int^{r(s_1)}_{r(s_0)} \e^ {- m \big(\hat{f}_x(f_i) \big)} \dx& = &  \int_{s_0}^{s_1}  \lambda \: r(y) \sqrt{\sigma_t}  \: \e^{ -\log(r(y) \sqrt{\sigma_t }) - f_i -  \frac{1}{2}\big(f_i g' (\log(r(y) \sqrt{\sigma_t }))\big)^2 m''(c_3)} \: \dy  \\	
		& = & \lambda \:  \int^{s_1}_{s_0} \exp \Big \{ - f_i - \sfrac{f_i^2}{2} g' \big(\log(r(y) \sqrt{\sigma_t }) \big)^2 m''(c_3)\Big \}  \:\dy . 
	\end{eqnarray*}	
{\cec Recall that $c_3 \in \big[g\big( \log\big(r(y)\sqrt{\sigma_t }\big) \big), g\big( \log\big(r(y)\sqrt{\sigma_t }\big) \big) + f_0g' \big( \log \big(r(y)\sqrt{\sigma_t }\big)\big) \big]$. By Assumption {\bf (A5.4)} and since $\lim_{x\uparrow 1}m(x) = \infty$, we have $\lim_{x\uparrow\infty}g'(x) = 0$, and thus $c_3\sim g\big( \log\big(r(y)\sqrt{\sigma_t }\big) \big)$ when $t\uparrow\infty$.} By Assumption~{\bf (A5.2)}, we thus get
$$g' \big(\log( r(y) \sqrt{\sigma_t }) \big)^2 m''(c_3) = \frac{m''(c_3)}{ m'\big(g(\log(r(y) \sqrt{\sigma_t }))\big)^2} \rightarrow 0,$$ 
as $t \rightarrow \infty$. By the dominated convergence theorem, as $t \rightarrow \infty,$ we get 
	\[
		\int^{r(s_1)}_{r(s_0)} \e^ {- m \big(\hat{f}_x(f_i) \big)} \dx
		%= & \lambda \:  \int^{s_1}_{s_0} \e^{ - f_i  + o(1)} \:\dy 
		= \lambda (s_1 - s_0) \e^{-f_i} + o(1). 
	\]
Therefore, as $t \rightarrow \infty$ we get 
	\begin{eqnarray*}
		\pr(\hat{\Psi}_t^*(\hat{B}) = 0) & \sim &   \exp \bigg\{ - \lambda (s_1 - s_0) \e^{-f_0} + \lambda (s_1 - s_0) \e^{-f_1} + o(1)\bigg\}\\
		& \rightarrow &   \exp \bigg\{ - \lambda (s_1 - s_0) \big( \e^{-f_0} - \e^{-f_1}) \bigg\} = \exp \big\{- \hat \zeta^* (\hat B) \big\}.
	\end{eqnarray*}	
{Using Kallenberg's theorem, we thus get that, in distribution when $t\to\infty$, $\hat \Psi_t^*$ converge vaguely on $(-\infty, \infty)\times (-\infty, \infty]$ to the Poisson point process of intensity $\hat\zeta^*$.}
By assumption, $(F_n, \xi_n)_{n\ge1}$ is a sequence of i.i.d.\ random variables with each $F_n$ being independent of $\xi_n$. Together with the fact that $\pr(\xi_n \in (z_0, z_1)) = \int_{z_0}^{z_1} \nu(x) \dx$, this completes the proof of (a). 
	\item To calculate the limit of $\E[\Psi^*_t(B)]$ we apply similar asymptotic estimates as in part (a), and get that, {\cec when $t\uparrow\infty$}
		\begin{eqnarray*}
			\E[\Psi_t^*(B)] & = & \sum_{r(a_0 ) \le n \le r(a_1)} \mu \Big( \hat{f}_n(f_0), \hat{f}_n(f_1) \Big) \times \pr(\xi_1 \in [z_0, z_1]) \\
			& \sim & \int_{r(s_0)}^{r(s_1)} \mu \Big( \hat{f}_x(f_0), \hat{f}_x(f_1) \Big) \times \pr(\xi_1 \in [z_0, z_1]) \dx \\
			& \sim & \lambda (s_1 - s_0)\big( \e^{-f_0} - \e^{-f_1}\big) \int_{z_0}^{z_1} \nu(x) \dx 
			= \zeta^*(B).
		\end{eqnarray*}
	\end{enumerate}
\ \\[-1.3cm]
\end{proof}
\medskip

%%%%%%%%%%%%%%%%%%%%%%%%%%%%%%%%%%%%%%%%%%%%%%%%%
%%%%%%%%%%%%%%%%%%%%%%%%%%%%%%%%%%%%%%%%%%%%%%%%%
\begin{lem}[]
For all Lipschitz continuous, compactly supported functions $f:(-\infty, \infty) \times (-\infty, \infty] \times [0, \infty] \rightarrow \R$,
\begin{equation*}
	\bigg | \int f \di \Psi^*_t - \int f \di \Psi_t \bigg | \rightarrow 0 \text{ in probability, as } t \uparrow \infty.
\end{equation*}
\label{lem_9}
\end{lem}
\  \\[-15mm]
{\cec \begin{note}
By density of the set of Lipschitz-continuous compactly-supported functions in the set of continuous compactly supported functions for the topology of the $L^\infty$ norm, Lemma\ref{lem_9} implies that, for all continuous compactly supported functions $f:(-\infty, \infty) \times (-\infty, \infty] \times [0, \infty] \rightarrow \R$,
\begin{equation*}
	\bigg | \int f \di \Psi^*_t - \int f \di \Psi_t \bigg | \rightarrow 0 \text{ in probability, as } t \uparrow \infty.
\end{equation*}
\end{note}}

\begin{proof}
	Let $f$ be a Lipschitz continuous function supported on $K = [-a, a] \times [-b, \infty] \times [0, \infty]$ for $1 \le a, b < \infty$. We have, {\cec for all $t\geq 0$,}
	\begin{align}
	\bigg | \int f \di & \Psi^*_t  -  \int f \di \Psi_t \bigg |  \nonumber \\
		& \le  \sum_{n=1}^{M(t)} \bigg | f \Big( \frac{\tau_n -  \sigma_t  }{\sqrt{\sigma_t }}, 
		\sfrac{F_n - g\big(\log(n \sqrt{\sigma_t })\big)}{g'\big(\log(n \sqrt{\sigma_t })\big)}, \e^{-\gamma F_n(t-\tau_n)} Z_n(t) \Big)  
%\nonumber \\
%		& & \hspace{3.3cm} 
- f \Big( \frac{\frac{1}{\lambda} \log n -  \sigma_t }{\sqrt{\sigma_t }}, 
	\sfrac{F_n - g\big(\log(n \sqrt{\sigma_t })\big)}{g'\big(\log(n \sqrt{\sigma_t })\big)}, \xi_n \Big) \bigg |  \nonumber \\
		& \le  c_L \sum_{n\in \hat{I}(t)} \bigg( \Big | \frac{\tau_n - \tau_n^*}{\sqrt{\sigma_t }} \Big | + \Big | \e^{-\gamma F_n(t-\tau_n)} Z_n(t) - \xi_n \Big | \bigg), 
	\label{equ:5_4}
	\end{align}
where $c_L$ is the Lipschitz constant of the function $f$, $\xi_n$ are i.i.d.\:copies of $\xi$ (defined in Assumption~{\bf (A3)}), $\tau_n^* = \frac{1}{\lambda} \log n$, and $\hat{I}(t)$ is the random set of indices $n\in \N$ such that
\begin{enumerate}[(a)]
	\item $\big | \frac{\tau_n -  \sigma_t }{\sqrt{\sigma_t }}\big| \le a$ and $ \frac{F_n - g (\log(n \sqrt{\sigma_t } ) )}{g'(\log(n \sqrt{\sigma_t }))} \ge -b$ or
	\item $\big | \frac{\tau^*_n -  \sigma_t }{\sqrt{\sigma_t }}\big| \le a$ and $ \frac{F_n - g(\log(n \sqrt{\sigma_t }))}{g'(\log(n \sqrt{\sigma_t }))} \ge -b$.
\end{enumerate}
{\cec The last inequality of Equation~\eqref{equ:5_4} comes from the fact that, by definition of $\hat I(t)$, all summands associated to integers $n\notin \hat{I}(t)$ are zero because the support of $f$ is included in $[-a,a]\times [-b,\infty]\times [0,\infty]$.}
{\cec By Lemma~\ref{lem:sigma}, there exists $t_0$ such that, for all $t\geq t_0$,} $ \sigma_t  \le \frac{t}{3}$ and $\sqrt{\sigma_t }\le \sigma_t$. For $\varepsilon \in (0, \nicefrac12)$ we denote by $\Upsilon_\varepsilon(t)$ the event that
	\begin{equation*}
 		|\tau_n - \tau_n^* | \le \varepsilon \sqrt{\sigma_t } \quad \text{for all $n \in \N$.}
	\end{equation*}

\noindent 
Assumption {\bf (A1)} {together with Lemma~\ref{lem:sigma} implies that}
%implies that $\sup_{n\geq m} |\tau_n^* - \tau_n - T| \to 0$ in probability when $m\to\infty$. Together with Lemma~\ref{lem:sigma} this implies
$\pr(\Upsilon_\varepsilon(t)) \rightarrow 1$, as $t \rightarrow \infty$ for all $\varepsilon {\cec \in(0,\nicefrac12)}$. Set 
\begin{equation*}
	\bar{I}(t) := \Big\{ n \in \N : \sfrac{|\tau^*_n -  \sigma_t |}{\sqrt{\sigma_t }} \le 2a,  \text{ and } \sfrac{F_n - g(\log(n \sqrt{\sigma_t }))}{g'(\log(n \sqrt{\sigma_t }))}  \ge - b \Big\}.
\end{equation*}
We have that $\hat{I}(t) \subset \bar{I}(t)$ on $\Upsilon_\varepsilon(t)$. 
Indeed, if (a) and $\Upsilon_\varepsilon(t)$ hold then 
\begin{equation*}
	\frac{|\tau^*_n -  \sigma_t |}{\sqrt{\sigma_t }} \le \frac{|\tau^*_n - \tau_n|}{\sqrt{\sigma_t }} + \frac{|\tau_n -  \sigma_t |}{\sqrt{\sigma_t }} \le \varepsilon + a \le 2a,
\end{equation*}
and similarly if (b) hold. We now consider the sum on the right hand side of Equation~\eqref{equ:5_4}, but taken over all $n \in \bar{I}(t)$. First note that, for $n \in \bar{I}(t)$ on $\Upsilon_\varepsilon (t)$, we have
\begin{eqnarray}
	\tau_n \le 2a\sqrt{\sigma_t } +  \sigma_t  \le 2a \sigma_t +  \sigma_t  = \sigma_t  (2 a+1)  
\le \sfrac{t}{2},
	\label{eq:inequality_1}
\end{eqnarray}
for all {\cec $t\geq t_0$}. Since $(\log g)'(\log(n \sqrt{\sigma_t })) \rightarrow 0$ as $t \rightarrow \infty$, and $g(\log(n \sqrt{\sigma_t })) \rightarrow 1$, we have
\begin{eqnarray}
	\label{eq:inequality_2}
	F_n & \ge & g(\log (n \sqrt{\sigma_t })) - b g' (\log (n \sqrt{\sigma_t })) \\
	& =  & g(\log (n \sqrt{\sigma_t }))\Big(1 - b ( \log g )'(\log(n \sqrt{\sigma_t }) \Big) 
	 \rightarrow  1, \nonumber 
\end{eqnarray}
as $t \rightarrow \infty$. Recall $\Delta_n(t)$ from Assumption {\bf (A2)},
$\xi_n=\lim\limits_{t\to\infty} \e^{-\gamma t} Y_n(t)$, and define
$$R_n(t) :=\sup_{w>t} \big| \e^{-\gamma w} Y_n(w) - \xi_n \big|.$$
By Assumption {\bf (A3)} we have $R_n(t)\to 0$ in probability and, for all $t$ large enough, 
we have
\begin{eqnarray*}
	\Big| \e^{-\gamma F_n(t - \tau_n)} Z_n(t) - \xi_n \Big| &  \le & \Big| \e^{-\gamma F_n(t-\tau_n)} Z_n(t)  -  \e^{-\gamma F_n (t - \tau_n)} Y_n(F_n(t - \tau_n))\Big| \\
	& & \hspace{1cm} + \Big| \e^{-\gamma F_n (t - \tau_n)} Y_n(F_n(t - \tau_n)) - \xi_n \Big|  \\
	&\le& \Delta_n(F_n(t-\tau_n)) + R_n (F_n(t-\tau_n) )  \\
	& \le & \Delta_n\big(\sfrac t2\big) + R_n \big(\sfrac t2\big),
\end{eqnarray*}
where we have used Equations \eqref{eq:inequality_1} and \eqref{eq:inequality_2}. Hence we get that, for sufficiently large $t$, on $\Upsilon_\varepsilon (t)$, 
\begin{eqnarray*}
	\bigg| \int f \di \Psi_t - \int f \di \Psi^*_t \bigg| & \le & c_L \sum_{n \in \bar{I}(t)} \bigg( \frac{|\tau_n - \tau_n^*|}{\sqrt{\sigma_t }} + \Big | \e^{-\gamma F_n(t-\tau_n)} Z_n(t) - \xi_n \Big | \bigg) \\
	& \le & c_L \sum_{n \in \bar{I}(t)} \Big( \frac{\sup_n | \tau_n - \tau_n^*|}{\sqrt{\sigma_t }} + \Delta_n\Big(\frac t2\Big) +  R_n \Big( \frac t2 \Big) \Big) \\
	& \le & c_L \frac{|\bar{I}(t)| \sup_n |\tau_n - \tau_n^*|}{\sqrt{\sigma_t }} +  c_L \sum_{n\in \bar{I}(t)} \Delta_n\Big(\frac t2\Big) + c_L \sum_{n\in \bar{I}(t)} R_n \Big( \frac t2\Big).
\end{eqnarray*}
By assumption, the random processes $(R_n)_{n\ge 1}$ are independent of $(F_n)_{n\ge 1}$ and thus also of the random set $\bar{I}(t)$. Recall that, by Lemma \ref{lem:Kallenberg}, $|\bar{I}(t)|$ converges in distribution to a Poisson distribution and 
%\begin{equation*}
%\lim_{v\rightarrow \infty} R_n(v) =0, \quad \text{in probability,} 
%\end{equation*}
%by Assumption {\bf (A3)}, implying that
hence
\begin{equation*}
	\lim_{t \rightarrow \infty} \sum_{n \in \bar{I}(t)} R_n \big( \sfrac{t}{2}\big) = 0, \quad \text{in probability.}
\end{equation*}
\noindent To prove that $\sum_{n\in \bar{I}(t)} \Delta_n\big(\frac t2\big) \rightarrow 0$ in probability as $t \rightarrow \infty$ we use Assumption (A.2). We have
\begin{align*}
\pr \bigg( \sum_{n\in \bar{I}(t)} \Delta_n\Big(\frac t2\Big) \ge \varepsilon \bigg) 
& = \E \bigg[  \pr \bigg( \sum_{n\in \bar{I}(t)} \Delta_n\Big(\frac t2\Big) \ge \varepsilon \Big| \, (F_n) \bigg) \bigg] \\
	& \le  \sum_{k = 0}^\infty \E \Big[ \pr \Big( \exists n \in \bar{I}(t) : \Delta_n\Big(\frac t2\Big) \ge \frac{\varepsilon}{k} \Big | (F_n) \Big) \id_{\{|\bar{I}(t)| = k\}} \Big] \\
	& \le  \sum_{k = 0}^\infty \E \bigg[ \sum_{n\in \bar{I}(t)} \pr \bigg( \Delta_n\Big(\frac t2\Big) \ge \frac{\varepsilon}{k} \Big | (F_n) \bigg) \id_{\{|\bar{I}(t)| = k\}} \bigg] \\
& \le  \sum_{k = 0}^\infty \E \bigg[ k \, \max_{n\in \bar{I}(t)} \pr \bigg( \Delta_n\Big(\frac t2\Big) \ge \frac{\varepsilon}{k} \Big | (F_n) \bigg) \id_{\{|\bar{I}(t)| = k\}} \bigg] .
\end{align*}
Now, given $\delta>0$ pick $K\in\N$ such that, for sufficiently large $t$,
\[\sum_{k=K+1}^\infty \E \bigg[ k \, \max_{n\in \bar{I}(t)} \pr \bigg( \Delta_n\Big(\frac t2\Big) \ge \frac{\varepsilon}{k} \Big | (F_n) \bigg) \id_{\{|\bar{I}(t)| = k\}} \bigg] \leq \E \Big[ |\bar{I}(t)| \id_{|\bar{I}(t)|>K}\Big].\]
{\cec In the proof of Lemma~\ref{lem:Kallenberg}, we have proved that for all pre-compact relatively open box $B$, $\mathbb E[\Psi^*(B)] \to \zeta^*(B)$. The exact same proof applies to any compact box (because $\mathrm{PPP}(\zeta^*)(\partial B)=0$ {\peter as $\zeta^*$ is diffuse}), and applying this convergence to $B = [-2a, 2a]\times [-b, \infty]\times [0, \infty]$ gives that $\mathbb E|\bar I(t)| \to \zeta^*(B)$.} {\peter Moreover, by Lemma~\ref{lem:Kallenberg}, we get that $\bar I(t)$ converges to a Poisson distribution of parameter $\zeta^*(B)$.}
Thus, by dominated convergence, for all $\delta>0$, there exists $K$ sufficiently large such that
\[\sum_{k=K+1}^\infty \E \bigg[ k \, \max_{n\in \bar{I}(t)} \pr \bigg( \Delta_n\Big(\frac t2\Big) \ge \frac{\varepsilon}{k} \Big | (F_n) \bigg) \id_{\{|\bar{I}(t)| = k\}} \bigg] \leq \frac\delta2.\]
{Since, by definition, $\bar I(t)\subseteq \bor{I}(t)$ for $\kappa = 2a$, we get}
\begin{align*}
\sum_{k = 0}^K\E \bigg[ k \, \max_{n\in \bar{I}(t)} \pr & \bigg( \Delta_n\Big(\frac t2\Big) \ge \frac{\varepsilon}{k} \Big | (F_n) \bigg) \id_{\{|\bar{I}(t)| = k\}} \bigg] 
\leq K(K+1) \E \bigg[ \max_{n\in \bor{I}(t)} \pr  \bigg( \Delta_n\Big(\frac t2\Big) \ge \frac{\varepsilon}{K} \Big | (F_n) \bigg) \bigg], 
\end{align*}
which converges to zero by (A.2) and dominated convergence. This shows that
$\sum_{n\in \bar{I}(t)} \Delta_n\big(\frac t2\big) \rightarrow 0$ in probability.
\noindent Summarising, we get that, {\cec in probability when $t\uparrow\infty$,}
\begin{equation*}
	\bigg| \int f \di \Psi_t - \int f \di \Psi^*_t \bigg| \le c_L \big| \bar{I}(t) \big| \frac{\sup_n |\tau_n - \tau_n^*|}{\sqrt{\sigma_t }} + o(1),
\end{equation*}
which converges to zero in probability, as $t \uparrow \infty$.
\end{proof}

\begin{proof}[Proof of Proposition \ref{prop:cv_gamma2}]
	Let $f \colon (-\infty, \infty) \times (-\infty, \infty] \times [0, \infty] \rightarrow \R$ be Lipschitz continuous and compactly supported. Combining Lemmas \ref{lem:Kallenberg} and \ref{lem_9}, together with Slutsky's theorem (see for example \cite[ch.7.2]{Probability}) we get %the desired result, 
	%\begin{equation*}
		$\int f \di \Psi_t \Rightarrow \int f \di \text{PPP}(\zeta^*)$ as $t \rightarrow \infty$,
	%\end{equation*}
where PPP($\zeta^*$) denotes the Poisson point process with intensity $\zeta^*$.
\end{proof}

%%%%%%%%%%%%%%%%%%%%%%%%%%%%%%%%%%%
%%%%%%%%%%%%%%%%%%%%%%%%%%%%%%%%%%%
\subsection{Proof of the local convergence result} \label{sec:proof_local}
%Proposition \ref{prop:cv_gamma}}
%We  are now able to show convergence of the point processes $\Gamma_t$.

\begin{prop}[]
{\cec Asymptotically when $t\to\infty$,} the point process
	\begin{equation*}
			\Gamma_t = \sum^{M(t)}_{n=1} \delta\Big( \frac{\tau_n - \sigma_t } {\sqrt{\sigma_t }}, \sfrac{F_n - g\big(\log (n \sqrt{\sigma_t } )\big)}{g'\big(\log(n \sqrt{\sigma_t })\big)}, \e^{-\gamma g(\lambda \sigma_t)(t-\sigma_t ) - a_1 g(\lambda \sigma_t) \log \sigma_t + \gamma T} Z_n(t) \Big), 
	\end{equation*}
converges vaguely in distribution on $(-\infty, \infty) \times (-\infty, \infty) \times [0, \infty]$ to the Poisson point process with intensity
	\begin{equation*}
		\zeta(\di s, \: \df, \: \di z) = \lambda \e^{-f} \e^{s^2 a_2 - f a_3} \nu(z\e^{s^2 a_2 -f a_3}) 
\, \di s \: \df \: \di z.
	\end{equation*}
	\label{prop:cv_gamma}
\end{prop}
%%%%%%%%%%%%%%%%%%%%%%%%%%%%%%%%%%
\begin{proof}[Proof of Proposition \ref{prop:cv_gamma}]
Consider the continuous function
	\begin{equation*}
		\phi \colon (s,f,z) \rightarrow (s,f,\e^{-s^2 a_2 + f a_3} z),
	\end{equation*}
so that $\zeta \circ \phi^{-1}  = \zeta^*$. We argue that $\Psi_t\circ \phi^{-1}$ is asymptotically equivalent to $\Gamma_t$, i.e.\ for all Lipschitz continuous, compactly supported functions $f \colon (-\infty, \infty) \times (-\infty, \infty) \times [0, \infty] \rightarrow \R $,
	\begin{equation*}
		\bigg| \int f \di \Psi_t \circ \phi^{-1}- \int f \di \Gamma_t \bigg| \rightarrow 0 \quad \text{in probability, as $t \uparrow \infty.$}
	\end{equation*}
To prove this let $f$ be a Lipschitz continuous function with Lipschitz constant $c_L$, supported on $K = [-a, a] \times [-b, b] \times [0, \infty]$ for $1 \le a, b < \infty$ and abbreviate
\begin{equation*}
	s_n = \frac{\tau_n -  \sigma_t }{\sqrt{\sigma_t }}  \quad \text{ and } \quad 
	f_n =  \frac{F_n - g(\log(n \sqrt{\sigma_t }))}{g'(\log(n \sqrt{\sigma_t }))}, \quad \mbox{ for } n \ge 1.
\end{equation*}
{\cec For all $t\geq 0$,} we have
	\begin{align}
		\bigg| \int f \di \Psi_t\circ \phi^{-1}& -  \int f \di \Gamma_t \bigg|  \nonumber \\
		& \le \sum_{n = 1}^{M(t)} \bigg| f \Big( \frac{\tau_n - \sigma_t}{\sqrt{\sigma_t}}, \sfrac{F_n - g \big(\log(n \sqrt{\sigma_t}) \big)}{ g \big(\log(n \sqrt{\sigma_t}) \big)}, \e^{-a_2 s_n^2 + a_3 f_n} \e^{-\gamma F_n(t - \tau_n)} Z_n(t) \Big) \nonumber \\
		& \hspace{0.8 cm} - f \Big(\frac{\tau_n - \sigma_t}{\sqrt{\sigma_t}}, \sfrac{F_n - g \big(\log(n \sqrt{\sigma_t}) \big)}{ g \big(\log(n \sqrt{\sigma_t}) \big)}, \e^{-\gamma g(\lambda \sigma_t)(t-\sigma_t ) - a_1 g(\lambda \sigma_t) \log \sigma_t + \gamma T} Z_n(t) \Big)\bigg| \nonumber \\
		&\le c_L \sum_{n\in \tilde{I}(t)}  \Big| \e^{-\gamma F_n(t - \tau_n) -a_2 s_n^2 + a_3 f_n } Z_n(t) - \e^{-\gamma g(\lambda \sigma_t)(t-\sigma_t ) - a_1 g(\lambda \sigma_t) \log \sigma_t + \gamma T} Z_n(t) \Big|,
		\label{equ:pgs}
	\end{align}
where $\tilde{I}(t)$ is the random set of indices $n \in \N$ such that $|s_n|\leq a$ and $|f_n|\leq b$
{\cec (this definition implies that all summands associated to integers $n\notin\tilde I(t)$ are zero because the support of $f$ is contained in~$K$)}.
%	\begin{equation*}
%		\Big | \frac{\tau_n -  \sigma_t }{\sqrt{\sigma_t }}\Big| \le a \quad \text{ and } \quad \Big| \frac{F_n - g (\log(n \sqrt{\sigma_t } ) )}{g'(\log(n \sqrt{\sigma_t }))} \Big| \le b.
%	\end{equation*}
We now show that the exponents of \eqref{equ:pgs} are asymptotically equivalent, namely
	\begin{equation}
		-\gamma F_n(t - \tau_n) -a_2 s_n^2 + a_3 f_n = -\gamma g(\lambda \sigma_t)(t-\sigma_t ) - a_1 g(\lambda \sigma_t) \log \sigma_t + \gamma T + o(1),
		\label{equ:exponents}
	\end{equation}
{\cec almost surely when $t\uparrow\infty$}, where the $o(1)$-term {\cec is uniform in} $n$. Indeed, by definition of $s_n$ and using Assumption {\bf (A1)}, we get
\begin{equation*}
	\log n = \lambda ( \sigma_t  + s_n \sqrt{\sigma_t } - T_n ), \quad \mbox{ for } n \ge 1,
\end{equation*} 
where we set $T_n = T+\varepsilon_n$. Therefore, we have
\begin{eqnarray*}
	F_n & = & g \big(\log (n \sqrt{\sigma_t }) \big) + f_n g'  \big(\log (n \sqrt{\sigma_t })  \big)\\
	& = & g  \Big(\lambda ( \sigma_t  + s_n \sqrt{\sigma_t } - T_n) + \sfrac{1}{2} \log \sigma_t \Big) + f_n g'  \Big(\lambda ( \sigma_t  + s_n \sqrt{\sigma_t } - T_n) + \sfrac{1}{2} \log \sigma_t  \Big).
\end{eqnarray*}
Let $x_n := \lambda s_n \sqrt{\sigma_t }+ \sfrac{1}{2} \log \sigma_t - \lambda T_n $, so that
\begin{equation*}
	F_n(t- \tau_n) = \Big (g  \big(\lambda \sigma_t  + x_n \big) + f_n g'  \big(\lambda \sigma_t  +x_n \big) \Big ) \big(t - \sigma_t  - s_n \sqrt{\sigma_t } \big).
\end{equation*}
By the mean value theorem, there exist $c_1, c_2 \in [\lambda \sigma_t, \lambda \sigma_t  + x_n]$, such that 
\begin{eqnarray}
	g  \big(\lambda \sigma_t + x_n \big) & = & g(\lambda \sigma_t) + x_n g'(\lambda \sigma_t) + \frac{1}{2} x_n^2 g'' (c_1), \quad \text{ and }  \label{equ:MVT1}\\
	g'(\lambda \sigma_t + x_n) & = & g'(\lambda \sigma_t ) + x_n g''(c_2).
	\label{equ:MVT2}
\end{eqnarray}
Hence, for $n \in \tilde{I}(t)$ we can rewrite
\begin{eqnarray*}
	F_n(t- \tau_n) & = & \Big( g(\lambda \sigma_t ) + x_n g'(\lambda \sigma_t) + \frac{1}{2} x_n^2 g'' (c_1) + f_n g'(\lambda \sigma_t) + x_n f_n g''(c_2) \Big)\big( t - \sigma_t  -s_n \sqrt{\sigma_t } \big) \\	
	& = & g(\lambda \sigma_t)(t-\sigma_t ) - g(\lambda \sigma_t) s_n \sqrt{\sigma_t } + \lambda s_n\sqrt{\sigma_t } g' (\lambda \sigma_t)(t - \sigma_t ) - \lambda s^2_n \sigma_t  g'(\lambda \sigma_t) \\
	&  &  + \Big(\sfrac{1}{2} \log \sigma_t  - \lambda T_n\Big) g'(\lambda \sigma_t)\big(t - \sigma_t  - s_n\sqrt{\sigma_t }\big) + \sfrac{1}{2}x_n^2 g'' (c_1) (t - \sigma_t  - s_n \sqrt{\sigma_t })\\
	&  & +  f_n g'(\lambda \sigma_t) (t - \sigma_t ) - f_n g'(\lambda \sigma_t ) s_n \sqrt{\sigma_t } + f_n x_n g''(c_2)(t -\sigma_t  - s_n\sqrt{\sigma_t }). 
\end{eqnarray*}
Recall that by definition $g'(\lambda \sigma_t) (t - \sigma_t ) = \frac{g(\lambda \sigma_t)}{\lambda}$, and $g(\lambda \sigma_t) = 1 + o(1)$ when $t \rightarrow \infty$. We get 
\begin{equation*}
	f_n g'(\lambda \sigma_t)\big( t - \sigma_t \big)= \frac{f_n}{\lambda} + o(1)
	\quad{\cec{\text{almost surely when }t\uparrow\infty.}}
\end{equation*}
By definition $g(\lambda \sigma_t) \uparrow 1$ as $t \uparrow \infty$ and by Lemma \ref{lem:sigma}, we have $\sigma_t = o(t)$ and $g'(\lambda \sigma_t) \sim \frac{1}{\lambda t}$ (see Equations \eqref{equ:sigma} and \eqref{equ:c1}). {Furthermore, for $n \in \tilde{I}(t)$, Assumption {\bf (A1)} implies $T_n = T + \varepsilon_n \rightarrow T$, as $t \rightarrow \infty$.} Combining these with the fact that for all $n \in \tilde{I}(t)$, $|s_n| \le a$ and $|f_n| \le b$, we can show that for all $n \in \tilde{I}(t)$, {\cec almost surely} as $t \rightarrow \infty$, the following terms go to zero:
\begin{eqnarray*}
	&&\big|\lambda s_n^2 \sigma_t g'(\lambda \sigma_t)\big| \le \frac{a^2  \sigma_t}{t - \sigma_t} = \mathcal{O}\Big(\frac{\sigma_t}{t} \Big) = o(1), \\
	&& \Big|\Big(\sfrac{1}{2} \log \sigma_t  - \lambda T_n\Big) g'(\lambda \sigma_t)s_n\sqrt{\sigma_t }\Big| \le \Big|\sfrac{1}{2} \log \sigma_t  - \lambda T_n\Big| \frac{ a\sqrt{\sigma_t }}{\lambda(t - \sigma_t)} \sim \Big(\sfrac{1}{2} \log \sigma_t  - \lambda T\Big) \frac{a\sqrt{\sigma_t }}{\lambda(t - \sigma_t)} = o(1), \\
	&& \Big| f_n g'(\lambda \sigma_t ) s_n \sqrt{\sigma_t } \Big| \le  \frac{a b \:  \sqrt{\sigma_t}}{\lambda (t - \sigma_t)} =  \mathcal{O} \Big( \frac{ \sqrt{  \sigma_t}}{ t } \Big) = o(1).
\end{eqnarray*}
Therefore, {\cec almost surely as $t\uparrow\infty$},
\begin{eqnarray}
	F_n(t- \tau_n) & = & 	g(\lambda \sigma_t)(t - \sigma_t) + \frac{g(\lambda \sigma_t)}{2 \lambda} \log \sigma_t - g(\lambda \sigma_t) T_n + \frac{f_n}{\lambda} + \frac{1}{2} x_n^2 g''(c_1)(t - \sigma_t -s_n \sqrt{\sigma_t})  \nonumber \\
	& & +  f_n x_n g''(c_2) (t - \sigma_t - s_n \sqrt{\sigma_t}) + o(1). 	
	\label{equ:sum_for_F}
\end{eqnarray}
We can write $g(\lambda \sigma_t) = 1 + o(1)$, as $t \uparrow \infty$, and {\peter by Assumption {\bf (A1)}, $T_n = T+o(1)$  uniformly in $n\in \tilde{I}(t)$ where the $o(1)$-term converges to zero almost surely as $t\to\infty$.} Therefore we get 
\begin{equation}
	g(\lambda \sigma_n)\: T_n = T + o(1) \quad \text{ as $t \rightarrow \infty$.} 
	\label{equ:F_T_term}
\end{equation}
To simplify the last two terms in Equation \eqref{equ:sum_for_F}, we recall that Lemma \ref{lem:small_o} implies \smash{$g''(c_i) \sim \frac{-\varkappa}{ \lambda^2 \sigma_t t}$} for $i = 1,2$. Combing this with the fact that $\sigma_t \rightarrow \infty$ as $t \rightarrow \infty$ (by Lemma \ref{lem:sigma}), we get for 
$n \in \tilde{I}(t)$,
\begin{eqnarray}
	\big| f_n x_n g''(c_2) (t - \sigma_t - s_n \sqrt{\sigma_t}) \big| & = & \big| f_n \big( \lambda s_n \sqrt{\sigma_t }+ \sfrac{1}{2} \log \sigma_t - \lambda T_n  \big) g''(c_2) (t - \sigma_t - s_n \sqrt{\sigma_t}) \big| \nonumber \\
	& \le & \big| b \big( \lambda a \sqrt{\sigma_t }+ \sfrac{1}{2} \log \sigma_t - \lambda T_n  \big) g''(c_2) (t - \sigma_t + a \sqrt{\sigma_t}) \big| \nonumber \\
	& = & \Big| b \big( \lambda a \sqrt{\sigma_t }+ \sfrac{1}{2} \log \sigma_t - \lambda T   + o(1)\big) \frac{\varkappa(t - \sigma_t + a \sqrt{\sigma_t})}{\lambda^2 \sigma_t t} + o(1) \Big|  \nonumber \\
	& = & \mathcal{O} \Big(\frac{1}{\sqrt{\sigma_t}} \Big) = o(1), \quad \text{{\cec almost surely as } $t \rightarrow \infty$.}
	\label{equ:F_pen_term}
\end{eqnarray}
Consider the penultimate term of Equation \eqref{equ:sum_for_F}. By definition of $x_n$ we can rewrite it as follows,
\begin{align}
	\sfrac{1}{2} x_n^2 g''(c_1)(t - & \sigma_t - s_n \sqrt{\sigma_t})  = \sfrac{1}{2} \Big( \lambda s_n \sqrt{\sigma_t }+ \sfrac{1}{2} \log \sigma_t - \lambda T_n  \Big)^2  g''(c_1)(t - \sigma_t - s_n \sqrt{\sigma_t}) \notag\\
	& =  \sfrac{1}{2}\lambda^2 s_n^2 \sigma_t g''(c_1)(t - \sigma_t)  - \sfrac{1}{2}\lambda^2 s_n^3 \sigma^{\nicefrac{3}{2}}_t g''(c_1) \notag\\
	&  \hspace{0.4cm} + \sfrac{1}{2}\Big( 2 \lambda s_n \sqrt{\sigma_t }\Big(\sfrac{1}{2} \log \sigma_t - \lambda T_n\Big) +\Big(\sfrac{1}{2} \log \sigma_t - \lambda T_n  \Big)^2 \Big) g''(c_1)(t - \sigma_t - s_n \sqrt{\sigma_t}). \label{eq:sum_cec}
\end{align}
The first summand is the largest term and, by Lemma \ref{lem:small_o}, we get that, {\cec almost surely} as $t \rightarrow \infty$, 
\begin{equation*}
	 \frac{1}{2}\lambda^2 s_n^2 \sigma_t g''(c_1)(t - \sigma_t) = - \frac{\lambda^2 s_n^2 \sigma_t \varkappa (t - \sigma_t)}{2 \lambda^2 \sigma_t t} + o(1) =  - \frac{1}{2} s_n^2 \varkappa + o(1).
\end{equation*}
Using Lemmas \ref{lem:sigma}, \ref{lem:small_o} and Assumption {\bf (A1)}, 
we show that the second and third summands in Equation~\eqref{eq:sum_cec} go to zero {\cec almost surely} as $t \rightarrow \infty$. Indeed,
\begin{align*}
&	 \Big|  - \sfrac{1}{2}\lambda^2 s_n^3 \sigma^{\nicefrac{3}{2}}_t g''(c_1)+ \frac{1}{2}\Big( 2 \lambda s_n \sqrt{\sigma_t }\Big(\sfrac{1}{2} \log \sigma_t - \lambda T_n\Big) +\Big(\sfrac{1}{2} \log \sigma_t - \lambda T_n  \Big)^2 \Big) g''(c_1)(t - \sigma_t - s_n \sqrt{\sigma_t}) \Big| \\
	& \hspace{0.1cm} \le  \Big | - \sfrac{1}{2}\lambda^2 a^3 \sigma^{\nicefrac{3}{2}}_t g''(c_1)+ \frac{1}{2}\Big( 2 \lambda a \sqrt{\sigma_t }\Big(\sfrac{1}{2} \log \sigma_t - \lambda T_n\Big) +\Big(\sfrac{1}{2} \log \sigma_t - \lambda T_n  \Big)^2 \Big) g''(c_1)(t - \sigma_t + a \sqrt{\sigma_t}) \Big | \\
	& \hspace{0.1cm} = \Big|  \sfrac{1}{2}a^3 \sqrt{\sigma_t} \sfrac{\varkappa}{ t} - \sfrac{1}{2}\Big( 2 \lambda a \sqrt{\sigma_t }\Big(\sfrac{1}{2} \log \sigma_t - \lambda T + o(1) \Big) +\Big(\sfrac{1}{2} \log \sigma_t - \lambda T + o(1)\Big)^2 \Big) \sfrac{\varkappa (t - \sigma_t +a \sqrt{\sigma_t})}{\lambda^2 \sigma_t t}  + o(1) \Big| \\
	& \hspace {0.1 cm} = \mathcal{O}\Big(\frac{t \sqrt{\sigma_t} \log \sigma_t}{\sigma_t t} \Big) = o(1),
\end{align*}
{\cec almost surely when $t\to\infty$.}
Therefore, for all $n \in \tilde{I}(t)$, we have 
\begin{equation}
	\sfrac{1}{2} x_n^2 g''(c_1)(t - \sigma_t - s_n \sqrt{\sigma_t}) = - \sfrac{1}{2} s_n^2 \varkappa + o(1), \quad \text{a.s.\ as $t \rightarrow \infty$.}
	\label{equ:F_final_term}
\end{equation}	
Combining \eqref{equ:F_T_term}, \eqref{equ:F_pen_term} and \eqref{equ:F_final_term}, Equation \eqref{equ:sum_for_F} becomes
\begin{eqnarray*}
	F_n(t- \tau_n) & = & g(\lambda \sigma_t)(t - \sigma_t) + \frac{g(\lambda \sigma_t)}{2 \lambda} \log \sigma_t - T - \frac{1}{2}s_n^2\varkappa + \frac{1}{\lambda}f_n +  o(1),
\end{eqnarray*}
and thus 
\begin{eqnarray*}
	- \gamma F_n(t - \tau_n) 
%& = & -\gamma g(\lambda \sigma_t)(t - \sigma_t ) - \frac{\gamma g(\lambda \sigma_t)}{2 \lambda} \log \sigma_t  + \gamma T + \frac{\gamma}{2 } \varkappa s_n^2  - \frac{\gamma}{\lambda}  f_n + o(1)\\
	& = & -\gamma g(\lambda \sigma_t ) (t - \sigma_t ) - a_1 g(\lambda \sigma_t) \log \sigma_t  + \gamma T  +  a_2 s_n^2  - a_3 f_n  + o(1),
\end{eqnarray*}
{\cec almost surely as $t\to\infty$,}
where $a_1 =  \nicefrac{\gamma}{2\lambda}$, $a_2 = \nicefrac{\gamma\varkappa}{2}$ and $a_3 = \nicefrac{\gamma}{\lambda} $. Rearranging we get Equation~\eqref{equ:exponents}. \\

\noindent Substituting Equation~\eqref{equ:exponents} into \eqref{equ:pgs} we get
	\begin{eqnarray*}
		\bigg| \int f \di \Psi_t \circ \phi^{-1} & - & \int f \di \Gamma_t \bigg| \le c_L \sum_{n \in \tilde{I}(t)} Z_n(t) \e^{-\gamma F_n(t - \tau_n)} \e^{-a_2 s_n^2+a_3 f_n} \big| 1 - \e^{o(1)}\big|.
	\end{eqnarray*}	
Since the almost-sure $o(1)$-term is uniform in $n\in \tilde{I}(t)$, we get
	\begin{eqnarray*}
		\bigg| \int f \di \Psi_t \circ \phi^{-1} & - & \int f \di \Gamma_t \bigg| 
		= o\bigg( \sum_{n \in \tilde{I}(t)} Z_n(t) \e^{-\gamma F_n(t - \tau_n)} \e^{-a_2 s_n^2+a_3 f_n} \bigg),
	\end{eqnarray*}
	{\cec almost surely as $t\to\infty$.}
Furthermore, by definitions of $\Psi_t \circ \phi^{-1}$ and $\tilde{I}(t)$, 
	\begin{eqnarray*}
		\sum_{n \in \tilde{I}(t)} Z_n(t) \e^{-\gamma F_n(t - \tau_n)} \e^{-a_2 s_n^2+a_3 f_n} & = & \int_0^\infty  \id_{|s| \le a} \id_{|f| \le b} \: z \: \di \Psi_t \circ \phi^{-1}(s, \: f \:, z) \\
		& \rightarrow & \int_0^\infty  \id_{|s| \le a} \id_{|f| \le b} \: z \: \di \text{PPP}\big( \zeta^* \circ \phi^{-1}\big)\\
& = & \int_0^\infty  \id_{|s| \le a} \id_{|f| \le b}\: z\: \di \text{PPP}(\zeta),
	\end{eqnarray*}
{\cec in distribution} as $t \rightarrow \infty$, by Proposition \ref{prop:cv_gamma2}, 
{\cec since the function $(s,f,z)\mapsto \id_{|s| \le a} \id_{|f| \le b} \: z$ has compact support in $(-\infty,\infty)\times (-\infty, \infty]\times [0,\infty]$.} 
%\pagebreak[3]
%Since $\zeta$ is the image of $\zeta^*$ under $\phi$, so we can conclude
%	\begin{eqnarray*}
%		\sum_{n \in \tilde{I}(t)} Z_n(t) \e^{-\gamma F_n(t - \tau_n)} \e^{-a_2 s_n^2+a_3 f_n} & \rightarrow & \int_0^\infty  \id_{|s| \le a} \id_{|f| \le b}\: z\: \di \text{PPP}(\zeta).
%	\end{eqnarray*}
Recalling the definition of $\zeta$, and substituting $w = z \e^{s^2 a_2 - f a_3}$, we get
	\begin{align*}
		\E \bigg [ \int_0^\infty z \id_{|s| \le a} \id_{|f| \le b} \, \di \text{PPP}(\zeta) \bigg] & =  \int_{-a}^a \int_{-b}^b \int_0^\infty  \lambda \e^{-f} \e^{s^2 a_2 - fa_3} z \nu\big(z \e^{s^2 a_2 - f a_3}\big) \: \di s \: \di f \:\di z \\
		& =  \int_{-a}^a \lambda \e^{-s^2a_2} \: \di s \int_{-b}^b \e^{(a_3 - 1) f} \: \di f \int_0^\infty w \nu(w) \: \di w \\
		& =  \lambda \sqrt{\frac {\pi}{ a_2}} \: \text{erf}\big(a \sqrt{a_2}\big) \frac{1}{a_3 - 1} \Big( \e^{(a_3 - 1) b} - \e^{-(a_3 - 1)b} \Big) \int_0^\infty w \nu(w) \: \di w =:  C_1,
	\end{align*}
where $\text{erf}(x) = \frac{1}{\sqrt{\pi}} \int_{-x}^x \e^{-t^2} \di t$. Note that $C_1 < \infty$ since $\int_0^\infty w \nu(w) \: \di w < \infty$ by Assumption {\bf (A3)}. {\cec This implies that $\sum_{n \in \tilde{I}(t)} Z_n(t) \e^{-\gamma F_n(t - \tau_n)} \e^{-a_2 s_n^2+a_3 f_n}$ converges in distribution to an almost surely finite random variable, and thus}
	\begin{equation*}
		\bigg| \int f \di \Psi_t \circ \phi^{-1} -  \int f \di \Gamma_t \bigg| \to 0, \quad \text{{\cec in distribution, and thus in probability} as $t \rightarrow \infty$,}
	\end{equation*}
which means that the point process $\Gamma_t$ is asymptotically equivalent to $\Psi_t \circ \phi^{-1}$. {\cec Note that, by a change of variable, $\int f \mathrm d\Psi_t\circ \phi^{-1} = \int f\circ \phi\,\mathrm d\Psi_t$. For all functions $f$ continuous and compactly supported, since $\phi$ is continuous, the function $f\circ \phi$ is also continuous and compactly supported, implying that
\[\int f \mathrm d\Psi_t\circ \phi^{-1}
= \int f\circ \phi\,\mathrm d\Psi_t 
\to \int f\circ \phi\,\mathrm d\mathrm{PPP}(\zeta^*)
= \int f\mathrm d\mathrm{PPP}(\phi(\zeta^*)),\]
where we have used Proposition~\ref{prop:cv_gamma2}.
One can check that $\zeta$ is the image of $\zeta^*$ by $\phi$.
This implies that $\Psi_t\circ \phi^{-1}$ converges vaguely in distribution on $(-\infty, \infty) \times (-\infty, \infty] \times [0, \infty] $ to $\text{PPP}(\zeta)$, implying that$\Gamma_t$ does too, since the two point processes are asymptotically equivalent.
}
%Recall that $\zeta$ is the image of $\zeta^*$ by the same continuous function, $\phi$. By , . 
\end{proof}

%%%%%%%%%%%%%%%%%%%%%%%%%%%%%%%%%%%%%%%%%%%%%%%%%
\section{Compactification and completion of the proofs} \label{sec:compactification}
%{Proof of Theorem \ref{theo:theo_main}}
%%%%%%%%%%%%%%%%%%%%%%%%%%%%%%%%%%%%%%%%%%%%%%%%%
To deduce Theorem~\ref{theo:theo_main} from Proposition~\ref{prop:cv_gamma}, one has to control the contribution of the point process near the closed boundaries of $[-\infty, \infty] \times [-\infty, \infty] \times (0, \infty]$. We prove that the families that are born outside of the main window, namely the ones that are unfit or born late, are too small to contribute in the limit. We first consider families which are born either early or late. We then show the negligibility of families lying under the main window by looking at families with small fitness {\cec (see Figure~\ref{fig})}.

\subsection{Contribution of young and old families} \label{sec:young}
\begin{lem}[Contribution of young and old families]
	For every $\eta > 0$ and $\varepsilon > 0$ there exists $v >1$ such that, for all sufficiently large $t$, we have
		\begin{equation*}
			\pr \bigg( \max_{n \in \I_t(v)} \e^{-\gamma g(\lambda \sigma_t)(t - \sigma_t ) - a_1 g(\lambda \sigma_t) \log \sigma_t + \gamma T} Z_n(t) \ge \varepsilon \bigg) \le \eta,
		\end{equation*}
	where $\I_t(v) := [0, n_t(-v)] \cup [n_t(v), \infty]$, $n_t(\pm v):=\exp\big\{\lambda\big(\sigma_t  \pm v \sqrt{\sigma_t } \big)\big\}$. 
		\label{lem:lem6_3}
\end{lem}
\begin{proof}
	Let $\eta, \varepsilon > 0$. For all $n \ge 1$, we define 
\begin{equation*}
	A_n:= \max_{u\ge\tau_n} Z_n(u) \e^{-\gamma F_n (u - \tau_n)}.
\end{equation*}
If there exists $t \ge \tau_n$ such that
\begin{eqnarray}
	\label{Zn}
 	Z_n(t) \ge \varepsilon \e^{\gamma g(\lambda \sigma_t) (t - \sigma_t ) + a_1 g(\lambda \sigma_t) \log \sigma_t - \gamma T},
\end{eqnarray}
then we get,
\begin{eqnarray}
	\label{An}	
 	A_n \ge Z_n(t) \e^{-\gamma F_n(t - \tau_n)}\ge\varepsilon \e^{ \gamma g(\lambda \sigma_t)(t - \sigma_t ) + a_1 g(\lambda \sigma_t) \log \sigma_t - \gamma T - \gamma F_n(t-\tau_n)}.  
\end{eqnarray}
By Assumption {\bf (A1)}, we have $\tau_n = \frac{1}{\lambda} \log n + T_n$, where $T_n = T + \varepsilon_n$; therefore \eqref{An} is equivalent to 
\begin{eqnarray*}
A_n \ge c_{n,t} \e^{- \gamma (1-F_n) T  + \gamma F_n \varepsilon_n },
\end{eqnarray*}
where we have set 
\begin{equation*}
	c_{n,t} := \varepsilon \exp\big(\gamma g(\lambda \sigma_t) - \gamma F_n) t + (\gamma F_n - \gamma g(\lambda \sigma_t)) \sigma_t  + a_1 g(\lambda \sigma_t)\log \sigma_t - \gamma F_n(\sigma_t  - \sfrac{1}{\lambda}\log n) \big).
\end{equation*}
Hence,
\begin{eqnarray*}
 	 && \pr \bigg( \max_{n \in \I_t(v)} Z_n (t)  \ge  \varepsilon \e^{\gamma g(\lambda \sigma_t)(t-\sigma_t ) + a_1 g(\lambda \sigma_t)\log \sigma_t - \gamma T}  \bigg)  \le \pr\bigg (\bigcup_{n \in \I_t(v)} \Big\{A_n \ge c_{n,t} \e^{- \gamma (1-F_n) T  + \gamma F_n \varepsilon_n }\Big\}\bigg). 
\end{eqnarray*}
Moreover, for any $y>0$, we have
%	\begin{eqnarray*}
%	&&	\pr\bigg (\bigcup_{n \in \I_t(v)} \Big\{A_n\ge c_{n,t}\e^{-\gamma (1-F_n) T + \gamma F_n \varepsilon_n}\Big\}\bigg ) \\
%	&& \hspace{0.25cm} \le \pr \bigg (\bigcup_{n \in \I_t(v)} \Big\{A_n \ge  c_{n,t}\e^{-\gamma (1-F_n) y_1 - \gamma F_n y_2}\Big\}\bigg )  + \pr (|T| \ge y_1) + \pr \Big(\sup_{n \in \I_t(v)} |\varepsilon_n| \ge y_2 \Big).
%	\end{eqnarray*}
%In particular, for $y=y_1=y_2>0$, this yields
\begin{align*}
	\pr\bigg(\bigcup_{n \in \I_t(v)} \Big\{ A_n\ge c_{n,t} & \e^{-\gamma (1-F_n) T +  \gamma F_n \varepsilon_n} \Big\}\bigg) \\ & \le \sum_{n \in \I_t(v)} \pr \big( A_n \ge  c_{n,t}\e^{-\gamma y} \big) + \pr (|T| \ge y) + \pr \Big(\sup_{n \in \I_t(v)}|\varepsilon_n| \ge y \Big).
\end{align*}
Since $\varepsilon_n\to0$ almost surely and $|T|$ is finite, we can fix $y > 0 $ large enough, such that $\pr(|T| \ge y) \le \frac \eta 3$ and $\pr(\sup_{n \in \I_t(v)} |\varepsilon_n| \ge y ) \le \frac \eta 3$.
Consider 
\begin{equation*}
	S := \sum_{n \in \I_t(v)} \pr \big( A_n \ge  c_{n,t}\e^{-\gamma y}\big) =  \sum_{n \in \I_t(v)} \E \big[\pr( A_n \ge  c_{n,t}\e^{-\gamma y}| {(F_m)_{m\in \N}})\big].
\end{equation*} 
% P here 1/2 must be turned into eta
By Assumption~{\bf (A4)}, $\pr(A_n \ge u | {(F_m)_{m \in \N}}) \le c_0 \e^{-\eta u}$, so we get
	\begin{eqnarray*}
		S & \le & c_0 \sum_{n \in \I_t(v)} \E \Big[\exp \Big\{ -\eta \varepsilon \e^{(\gamma g(\lambda \sigma_t) - \gamma F_n) t + (\gamma F_n -\gamma g(\lambda \sigma_t)) \sigma_t  + a_1 g(\lambda \sigma_t) \log \sigma_t -\gamma y - \gamma F_n(\sigma_t  - \frac{1}{\lambda} \log n)} \Big\} \Big] \\
		& \le & c_0 \int_{\I_t(v)} \E \Big[ \exp \Big\{ -\eta \varepsilon \e^{( \gamma g(\lambda \sigma_t) - \gamma F) t + (\gamma F - \gamma g(\lambda \sigma_t)) \sigma_t  + a_1 g(\lambda \sigma_t) \log \sigma_t - \gamma y   - \gamma F(\sigma_t  -\frac{1}{\lambda} \log x)} \Big\} \Big] \di x,
	\end{eqnarray*}
where $F$ is a random variable of law $\mu$. Let $x = \exp\big\{\lambda (\sigma_t  + w \sqrt{\sigma_t }) \big\}$, therefore we can write
\begin{align*}
	S & \le   c_0 \int_{|w| \ge v} \!\!\!\lambda \sqrt{\sigma_t } \e^{\lambda(\sigma_t  + w \sqrt{\sigma_t })} \E \Big[ \exp \big\{\!\!-\eta \varepsilon 
	\e^{(\gamma g(\lambda \sigma_t) - \gamma F)t + (\gamma F - \gamma g(\lambda \sigma_t))\sigma_t  + a_1 g(\lambda \sigma_t)\log \sigma_t - \gamma y + \gamma F w \sqrt{\sigma_t }} \big\} \Big] \di w \\
	& \le c_0 \int_{|w| \ge v} \lambda \sqrt{\sigma_t } \e^{\lambda(\sigma_t  + w \sqrt{\sigma_t } )} \\
	&  \, \, \,\times \int_0^1 \pr \Big(\exp \Big\{-\eta \varepsilon 
	\e^{(\gamma g(\lambda \sigma_t) - \gamma F)t + (\gamma F - \gamma g(\lambda \sigma_t))\sigma_t  + a_1 g(\lambda \sigma_t) \log \sigma_t - \gamma  y + \gamma F w \sqrt{\sigma_t }} \Big\} \ge x \Big) \di x \: \di w \\
	&  =:  c_0 \int_{|w| \ge v} \lambda \sqrt{\sigma_t } \e^{\lambda(\sigma_t  + w \sqrt{\sigma_t } )} \int_0^1 P(x) \di x  \:  \di w.
\end{align*}
Letting $\tilde{x}_0 = 1 + w\sigma_t^{\nicefrac{-1}{2}}$ and substituting into $\mu(x,1) = \exp\{-m(x)\}$, we get
\begin{align*}
	& P(x)  =  \pr \Big(\exp \Big\{-\eta\varepsilon \e^{\gamma g(\lambda \sigma_t) t - \gamma g(\lambda \sigma_t)\sigma_t + a_1 g(\lambda \sigma_t) \log \sigma_t  - \gamma y  + \gamma F(-t + \tilde{x}_0 \sigma_t)} \Big\} \ge x \Big) \\
	& =  \pr \Big(F \ge \big(\gamma t - \gamma \tilde{x}_0 \sigma_t\big)^{-1}\Big(\gamma g(\lambda \sigma_t) t - \gamma g(\lambda \sigma_t) \sigma_t + a_1 g(\lambda \sigma_t) \log \sigma_t  - \gamma  y -   \log \big( -\sfrac{1}{\eta\varepsilon} \log x \big) \Big)  \Big) \\
	& =  \exp \Big \{ -m \Big( (\gamma t - \gamma \tilde{x}_0 \sigma_t)^{-1} \Big(\gamma g(\lambda \sigma_t) t - \gamma g(\lambda \sigma_t) \sigma_t + a_1 g(\lambda \sigma_t) \log \sigma_t  - \gamma y - \log \big( -\sfrac{1}{\eta\varepsilon} \log x\big) \Big) \Big) \Big \} \\
	& =  \exp \Big \{ -m \Big(\Big(1- \tilde{x}_0 \sfrac{\sigma_t}{t}\Big)^{-1} \Big(g(\lambda \sigma_t) - \sfrac{g(\lambda \sigma_t)}{t} \sigma_t + \sfrac{a_1 g(\lambda \sigma_t)}{\gamma t} \log \sigma_t  - \sfrac{ y }{t}- \sfrac{1}{\gamma t}\log \big( -\sfrac{1}{\eta\varepsilon} \log x \big) \Big) \Big) \Big \}.
\end{align*}
We can approximate $P(x)$ by 
\begin{align*}
&P(x)  \\
&	= \exp \Big \{ -m \Big( \Big(1+ \tilde{x}_0 \sfrac{\sigma_t}{t} \Big) \Big(g(\lambda \sigma_t) - \sfrac{\gamma g(\lambda \sigma_t)}{\gamma t} \sigma_t + \sfrac{a_1 g(\lambda \sigma_t)}{\gamma t} \log \sigma_t - \sfrac{y }{t} - \sfrac{1}{\gamma t}\log \big( -\sfrac{1}{\eta\varepsilon}\log x \big)\Big) + \mathcal{O}\big(\sfrac{\sigma_t}{t}\big)^2  \Big) \Big \} \\
&	 = \exp \Big \{ -m \Big(  g(\lambda \sigma_t)
	+ \sfrac{w \sqrt{\sigma_t}}{t} g(\lambda \sigma_t)
	+ \sfrac{a_1 g(\lambda \sigma_t)}{\gamma t} \log \sigma_t - \sfrac{ y }{t} - \sfrac{1}{\gamma t}\log \big( -\sfrac{1}{\eta\varepsilon} \log x \big)  + \mathcal{O}\big(\sfrac{\sigma_t}{t}\big)^2  \Big) \Big \}. 
\end{align*}
Lemma \ref{lem:MVT} (Equation \eqref{equ:MVT3}) implies
\begin{eqnarray*}
	P(x) & = & \exp \Big \{ - m\big(g(\lambda \sigma_t)\big)  - m'\big( g(\lambda \sigma_t)\big)\Big( \sfrac{w \sqrt{\sigma_t}}{t} g(\lambda \sigma_t) + \sfrac{a_1 g(\lambda \sigma_t)}{\gamma t} \log \sigma_t - \sfrac{y }{t} - \frac{1}{\gamma t}\log \big( -\sfrac{1}{\eta\varepsilon}\log x \big) \Big) \\
	& & \hspace{1cm} - \frac{1}{2} m''(c_1)\Big( \frac{w \sqrt{\sigma_t}}{t} g(\lambda \sigma_t)+ \frac{a_1 g(\lambda \sigma_t)}{\gamma t} \log \sigma_t - \sfrac{ y }{t} - \sfrac{1}{\gamma t}\log \big( -\sfrac{1}{\eta\varepsilon} \log x \big) \Big)^2 \Big \}.
\end{eqnarray*}
Recall that $m(g(\lambda \sigma_t)) = \lambda \sigma_t$ and $m'(g(\lambda \sigma_t)) = \frac{\lambda (t -\sigma_t)} {g(\lambda \sigma_t)} $. Using Assumption {\bf (A5.3)}, one can show that $m''(g(\lambda \sigma_t)) \sim \frac{\lambda \varkappa t^2}{\sigma_t (g(\lambda \sigma_t))^3}$ as $t$ goes to infinity. Therefore we get 
\begin{eqnarray*}
	P(x) & = & \exp \Big \{ -\Big( \lambda \sigma_t+ \sfrac{ \lambda(t - \sigma_t)}{g(\lambda \sigma_t)}\Big(\sfrac{w \sqrt{\sigma_t}}{t} g(\lambda \sigma_t) + \sfrac{g(\lambda \sigma_t)}{2 \lambda t} \log \sigma_t  - \sfrac{ y }{t} - \sfrac{1}{\gamma t}\log \big( -\sfrac{1}{\eta\varepsilon} \log x \big) \Big) \\
	&& \hspace{1cm} + \sfrac{\lambda \varkappa t^2}{2 \sigma_t (g(\lambda \sigma_t))^3} \Big(  \sfrac{w \sqrt{\sigma_t}}{t} g(\lambda \sigma_t) + \sfrac{g(\lambda \sigma_t)}{2 \lambda t} \log \sigma_t - \sfrac{ y }{t} - \sfrac{1}{\gamma t}\log \big( -\sfrac{1}{\eta\varepsilon}\log x \big) \Big)^2  \Big) + o(1)\Big \}\\
	& = & \sigma_t^{-1/2}\exp \Big \{ -\lambda \sigma_t - \lambda w \sqrt{\sigma_t} + \sfrac{\lambda y}{g(\lambda \sigma_t)} + \sfrac{\lambda}{ \gamma g(\lambda \sigma_t)} \log \big(-\sfrac{1}{\eta\varepsilon} \log x \big) - \sfrac{\lambda \varkappa w^2}{2 g(\lambda \sigma_t)} + o(1) \Big \}.
\end{eqnarray*}
Hence we get
\begin{eqnarray*}
	S & \le & c_0 \int_{|w| \ge v} \lambda \sqrt{\sigma_t } \e^{\lambda(\sigma_t  + w \sqrt{\sigma_t })}
	\sigma_t^{-\frac{1}{2}}\e^{ - \lambda (\sigma_t  + w \sqrt{\sigma_t })  + \frac{\lambda y}{g(\lambda \sigma_t)} - \frac{\lambda \varkappa w^2}{2 g(\lambda \sigma_t)}+ o(1)}	
	\int_0^1 \e^{\frac{\lambda}{ \gamma g(\lambda \sigma_t)} \log \big(-\sfrac{1}{\eta\varepsilon} \log x \big)} \dx \: \di w \\
	&\le & c_0  \int_{|w| \ge v} \lambda \e^{\frac{\lambda y}{g(\lambda \sigma_t)} -\frac{\lambda \varkappa  w^2}{2 g(\lambda \sigma_t)} + o(1)} \Gamma \Big(\sfrac{\lambda}{\gamma g(\lambda \sigma_t)} + 1 \Big) \, \di w 
	 =  \mathcal{O} \bigg( \int_{|w| \ge v}  \exp \Big \{ \frac{\lambda y}{g(\lambda \sigma_t)} -\frac{\lambda \varkappa w^2}{2 g(\lambda \sigma_t)} \Big\} \di w \bigg),
\end{eqnarray*}
which goes to $0$ as $v$ goes to infinity, uniformly for all $t\geq 1$.
\end{proof}
%%%%%%%%%%%%%%%%%%%%%%%%%%%%%%%%%%%%%%%%%%%
\subsection{Contribution of unfit families} \label{sec:unfit}
\begin{lem}[Negligibility of families with small fitnesses]{}
For every $\eta > 0$ and $\varepsilon > 0$, there exists $\kappa>0$ such that for all sufficiently large $t$, we have 
	\begin{equation*}
		\pr \bigg( \max_{n \le M(t)}   \id \Big \{  \frac{F_n - g \big(\log(n \sqrt{\sigma_t })\big)}{g'\big(\log(n \sqrt{\sigma_t })\big)} \le - \kappa \Big \} \e^{-\gamma g(\lambda \sigma_t)(t - \sigma_t ) - a_1 g(\lambda \sigma_t) \log \sigma_t +\gamma T} Z_n(t) \ge \varepsilon \bigg) \le \eta. 
	\end{equation*} 
	\label{lem:lem6_2}
\end{lem}

\begin{proof}
Let $\varepsilon, \eta > 0$ and $\kappa > 0$. We analyse the event that there exists a family with fitness at most
\begin{equation*}
	f_n(\kappa) := g\big(\log(n\sqrt{\sigma_t })\big) -\kappa g'\big(\log(n\sqrt{\sigma_t })\big)
\end{equation*}
and size at least $\varepsilon \exp \{\gamma g(\lambda \sigma_t)(t - \sigma_t ) + a_1 g(\lambda \sigma_t) \log \sigma_t - \gamma T \}$. Similarly to the proof of Lemma \ref{lem:lem6_3} we define, for all $n \ge 1$,
\begin{equation*}
	A_n:= \max_{u\ge\tau_n} Z_n(u) \e^{-\gamma F_n (u - \tau_n)},
\end{equation*}
and as before we define 
	\begin{equation*}
	c_{n,t} := \varepsilon \exp\big\{(\gamma g(\lambda \sigma_t) - \gamma F_n) t + (\gamma F_n - \gamma g(\lambda \sigma_t)) \sigma_t  + a_1 g(\lambda \sigma_t)\log \sigma_t - \gamma F_n(\sigma_t  - \sfrac{1}{\lambda}\log n)\big\}.
	\end{equation*}
It can be shown that 
	\begin{multline*}
		\qquad \pr \Big( \max_{n\le M(t)} \id_{F_n  \le  f_n(\kappa)} Z_n(t)   \ge  \varepsilon \e^{\gamma g(\lambda \sigma_t)(t-\sigma_t ) + a_1 g(\lambda \sigma_t) \log \sigma_t }  \Big)  \\
		\le \sum_{n =1 }^{\infty} \pr \big( A_n \id_{F_n \le f_n(\kappa)} \ge  c_{n,t}\e^{-\gamma y} \big) + \pr (|T| \ge y) + \pr \big(\sup_{n \in \N}|\varepsilon_n| \ge y \big),\qquad
	\end{multline*}
where $y > 0$ is large enough, so that $\pr (|T| \ge y)\le \frac \eta 3$ and $\pr (\sup_{n \in \N}|\varepsilon_n| \ge y ) \le \frac \eta 3$. Set
\begin{equation*} 
	S := \sum_{n =1}^\infty \pr ( A_n  \id_{F_n \le f_n(\kappa)} \ge  c_{n,t}\e^{-\gamma y}) = \sum_{n =1}^\infty \E \big[\id_{F_n \le f_n(\kappa)}  \pr ( A_n  \ge  c_{n,t}\e^{-\gamma y}\mid {(F_m)_{m\in \N}}) \big].
\end{equation*} 
% P here 1/2 must be turned into eta
By Assumption~{\bf (A4)}, $\pr(A_n \ge u \mid {(F_m)_{m \in \N}}) \le c_0 \e^{-\eta u}$, which implies
	\begin{eqnarray*}
		S & \le & c_0 \sum_{n=1}^\infty \E \Big[ \id_{F_n \le f_n(\kappa)} \exp \Big\{ -\eta \varepsilon \e^{(\gamma g(\lambda \sigma_t) - \gamma F_n) t + (\gamma F_n -\gamma g(\lambda \sigma_t)) \sigma_t  + a_1 g(\lambda \sigma_t) \log \sigma_t -\gamma y - \gamma F_n(\sigma_t  - \frac{1}{\lambda} \log n)} \Big\} \Big] \\
		& \le & c_0 \int_0^\infty \E \Big[ \id_{F \le f_x(\kappa)} \exp \Big\{ -\eta \varepsilon \e^{(\gamma g(\lambda \sigma_t) - \gamma F) t + (\gamma F -\gamma g(\lambda \sigma_t)) \sigma_t  + a_1 g(\lambda \sigma_t) \log \sigma_t -\gamma y - \gamma F(\sigma_t  - \frac{1}{\lambda} \log x)} \Big\} \Big] \di x,
	\end{eqnarray*}
where $F$ is a random variable of law $\mu$. {\cec We change variables by setting} $x = \exp\big\{\lambda (\sigma_t  + w \sqrt{\sigma_t }) \big\}$ and set $\hat{f}_w(\kappa) := f_{ \exp\{\lambda (\sigma_t  + w \sqrt{\sigma_t })\}} (\kappa)$. This yields
	\begin{eqnarray*}
	S & \le & c_0 \int_{-\infty}^\infty \lambda \sqrt{\sigma_t } \e^{\lambda(\sigma_t  + w \sqrt{\sigma_t })} \\
	& &  \hspace{1cm} \times \E \Big[ \id_{F \le \hat{f}_w(\kappa)} \exp \Big\{-\eta \varepsilon 
	\e^{(\gamma g(\lambda \sigma_t) - \gamma F)t + (\gamma F - \gamma g(\lambda \sigma_t))\sigma_t  + a_1 g(\lambda \sigma_t) \log \sigma_t -\gamma y + \gamma F w \sqrt{\sigma_t }} \Big\} \Big] \di w.
	\end{eqnarray*}
Denoting by $\tilde{x}_0 := 1 + w {\sigma_t }^{\nicefrac{-1}{2}}$ and 
\begin{eqnarray*}
	E&:=& \E \Big[ \id_{F \le \hat{f}_w(\kappa)} \exp \Big\{-\eta \varepsilon 
	\e^{(\gamma g(\lambda \sigma_t) - \gamma F)t + (\gamma F - \gamma g(\lambda \sigma_t))\sigma_t  + a_1 g(\lambda \sigma_t) \log \sigma_t -\gamma y + \gamma F w \sqrt{\sigma_t }} \Big\} \Big],
\end{eqnarray*}
we get
\begin{eqnarray*}
	E & = &  \int_0^1 \pr \Big(F \le \hat{f}_w(\kappa);  \: \exp \Big\{-\eta \varepsilon \e^{\gamma g(\lambda \sigma_t) t - \gamma g(\lambda \sigma_t) \sigma_t  +a_1 g(\lambda \sigma_t) \log \sigma_t -\gamma y + \gamma F(-t + \tilde{x}_0 \sigma_t )} \Big\}  \ge x \Big)  \dx  \\
	& = & \int_0^1 \pr \Big( \sfrac {\gamma g(\lambda \sigma_t) t - \gamma g(\lambda \sigma_t) \sigma_t  + a_1 g(\lambda \sigma_t) \log \sigma_t -\gamma y - \log \big( \sfrac{1}{\eta \varepsilon}\log \big( \frac{1}{x} \big)\big)}{\gamma (t - \tilde{x}_0 \sigma_t )} \le F \le \hat{f}_w(\kappa) \Big) \, \dx.
\end{eqnarray*}
{\cec This integral is zero if the lower bound in the probability is larger than the upper bound.}
Note that 
	\begin{eqnarray*}
		f_x & := & \sfrac {\gamma g(\lambda \sigma_t) t - \gamma g(\lambda \sigma_t) \sigma_t  + a_1 g(\lambda \sigma_t) \log \sigma_t -\gamma y - \log \big( \sfrac{1}{\eta \varepsilon}\log \big( \frac{1}{x} \big)\big)}{\gamma (t - \tilde{x}_0 \sigma_t )} \\
		& = & \Big(g(\lambda \sigma_t) - \sfrac{ g(\lambda \sigma_t)}{t} \sigma_t  + \sfrac{a_1 g(\lambda \sigma_t)}{\gamma t} \log \sigma_t -\sfrac yt - \sfrac{1}{\gamma t}\log \big( -\sfrac{1}{\eta\varepsilon}\log x\big) \Big) \Big( 1 + \sfrac{\tilde{x}_0 }{t} \sigma_t + \mathcal{O}\big(\sfrac{\sigma_t }{t}\big)^2 \Big)  \\
		& = & g(\lambda \sigma_t) - \big( g(\lambda \sigma_t)- g(\lambda \sigma_t) \tilde{x}_0 \big) \sfrac{\sigma_t }{t} + \sfrac{a_1 g(\lambda \sigma_t)}{\gamma t} \log \sigma_t -\sfrac yt  - \sfrac{1}{\gamma t}\log \big( -\sfrac{1}{\eta\varepsilon}\log x\big) + o \big( \sfrac{1}{t}\big). 
	\end{eqnarray*}
We have
%	\begin{eqnarray*}
		$E = \int_0^1 \big( \mu(f_x, 1) - \mu(\hat{f}_w(\kappa), 1)\big) \dx.$
%	\end{eqnarray*}
By Lemma \ref{lem:MVT} (Equation \eqref{equ:MVT3}),
\begin{eqnarray*}
	\mu\big(f_x,1\big) & = & \exp \Big \{ -m \Big(g(\lambda \sigma_t) - \big(g(\lambda \sigma_t) - g(\lambda \sigma_t) \tilde{x}_0 \big) \sfrac{\sigma_t }{t} + \sfrac{a_1 g(\lambda \sigma_t)}{\gamma t} \log \sigma_t -\sfrac yt  \\
	& & \hspace{1cm} -  \sfrac{1}{\gamma t}\log \big( -\sfrac{1}{\eta \varepsilon}\log x\big) + o \big( \sfrac{1}{t}\big) \Big) \Big \} \\
	& = & \exp \Big \{ -m \Big( g(\lambda \sigma_t) + g(\lambda \sigma_t) \sfrac{w \sqrt{\sigma_t}}{t} + \sfrac{g(\lambda \sigma_t)}{2 \lambda t} \log \sigma_t -\sfrac yt  - \sfrac{1}{\gamma t}\log \big( -\sfrac{1}{\eta \varepsilon}\log x\big) + o \big( \sfrac{1}{t}\big) \Big) \Big \} \\
	& = & \exp \Big \{-m \big( g(\lambda \sigma_t) \big) - m'\big( g(\lambda \sigma_t)\big)\Big(g(\lambda \sigma_t) \sfrac{w \sqrt{\sigma_t}}{t} + \sfrac{g(\lambda \sigma_t)}{2 \lambda t} \log \sigma_t -\sfrac yt  - \sfrac{1}{\gamma t}\log \big( -\sfrac{1}{\eta \varepsilon}\log x\big) \Big)  \\
	& &  \hspace{1cm} - \frac{1}{2}m''(c_3) \Big( g(\lambda \sigma_t) \sfrac{w \sqrt{\sigma_t}}{t} + \sfrac{ g(\lambda \sigma_t)}{2 \lambda t} \log \sigma_t -\sfrac yt  - \sfrac{1}{\gamma t}\log \big( -\sfrac{1}{\eta \varepsilon}\log x\big) \Big)^2 \Big \} \\
	& = & \exp \Big \{ - \lambda \sigma_t - \lambda w \sqrt{\sigma_t} - \log \sqrt{\sigma_t} + \sfrac{\lambda y}{g(\lambda \sigma_t)}+ \sfrac{\lambda}{\gamma g(\lambda \sigma_t) }\log \big( -\sfrac{1}{\eta \varepsilon}\log x\big) - \sfrac{\lambda \varkappa w^2}{2 g(\lambda \sigma_t)} + o(1) \Big \},
\end{eqnarray*}
since $m''(g(\lambda \sigma_t)) \sim \frac{\lambda \varkappa t^2}{ \sigma_t (g(\lambda \sigma_t))^3 }$, by Assumption {\bf (A5.3)}. Using Lemma \ref{lem:MVT}, (Equation \eqref{equ:MVT4}), and the fact that $m'(g(x))g'(x) = 1$ for all $x > 0$, we get
\begin{eqnarray*}
	\mu\big(\hat{f}_w(\kappa), 1\big) & = & \exp \Big\{- m \Big( g\big(\lambda \tilde{x}_0\sigma_t + \log \sqrt{\sigma_t}\big) - \kappa g' \big(\lambda \tilde{x}_0 \sigma_t + \log \sqrt{\sigma_t}  \big) \Big) \Big\} \\
	& = & \exp \Big \{ - m \Big (g \big( \lambda \tilde{x}_0 \sigma_t + \log \sqrt{\sigma_t} \big) \Big) + m' \Big( g \big( \lambda \tilde{x}_0 \sqrt{\sigma_t} + \log \sqrt{\sigma_t} \big) \Big) \kappa g' \big(\lambda \tilde{x}_0 \sigma_t + \log \sigma_t \big) \\
	& &  \hspace{1cm} -  \sfrac{1}{2}m''(c_4) \Big( \kappa g'\big( \lambda \tilde{x}_0 \sigma_t + \log \sqrt{\sigma_t} \big)\Big)^2 \Big \} \\
	& = & \exp \Big \{-\lambda \tilde{x}_0 \sigma_t - \log \sqrt{\sigma_t} + \kappa - \sfrac{\lambda \varkappa t^2}{2 \sigma_t( g(\lambda \sigma_t))^3}\Big(\kappa g'\big(\lambda \tilde{x}_0 \sigma_t + \log \sqrt{\sigma_t} \big) \Big)^2  \Big \}\\
	& = &  \exp \Big \{-\lambda \tilde{x}_0 \sigma_t - \log \sqrt{\sigma_t} + \kappa + o(1)  \Big \},
\end{eqnarray*}
as $t \rightarrow \infty$. This last equality holds in view of Lemma \ref{lem:MVT}, (Equation \eqref{equ:MVT5}), since \begin{eqnarray*}
	\frac{ \lambda \varkappa  t^2}{2 (g(\lambda \sigma_t))^3 \sigma_t}\Big(\kappa g'\big(\lambda \tilde{x}_0 \sigma_t + \log \sqrt{\sigma_t} \big) \Big)^2 & = &  \frac{ \lambda \varkappa  t^2}{2 (g(\lambda \sigma_t))^3 \sigma_t}\kappa^2 \Big(g'(\lambda \sigma_t) + g''(c_2)(\lambda w \sigma_t+ \log \sqrt{\sigma_t})\Big)^2 \\
	& = & \frac{\lambda \varkappa t^2}{2 (g(\lambda \sigma_t))^3 \sigma_t}\kappa^2 \Big( \frac{ g(\lambda \sigma_t)}{ \lambda (t-\sigma_t)} - \frac{\varkappa}{\sigma_t t} \big( \lambda w \sqrt{\sigma_t} + \log \sqrt{\sigma_t}\big) \Big)^2 \\
	& = & \mathcal{O}(\sigma_t^{-1}) = o(1).
\end{eqnarray*}
For $E > 0$, we need $\mu(f_x,1) > \mu(\hat{f}_w(\kappa),1) $, which holds if and only if 
\begin{equation*}
	x \le \exp \Big \{ -\frac{\varepsilon}{2} \exp \Big\{ \frac{\gamma}{\lambda} g(\lambda \sigma_t) \Big(\kappa + \frac{\lambda \varkappa  w^2}{2 g(\lambda \sigma_t)} -\frac{\lambda y}{g(\lambda \sigma_t)} \Big) \Big\} \Big \} =: f_1.
\end{equation*}
%Using the dominated convergence theorem, 
{Since $g(\lambda \sigma_t) \rightarrow 1$ as $t \rightarrow \infty$, we get }
	\begin{equation*}
		f_1 = \exp \Big \{ -\eta \varepsilon \exp \Big\{ \frac{\gamma}{\lambda}  \Big(\kappa + \frac{\lambda \varkappa w^2}{2} \Big)   - \lambda y + o(1) \Big\} \Big \}. 
	\end{equation*}
Hence we can rewrite $E$ as
\begin{eqnarray*}
	E & = & \big(1+o(1)\big)\e^{ -\lambda \sigma_t - \lambda w \sqrt{\sigma_t} - \log{\sqrt{\sigma_t}} } \int_0^{f_1} \Big( \exp \big \{ \sfrac{\lambda}{\gamma}\log \big( -\sfrac{1}{\eta \varepsilon}\log x\big) - \sfrac{\lambda \varkappa w^2}{2} + \lambda y \big \} - \e^\kappa \Big) \dx \\
	& = & \big(1+o(1)\big) \e^{ -\lambda \sigma_t - \lambda w \sqrt{\sigma_t} - \log{\sqrt{\sigma_t}} } \Big( \e^{\lambda y-\frac{\lambda \varkappa w^2}{2 }}  \Big ( \sfrac{1}{\eta \varepsilon}\Big)^{\frac{\lambda}{\gamma }} \int_0^{f_1} \big(\log \sfrac{1}{x} \big)^{\frac{\lambda}{\gamma }} \dx  - \int_0^{f_1} \e^\kappa \dx \Big) \\
	& = & \big(1+o(1)\big) \e^{ -\lambda \sigma_t - \lambda w \sqrt{\sigma_t} - \log{\sqrt{\sigma_t}} } \Big (\e^{\lambda y - \frac{\lambda \varkappa  w^2}{2 }}  \Big ( \sfrac{1}{\eta \varepsilon}\Big)^{\frac{\lambda}{\gamma }} \Gamma \Big( \sfrac{\lambda}{\gamma }+1, {\eta \varepsilon} \e^{\frac{\gamma}{\lambda} (\kappa + \frac{\lambda \varkappa w^2}{2})} \Big) \\
&& \phantom{lobberdiddiddldybumfiddldididldidodadadadldididoo}	-  \exp \big \{\kappa - {\eta \varepsilon} \e^{\frac{\gamma}{\lambda} (\kappa + \frac{\lambda \varkappa w^2}{2} -\lambda y)}\big \} \Big),
\end{eqnarray*}
where $\Gamma (s,x) = \int_x^\infty z^{s-1} \e^{-z} \di z$ is the upper incomplete gamma function. So we get
\begin{align*}
	S & \le  \big(1+o(1)\big) c_0 \int_{-\infty}^{\infty} \lambda \Big( \e^{\lambda y -\frac{\lambda \varkappa w^2}{2 }}  \big ( \sfrac{1}{\eta \varepsilon}\big)^{\frac{\lambda}{\gamma }} \Gamma \big( \sfrac{\lambda}{\gamma }+1, {\eta \varepsilon} \e^{\frac{\gamma}{\lambda}(\kappa + \frac{\lambda \varkappa w^2}{2 } -\lambda y)} \big) 
	 -  \exp \Big \{\kappa - {\eta \varepsilon} \e^{\frac{\gamma}{\lambda} (\kappa + \frac{\lambda \varkappa w^2}{2 } -\lambda y)}\Big \} \Big) \di w.
\end{align*}
Since $\frac{\Gamma(s,x)}{x^{s-1} \e^{-x}} \rightarrow 1$  as  $x \rightarrow \infty$,
as $\kappa \rightarrow \infty$ we have
\begin{equation*}
	\Gamma \big( \sfrac{\lambda}{\gamma }+1, {\eta \varepsilon} \e^{\frac{\gamma}{\lambda} (\kappa + \frac{\lambda \varkappa w^2}{2} - \lambda y)} \big)
	\sim \big( {\eta \varepsilon} \big)^{\frac{\lambda}{\gamma}} 
	\exp \Big \{\kappa + \sfrac{\lambda \varkappa w^2}{2} - \lambda y
	- {\eta \varepsilon} \e^{\frac{\gamma }{\lambda}(\kappa + \sfrac{\lambda \varkappa w^2}{2 } - \lambda y)}\Big \}, 
\end{equation*}	
and so $S \rightarrow 0$.
\end{proof}
%%%%%%%%%%%%%%%%%%%%%%%%%%%%%%%%%%%%%%%%%%%%%
\subsection{Contribution of old and fit families} \label{sec:old}

%%%%%%%%%%%%%%%%%%%%%%%%%%%%%%%%%%%%%%%%%%%%%
\begin{lem}[Absence of fit families above the ``window"]{}
For every $\varepsilon > 0$ and $\nu > 0$, there exists $\kappa > 0$ such that for all sufficiently large $t$, we have
\begin{equation}
	\pr \bigg( \max_{n \in \I_t^c(v)} \Big( \sfrac{F_n - g(\log(n \sqrt{\sigma_t }))}{g'(\log(n \sqrt{\sigma_t }))}  \Big) \le \kappa \bigg) \ge 1 - \varepsilon,
	\label{equ:fit_fam}
\end{equation}
where $\I_t^c(v) = [n_t(-v), n_t(v)]$, $n_t(\pm v):=\exp\big\{\lambda\big(\sigma_t  \pm v \sqrt{\sigma_t } \big)\big\}$.
\label{lem:fit_fam}
\end{lem}

\begin{proof}
	Let $\varepsilon, v > 0$ and $\kappa > 0$.	 We have\footnote{\peter For all $x,y\in\mathbb R$, we denote by $\prod_{n=x}^y$ and $\sum_{n=x}^y$ the product (resp.\ sum) over integers with $x\leq n\leq y$.}
	\begin{eqnarray*}
	\pr \bigg( \max_{n \in \I_t^c(v)} \Big( \sfrac{F_n - g (\log(n \sqrt{\sigma_t }))}{g'(\log(n \sqrt{\sigma_t }))}  \Big) \le \kappa \bigg) &=& \prod_{n_t(-v)}^{n_t(v)} \pr \Big(F_n \le  g \big(\log(n \sqrt{\sigma_t })\big) + \kappa g'\big(\log(n \sqrt{\sigma_t }) \big) \Big) \\
	& = & \prod_{n_t(-v)}^{n_t(v)}\bigg( 1 - \mu\Big(  g \big(\log(n \sqrt{\sigma_t })\big) + \kappa g'\big(\log(n \sqrt{\sigma_t }) \big) ,1 \Big) \bigg).
	\end{eqnarray*}
Using the fact that $\e^{-\mu(x,1)} = 1 - \mu(x,1) + o(\mu(x,1))$ when $x \rightarrow 1$, we get that when $t \rightarrow \infty$,
\begin{eqnarray*}
	\pr \bigg( \max_{n \in \I_t^c(v)} \Big( \sfrac{F_n - g (\log(n \sqrt{\sigma_t }))}{g'(\log(n \sqrt{\sigma_t }))}  \Big) \le \kappa \bigg) &\sim&  \exp \bigg\{ - \sum_{n_t(-v)}^{n_t(v)} \mu \Big(g \big(\log(n \sqrt{\sigma_t })\big) + \kappa g'\big(\log(n \sqrt{\sigma_t }) \big),1 \Big) \bigg\}.
\end{eqnarray*}
Recall that $\mu(x,1) = \e^{-m(x)}$, which implies that, {\cec as $t\to\infty$},
\begin{eqnarray*}
	\pr \bigg( \max_{n \in \I_t^c(v)} \Big( \sfrac{F_n - g (\log(n \sqrt{\sigma_t }))}{g'(\log(n \sqrt{\sigma_t }))}  \Big) \le \kappa \bigg) 
%&\sim&  \exp \bigg\{ - \sum_{n_t(-v)}^{n_t(v)} \e^{-m\Big( (g (\log(n \sqrt{\sigma_t })) + \kappa g'(\log(n \sqrt{\sigma_t }) ) \Big)} \bigg\}\\
	&\sim&  \exp \bigg\{ - \int_{n_t(-v)}^{n_t(v)} \e^{-m\big( g (\log(x \sqrt{\sigma_t })) + \kappa g'(\log(x \sqrt{\sigma_t }) ) \big)} \dx \bigg\}.
\end{eqnarray*}
Using the change of variables with $x = \e^{\lambda(\sigma_t + w \sqrt{\sigma_t})} = n_t(w)$, we get
\begin{align*}
	\pr \bigg( \max_{n \in \I_t^c(v)} \Big( & \sfrac{F_n - g (\log(n \sqrt{\sigma_t }))}{g'(\log(n \sqrt{\sigma_t }))}  \Big)  \le \kappa \bigg)  \\
	& \sim \exp \bigg\{ - \int_{-v}^v \e^{-m\big( g (\log(n_t(w) \sqrt{\sigma_t })) + \kappa g'(\log(n_t(w)\sqrt{\sigma_t }) ) \big)} \lambda \sqrt{\sigma_t} n_t(w) \dw \bigg\}.
\end{align*}
By the same technique as in Lemma 15(a), we get that there exists 
$$c_6 \in [g \big(\log(n_t(w) \sqrt{\sigma_t })\big), g \big(\log(n_t(w) \sqrt{\sigma_t })\big) + \kappa g'\big(\log(n_t(w)\sqrt{\sigma_t }) \big) ]$$ such that 
\begin{align*}
	 m\Big( g \big(\log(n_t(w) \sqrt{\sigma_t })\big) & + \kappa g'\big(\log(n_t(w)\sqrt{\sigma_t }) \big) \Big)\\
	&  = m\Big( g \big(\log(n_t(w) \sqrt{\sigma_t })\big) \Big)
	+\kappa m'\Big( g \big(\log(n_t(w) \sqrt{\sigma_t })\big) \Big) g'\big(\log(n_t(w)\sqrt{\sigma_t }) \big) \Big)\\ & \hspace{3cm} + \sfrac{1}{2} m''\big( c_6 \big) \Big(\kappa g'\big(\log(n_t(w)\sqrt{\sigma_t }) \big) \Big)^2 \\
	&  = \log(n_t(w) \sqrt{\sigma_t })\big) + \kappa  + \mathcal{O}(\sigma_t^{-1}) 
	 = \lambda(\sigma_t + w \sqrt{\sigma_t}) + \log(\sqrt{\sigma_t}) + \kappa + {\peter \mathcal{O}(\sigma_t^{-1}) },
\end{align*}
where we have used that $m(g(x)) = x$ and hence $m'(g(x))g'(x) = 1$ for all $x>0$. We also used the fact that $m''( c_6) (\kappa g'(\log(n_t(w)\sqrt{\sigma_t }) ))^2 \rightarrow 0$ as $t\rightarrow \infty$, by Assumption {\bf (A5.2)}. Therefore, the integral becomes
\begin{eqnarray*}
	&&  \int_{-v}^v \e^{-m\big( g (\lambda(\sigma_t + w \sqrt{\sigma_t})) + \kappa g'(\lambda(\sigma_t + w \sqrt{\sigma_t})) \big)} \sqrt{\sigma_t} \e^{\lambda(\sigma_t + w \sqrt{\sigma_t})}\dw \\
	 & & \hspace{1cm} =  \int_{-v}^v \e^{-\lambda(\sigma_t + w \sqrt{\sigma_t})  - \log (\sqrt{\sigma_t}) - \kappa - \mathcal{O}(\sigma_t^{-1})} \sqrt{\sigma_t} \e^{\lambda(\sigma_t + w \sqrt{\sigma_t})}\dw 
	 =  \int_{-v}^v \e^{-\kappa - \mathcal{O}(\sigma_t^{-1})} \dw  
	  \sim 2v\e^{-\kappa}. 
\end{eqnarray*}
Therefore, we get
\begin{eqnarray*}
	\pr \bigg( \max_{n \in \I_t^c(v)} \Big( \sfrac{F_n - g (\log(n \sqrt{\sigma_t }))}{g'(\log(n \sqrt{\sigma_t }))}  \Big) \le \kappa \bigg) &\sim& \exp \big \{ - 2v\e^{-\kappa}\big\} \rightarrow 1, \quad \text{as $\kappa \rightarrow \infty$.} \\[-1.3cm]
\end{eqnarray*}	
\end{proof}
\bigskip

\subsection{Proof of Theorem \ref{theo:theo_main}} \label{sec:main_proof}
%%%%%%%%%%%%%%%%%%%%%%%%%%%%%%%%%%%%%%%%%%%%%
%\begin{proof}[Proof of Theorem \ref{theo:theo_main}]
Let $\eta, \varepsilon > 0$. By Lemma \ref{lem:lem6_2} there exists $\kappa_1 = \kappa_1 (\varepsilon, \eta)$ such that
	\begin{equation*}
		\lim_{t\rightarrow \infty} \inf \pr \Big( \Gamma_t \big([-\infty, \infty] \times [-\infty, -\kappa_1] \times (\varepsilon, \infty]\big) = 0 \Big) \ge 1 - \eta.
	\end{equation*}
By Lemma \ref{lem:lem6_3} there exists $v = v(\varepsilon, \eta) > 1$ such that 
	\begin{equation*}
		\lim_{t\rightarrow \infty} \inf \pr \Big( \Gamma_t \big([-\infty, -v] \cup [v, \infty] \times [-\infty, \infty] \times (\varepsilon, \infty]\big) = 0 \Big) \ge 1 - \eta.
	\end{equation*}
By Lemma \ref{lem:fit_fam} there exists $\kappa_2 = \kappa_2 (\varepsilon, \eta)$ such that
	\begin{equation*}
		\lim_{t\rightarrow \infty} \inf \pr \Big( \Gamma_t \big([-v, v] \times [\kappa_2, \infty] \times (\varepsilon, \infty]\big) = 0 \Big) \ge 1 - \eta.
	\end{equation*}
Finally, Proposition \ref{prop:cv_gamma} gives that $\Gamma_t$ converges on $(-v,v) \times (- \kappa_1,\kappa_2) \times (\varepsilon, \infty]$ to the Poisson process with intensity measure $\zeta$. Combining these four facts and using that $\eta > 0$ is arbitrarily small, we get convergence on $[-\infty,\infty]\times[-\infty,\infty]\times(\varepsilon,\infty]$. As this holds for all $\varepsilon > 0$ the proof is complete.	
%\end{proof}

%%%%%%%%%%%%%%%%%%%%%%%%%%%%%%%%%%%%%%%%%%%%%%%%%
%\section{Further results}
%%%%%%%%%%%%%%%%%%%%%%%%%%%%%%%%%%%%%%%%%%%%%%%%%
\subsection{Proof of Corollary~\ref{cor:lim_fam}} \label{sec:proof_corollary}
%%%%%%%%%%%%%%%%%%%%%%%%%%%%%%%%%%%%%%%%%%%%%%%%%
%\begin{proof}[Proof of Corollary \ref{cor:lim_fam}] 
(i)\ {\peter Vague convergence in distribution of $\Gamma_t$ to $\mathrm{PPP}(\zeta)$ implies convergence in distribution of $\Gamma_t(B)$ to $\mathrm{PPP}(\zeta)(B)$ for compact sets $B$ with $\zeta(\partial B)=0$, see, e.g., \cite[Proposition~3.12]{Resnick}.}
We fix $x > 0$ and $B:=[-\infty, \infty ]\times[-\infty, \infty]\times[x, \infty]$. By Theorem~\ref{theo:theo_main}, we get that, as~$t \uparrow \infty$,
	\begin{eqnarray*}
		\sum^{M(t)}_{n=1}\id_B\Big( \sfrac{\tau_n - \sigma_t } {\sqrt{\sigma_t }}, \sfrac{F_n - g(\log (n \sqrt{\sigma_t } ))}{g'(\log(n \sqrt{\sigma_t }))}, \e^{-\gamma g(\lambda \sigma_t)(t-\sigma_t ) - a_1 g(\lambda \sigma_t) \log \sigma_t +\gamma T } Z_n(t) \Big) 
		\Rightarrow \text{Poisson}\Big(\int_B \di \zeta \Big),
	\end{eqnarray*}
since $B$ is a compact set. Hence, as $t \uparrow \infty$,
	\begin{align}
		\pr \Big( \e^{-\gamma g(\lambda \sigma_t)  (t-\sigma_t )  - a_1 g(\lambda \sigma_t) \log \sigma_t + \gamma T} & \max_{n\in \{1, ..., M(t) \}}Z_n(t) \ge x \Big)  \nonumber \\
\rightarrow  \pr \Big(\text{Poisson} \Big(\int_B \di \zeta\Big) \ge 1\Big) &
		 =  1 - \pr \Big(\text{Poisson} \Big(\int_B \di \zeta\Big) = 0 \Big) 
		 =  1 - \exp \Big(-\int_B \di \zeta \Big).
		\label{equ:lem2.3}
	\end{align}
Note that
	\begin{eqnarray}
		\int_B \di \zeta & = &  \int^\infty_{-\infty} \int^\infty_{-\infty} \int^\infty_x \lambda \e^{-f} \e^{s^2 a_2 - f a_3} \nu(z\e^{s^2 a_2 -f a_3}) \: \di z \;\! \df \;\! \di s \notag\\
		& = & \lambda   \int^\infty_{-\infty} \int^\infty_{-\infty} \int^\infty_{x \e^{s^2 a_2 - f a_3}} \e^{-f} \nu(w) \: \di w \; \df \;\!  \di s 
		 =   \lambda   \int^\infty_{-\infty} \int^\infty_0 \int^\infty_{ \frac{1}{a_3} (s^2 a_2 - \log \frac{w}{x})} \e^{-f} \nu(w) \: \df \;\!  \di w \;  \di s \notag\\
		& = &\lambda  \bigg ( \int^\infty_{-\infty} \e^{- \frac{a_2}{a_3} s^2 } \di s  \bigg) \bigg( \int^\infty_0 \nu(w)  \big(\sfrac{w}{x} \big)^\frac{1}{a_3} \di w \bigg) 
		 =  \lambda \sqrt{\pi \frac{a_3}{a_2}} \Big( \int_0^\infty \nu(w) w^{\frac{1}{a_3}} \di w \Big) x^{-\frac{1}{a_3}}. \label{maxintegral}
	\end{eqnarray}
Recall that $a_2 = \nicefrac{\gamma \varkappa}{2}$ and $a_3 = \frac{\gamma}{\lambda}$. Thus the right hand side in \eqref{equ:lem2.3} is $1 - \exp(-s^\eta x^{-\eta})$, 
%	\begin{equation*}
%		s^\eta=\sqrt{\sfrac{2 \pi \lambda}{ \varkappa}} \int_0^\infty \nu(w) w^{\frac{\lambda}{\gamma}} \: \di w , \quad \text{and } \eta = \frac{\lambda}{\gamma}. 
%	\end{equation*}	
{\peter for $\eta = \frac{\lambda}{\gamma}$. }
In summary, for all $x > 0$, we have
	\begin{eqnarray*}
		\pr \bigg( \e^{-\gamma g(\lambda \sigma_t)(t-\sigma_t ) - a_1 g(\lambda \sigma_t) \log \sigma_t +\gamma T} \max_{n \in \{1, ... , M(t)\}} Z_n(t) \le x \bigg) & \rightarrow & \e^{-\big(\frac{x}{s}\big)^{-\frac{\lambda}{\gamma}}}  = \pr\big( W \le x\big),
	\end{eqnarray*}
where $W \sim $ Fr\'echet $\big(\frac{\lambda}{\gamma},s\big)$.
\bigskip

(ii) {\peter We have 
\begin{align*}
\id_{\frac{S(t) - \sigma_t } {\sqrt{\sigma_t }} \geq x} & = \int   \id_{s\geq x}   \id_{\Gamma_t([-\infty,\infty]\times [-\infty, \infty] \times (z, \infty])=0} \, \di\Gamma_t(s, f, z),
\end{align*}
which is an almost everywhere vaguely continuous bounded function of $\Gamma_t$.
By Theorem~\ref{theo:theo_main}, we have
$$\lim_{t\to\infty} \pr\big( {\tfrac{S(t) - \sigma_t } {\sqrt{\sigma_t }}} \geq x \big)=
 \int   1_{s\geq x}   \pr\big({\rm PPP}(\zeta)([-\infty,\infty]\times [-\infty, \infty] \times [z, \infty]))=0\big) \, \di\zeta(s, f, z).$$}
Hence, the random variable $\frac{S(t) - \sigma_t } {\sqrt{\sigma_t }} $ converges to a random variable $U$ with density
			\begin{equation*}
				\int_{-\infty}^\infty \int_0^\infty \e^{-\zeta([-\infty, \infty] \times [-\infty, \infty] \times [z, \infty])} \zeta (s, \di f, \di z).
			\end{equation*}
	We recall from above that 
			\begin{equation*}
				\zeta ([-\infty, \infty]\times[-\infty, \infty]\times[z, \infty]) = \lambda \sqrt{\pi \frac{a_3}{a_2}} \Big( \int_0^\infty \nu(w) w^{\frac{1}{a_3}} \: \di w \Big)z^{-\frac{1}{a_3}} =: c_6 z^{-\frac{1}{a_3}}.
			\end{equation*}
We get, substituting $u = z \e^{s^2 a_2 - f a_3}$, 
	\begin{multline*}
		\int_{-\infty}^\infty \int_0^\infty \e^{-\zeta ([-\infty, \infty]\times[-\infty, \infty]\times[z, \infty])} \, \di \zeta(s,f,z) \\
		= \lambda \int_0^\infty \nu(u) \int_{-\infty}^\infty \exp \Big \{-f - c_6 u^{-\frac{1}{a_3}} \e^{s^2 \frac{a_2}{a_3}}\e^{-f}\Big \} \: \di f \;\! \di u \: \di s.
	\end{multline*}
Integrating with respect to $f$ and simplifying, gives us	
$$		\int_{-\infty}^\infty \int_0^\infty \e^{-\zeta ([-\infty, \infty]\times[-\infty, \infty]\times[z, \infty])} \di \zeta(s,f,z) 
		= \sfrac{\lambda}{c_6} \e^{- s^2 \frac{a_2}{a_3}} \: \di s \int_0^\infty \nu(u) u^{\frac{1}{a_3}} \: \di u 
	 	=  \frac{1}{\sqrt{\frac{2 \pi}{\lambda \varkappa}}} \e^{-s^2 \frac{\lambda \varkappa}{2}} \: \di s.$$
%\end{proof}

%%%%%%%%%%%%%%%%%%%%%%%%%%%%%%
\section{Appendix: an auxiliary lemma} 
%%%%%%%%%%%%%%%%%%%%%%%%%%%%%%
\label{sec:app}

For our proofs we need the following consequences of the mean value theorem. 

% P  1/2 has to be replaced by eta below
\begin{lem}[]
%	For $x > 0$ there exist $c_1, c_2 \in [\lambda \sigma_t, \lambda \sigma_t  + x]$, such that 
%\begin{eqnarray}
%	g  \big(\lambda \sigma_t + x \big) & = & g(\lambda \sigma_t) + x g'(\lambda \sigma_t) + \frac{1}{2} x^2 g'' (c_1) \text{ and }  \label{equ:MVT1}\\
%	g'(\lambda \sigma_t + x) & = & g'(\lambda \sigma_t ) + x g''(c_2).
%	\label{equ:MVT2}
%\end{eqnarray}
%\noindent Furthermore, 
For all $x \in [0,1]$, there exists $c_3 \in \big [g(\lambda \sigma_t), g(\lambda \sigma_t)+ \frac{w \sqrt{\sigma_t}}{t} g(\lambda \sigma_t) + \frac{a_1 g(\lambda \sigma_t)}{ \gamma t} \log \sigma_t - \frac{1}{\gamma t} \log \big( -\sfrac{1}{\eta\varepsilon} \log x \big) \big]$ such that 
\begin{equation} \label{equ:MVT3}
\begin{split}
	& m \Big(g(\lambda \sigma_t) +  \frac{w \sqrt{\sigma_t}}{t} g(\lambda \sigma_t) + \frac{a_1 g(\lambda \sigma_t)}{ \gamma t} \log \sigma_t - \frac{1}{\gamma t} \log \big( -\sfrac{1}{\eta\varepsilon} \log x \big) \Big) \\
	& \hspace{1cm} = m\big(g(\lambda \sigma_t)\big) + m'\big(g(\lambda \sigma_t)\big) \Big(  \frac{w \sqrt{\sigma_t}}{t} g(\lambda \sigma_t) + \frac{a_1 g(\lambda \sigma_t)}{ \gamma t} \log \sigma_t - \frac{1}{\gamma t} \log \big( -\sfrac{1}{\eta\varepsilon} \log x \big)\Big) \\
	&  \hspace{1cm} + \tfrac{1}{2}m''\big(c_3\big) \Big( \frac{w \sqrt{\sigma_t}}{t} g(\lambda \sigma_t) + \frac{a_1 g(\lambda \sigma_t)}{ \gamma t} \log \sigma_t - \frac{1}{\gamma t} \log \Big( -\sfrac{1}{\eta\varepsilon} \log x \Big)\Big)^2. 
\end{split}
\end{equation}	
For $\tilde{x}_0 > 0$ and $\kappa > 0$ there exist $c_4 \in \big[ g(\lambda \tilde{x}_0 \sigma_t + \log \sqrt{\sigma_t}), g(\lambda \tilde{x}_0 \sigma_t + \log \sqrt{\sigma_t}) - \kappa g'(\lambda \tilde{x}_0 \sigma_t + \log \sqrt{\sigma_t})\big]$, such that 	
	\begin{equation} \label{equ:MVT4}
	\begin{split}
	& m \Big( g(\lambda \tilde{x}_0 \sigma_t + \log \sqrt{\sigma_t}) - \kappa g'(\lambda \tilde{x}_0 \sigma_t + \log \sqrt{\sigma_t}) \Big)\\
	& \hspace{1cm} = m\Big( g(\lambda \tilde{x}_0 \sigma_t + \log \sqrt{\sigma_t}) \Big) +  m'\Big( g(\lambda \tilde{x}_0 \sigma_t + \log \sqrt{\sigma_t}) \big)\big(- \kappa g'(\lambda \tilde{x}_0 \sigma_t + \log \sqrt{\sigma_t}) \Big) \\
	&  \hspace{1cm} + \sfrac{1}{2}m''\big( c_4 \big)\Big(- \kappa g'(\lambda \tilde{x}_0 \sigma_t + \log \sqrt{\sigma_t}) \Big)^2. 
	\end{split}
	\end{equation}
And finally for all $w \in [-\infty, \infty]$ there exists $c_5 \in [\lambda\sigma_t, \lambda \sigma_t + \lambda w \sqrt{\sigma_t} + \log \sqrt{\sigma_t}]$ such that 	
	\begin{eqnarray}
		g'\big(\lambda \sigma_t + \lambda w \sqrt{\sigma_t} + \log \sqrt{\sigma_t} \big) = g' \big(\lambda\sigma_t \big) + g''\big( c_5 \big) \big(\lambda w \sqrt{\sigma_t} + \log \sqrt{\sigma_t} \big).
		\label{equ:MVT5}
	\end{eqnarray}
	\label{lem:MVT}
\end{lem}	
\begin{proof}
	This follows from \peter{a Taylor expansion.} % the mean value theorem. 
\end{proof}

\bigskip

%%%%%%%%%%%%%%%%%%%%%%%%%%%%%%%%%%%%%%%%%%	
%\newpage
%%%%%%%%%%%%%%%%%%%%%%%%%%%%%%%%%%%%%%%

\noindent
{\bf Acknowledgments:} This paper contains the main results of the third author's PhD thesis~\cite{Senkevich}. The first author is supported by EPSRC Fellowship EP/R022186/1, the second author by EPSRC grant EP/K016075/1, and the third is supported by a scholarship from the EPSRC Centre for Doctoral Training in Statistical Applied Mathematics at Bath (SAMBa), under the project EP/L015684/1.

\bigskip
%\bibliographystyle{acm}
%\bibliography{mybib}

\end{document}